\documentclass{amsart}
\usepackage[utf8]{inputenc}
\usepackage{amsmath, amssymb, bbm}
\usepackage{amsthm}
\usepackage{mathtools}
\usepackage{bbm}
\usepackage{blkarray}
\usepackage[matrix,arrow,curve]{xy}
\usepackage{tikz, tikz-cd}
\usepackage{graphicx}
\usepackage{color}
\usepackage{enumerate}
\definecolor{darkblue}{rgb}{0,0,0.6}
\usepackage[ocgcolorlinks,colorlinks=true, citecolor=darkblue, filecolor=darkblue, linkcolor=darkblue, urlcolor=darkblue]{hyperref}
\usepackage[capitalize]{cleveref}

\newcommand{\ignore}[1]{}
\renewcommand{\epsilon}{\varepsilon}
\renewcommand{\phi}{\varphi}
\newcommand{\bbZ}{\mathbb{Z}}
\newcommand{\bbL}{\mathbb{L}}
\newcommand{\Z}{\mathbb{Z}}

\newcommand{\R}{\mathbb{R}}
\newcommand{\CP}{\mathbb{CP}}

\newcommand{\bbN}{\mathbb{N}}
\newcommand{\toiso}{\xrightarrow{\cong}}
\newcommand{\wt}{\widetilde}
\newcommand{\ol}{\overline}
\newcommand{\Sq}{\mathrm{Sq}}
\newcommand{\red}{\mathrm{red}}
\newcommand{\Ext}{\mathrm{Ext}}
\newcommand{\Tor}{\mathrm{Tor}}
\newcommand{\sm}{\setminus}
\newcommand{\pt}{\mathrm{pt}}

\renewcommand{\Im}{\mathrm{Im}}
\DeclareMathOperator{\Aut}{Aut}

\DeclareMathOperator{\Out}{Out}
\DeclareMathOperator{\Hom}{Hom}
\DeclareMathOperator{\End}{End}
\DeclareMathOperator{\id}{Id}
\DeclareMathOperator{\im}{im}
\DeclareMathOperator{\pr}{pr}
\DeclareMathOperator{\inc}{inc}
\DeclareMathOperator{\coker}{coker}
\DeclareMathOperator{\Id}{Id}
\DeclareMathOperator{\Th}{Th}
\DeclareMathOperator{\cl}{cl}
\DeclareMathOperator*{\colim}{colim}
\DeclareMathOperator{\Arf}{Arf}
\DeclareMathOperator{\GL}{GL}
\DeclareMathOperator{\SL}{SL}
\DeclareMathOperator{\MSTOP}{MSTOP}

\numberwithin{equation}{section}
\newtheorem{thm}[equation]{Theorem}
\newtheorem{theorem}[equation]{Theorem}
\newtheorem{prop}[equation]{Proposition}

\newtheorem{cor}[equation]{Corollary}
\newtheorem{lemma}[equation]{Lemma}
\theoremstyle{definition}

\newtheorem{Conventions}[equation]{Conventions}

\newtheorem{example}[equation]{Example}
\newtheorem{defi}[equation]{Definition}
\newtheorem{definition}[equation]{Definition}
\newtheorem{rem}[equation]{Remark}
\newtheorem{remark}[equation]{Remark}
\newtheorem{nota}[equation]{Notation}

\newtheorem*{claim}{Claim}

\begin{document}

\title[$4$-manifolds with $3$-manifold fundamental groups]{Stable classification of $4$-manifolds with $3$-manifold fundamental groups}

\author[D.~Kasprowski]{Daniel Kasprowski}
\address{Max Planck Institut f\"{u}r Mathematik, Vivatsgasse 7, 53111 Bonn, Germany}
\email{kasprowski@mpim-bonn.mpg.de}

\author[M.~Land]{Markus Land}
\address{University of Regensburg, NWF I - Mathematik, 93049 Regensburg, Germany}
\email{markus.land@mathematik.uni-regensburg.de}

\author[M.~Powell]{Mark Powell}
\address{Universit\'e du Qu\'ebec \`a Montr\'eal, Montr\'eal, QC, Canada}
\email{mark@cirget.ca}

\author[P.~Teichner]{Peter Teichner}
\address{Max Planck Institut f\"{u}r Mathematik, Vivatsgasse 7, 53111 Bonn, Germany}
\email{teichner@mpim-bonn.mpg.de}

\def\subjclassname{\textup{2010} Mathematics Subject Classification}
\expandafter\let\csname subjclassname@1991\endcsname=\subjclassname
\expandafter\let\csname subjclassname@2000\endcsname=\subjclassname

\crefname{lemma}{Lemma}{Lemmas}
\crefname{section}{Section}{Sections}
\crefname{convention}{Convention}{Conventions}
\crefname{Conventions}{Conventions}{Conventions}
\crefname{definition}{Definition}{Definitions}
\crefname{defi}{Definition}{Definitions}
\crefname{prop}{Proposition}{Propositions}

\maketitle

\begin{abstract}
We study closed, oriented $4$-manifolds whose fundamental group is that of a closed, oriented, aspherical 3-manifold.  We show that two such $4$-manifolds are stably diffeomorphic if and only if they have the same $w_2$-type and their equivariant intersection forms are stably isometric.  We also find explicit algebraic invariants that determine the stable classification for spin manifolds in this class.
\end{abstract}

\section{Introduction}

Two smooth 4-manifolds $M,N$ are called {\em stably diffeomorphic} if there exist integers $m,n \in \Z$ such that stabilising $M,N$ with copies of $S^2 \times S^2$ yields diffeomorphic manifolds:
\[M \#m (S^2 \times S^2) \cong N \#n (S^2 \times S^2).
\]

In this paper we study the stable diffeomorphism classification of closed, oriented 4-manifolds whose fundamental group is that of some closed, oriented, aspherical $3$-manifold.  We give explicit, algebraically defined invariants of $4$-manifolds that detect the place of a $4$-manifold in the classification, working with orientation preserving diffeomorphisms.

We will also indicate the results for topological manifolds up to stable homeomorphism.

Special cases of the stable diffeomorphism classification have been investigated in A.~Cavicchioli, F.~Hegenbarth and D.~Repov\v{s} \cite{CAR-95}, F.~Spaggiari \cite{Spaggiari-03} and J.~Davis \cite{Davis-05}. The stable classification for manifolds with finite fundamental group was intensively studied by I.~Hambleton and M.~Kreck in \cite{Ham-kreck-finite}, as well as in the PhD thesis of the fourth author~\cite{teichnerthesis}.  The case of geometrically 2-dimensional groups was solved by Hambleton, Kreck and the last author in~\cite{HKT}.

\subsection*{Conventions}
All manifolds are assumed to be smooth, closed, connected and oriented if not otherwise stated. All diffeomorphisms are orientation preserving.

We call a group $\pi$ a \emph{COAT group} if it is the fundamental group of some Closed Oriented Aspherical Three-manifold.  Note that an irreducible $3$-manifold with infinite fundamental group is aspherical by the Sphere theorem and the Hurewicz theorem.  For a space $X$ with a homomorphism $\pi_1(X) \to \pi$, usually an isomorphism, we will denote any continuous map that induces the homomorphism by $c \colon X \to B\pi$.  \\

\subsection{Stable classification of 4-manifolds with COAT fundamental group}

The \emph{normal $1$-type} of a manifold $M$ is a $2$-coconnected fibration  $\xi \colon B \to BSO$
which admits a $2$-connected lift $\wt{\nu}_M \colon M \to B$, a \emph{normal $1$-smoothing}, of the stable normal bundle $\nu_M \colon M \to BSO$. See \cref{sec:basics} for more details.
The following fundamental result is a straightforward consequence of M.~Kreck's modified surgery theory, in particular \cite[Theorem C]{surgeryandduality}. See  \cite[p.~4]{teichnerthesis} for the discussion of the action of $\Aut(\xi)$. The last author also observed a direct proof in terms of handle structures~\cite[End~of~Section~4]{surgeryandduality}.

\smallskip
\begin{theorem}\label{cor:stablediffeoclasses}
The stable diffeomorphism classes of $4$-manifolds with normal $1$-type $\xi$ are in one-to-one correspondence with $\Omega_4(\xi)/\Aut(\xi)$.
\end{theorem}
\smallskip

The stable diffeomorphism classification programme therefore begins by determining the possible normal $1$-types~$\xi$. Since normal $1$-smoothings are $2$-connected,~$\xi$ determines the fundamental group. For a fixed fundamental group~$\pi$, the normal $1$-types are represented by elements
\[
w \in H^2(B\pi;\Z/2) \cup \{\infty\}.
\]
An oriented manifold $M$ is said to be \emph{totally non-spin} if $w_2(\wt{M}) \neq 0$; in this case we set $w=\infty$. Otherwise, there is a unique element $w \in H^2(B\pi;\Z/2)$ such that $c^*(w)=w_2(M)$ (see \cref{lem:w}). A $4$-manifold~$M$ is \emph{spin} if $w=0$ and we call $M$ \emph{almost spin} if $w\notin \{0,\infty\}$.  We remark that some authors use the terminology almost spin for the case $w \neq \infty$, but since the behaviour of the stable diffeomorphism classification differs when $w=0$, we chose nomenclature that differentiates this case.

We will use the fact (see for example~\cite{teichnerthesis}), that isomorphism classes of such pairs $(\pi,w)$ are in one-to-one correspondence with the fibre homotopy types of normal 1-smoothings.

The totally non-spin case $w=\infty$ corresponds to $\xi = \pr_2\colon B=B\pi \times BSO \to BSO$, where $\pr_2$ denotes the projection onto the second factor.  The spin case $w=0$ corresponds to $\xi \colon B=B\pi \times BSpin \to BSpin \to BSO$, the projection followed by the canonical map $BSpin \to BSO$.
The almost spin cases are twisted versions of the latter.

The main work in the stable classification comprises the computation of the bordism group~$\Omega_4(\xi)$ with the action of the automorphism group $\Aut(\xi)$, for each normal $1$-type~$\xi$. Finally, one tries to determine stable diffeomorphism invariants that detect all the possibilities.
The signature is an example of such an invariant; in the case of totally non-spin $4$-manifolds with fundamental group $\pi$ having \mbox{$H_4(B\pi)=0$,} such as COAT groups, the signature is a complete invariant.

The next theorem results from our successful application of all the above steps for COAT fundamental groups. The resulting invariants are explained in detail in \cref{sec:stablediffeoclasses}. They take values in a finite set, except for the signature, which in all three cases determines the image of a $4$-manifold in the various subgroups of $\Z$.

\smallskip
\begin{theorem}\label{thm:A}
For a COAT group $\pi$ and $w \in H^2(B\pi;\Z/2) \cup \{\infty\}$, stable diffeomorphism classes of closed oriented 4-manifolds with normal $1$-type isomorphic to $(\pi,w)$ are in bijection with the following sets:

\smallskip \begin{enumerate}[(1)]
\item\label{item:A1} the set of integers $\bbZ$ in the totally non-spin case $w = \infty$;
\item\label{item:A2} the set $16\cdot\Z \times \left(H_2(B\pi;\Z/2)/\Out(\pi) \cup \{ odd \}\right)$ in the spin case $w=0$; and
\item\label{item:A3} for almost spin, that is $w\notin\{0,\infty\}$, the set
\[ \big\{(n,\phi) \in 8\cdot \Z \times (H_2(B\pi;\Z/2)/\Out(\pi)_w) \,\,\big|\,\,  n/8 \equiv \langle w, \phi \rangle \mod{2} \big \}, \]
where $\langle w, - \rangle$ denotes the evaluation on $H_2(B\pi;\Z/2)$, and $\Out(\pi)_w$ denotes the set of outer automorphisms of $\pi$ whose induced action on $H^2(B\pi;\Z/2)$ fixes~$w$.
\end{enumerate}
\end{theorem}
\smallskip

Here the set $\{ odd\}$ consists of a single element. The nomenclature arises from some clairvoyance: as described in detail in the next section, for fixed signature, this element is represented by a (unique stable diffeomorphism class of a) spin 4-manifold~$M$ with odd {\em equivariant} intersection form~$\lambda_M$ on~$\pi_2(M)$. This notion differs significantly from saying that the ordinary intersection form on $H_2(M;\Z)$ is odd, which would mean that~$M$ is totally non-spin.

In the topological case the result is almost identical.  The deduction of \cref{thm:top-classification} from \cref{thm:A} can be found in \cref{stable-homeo-classification}.

\smallskip
\begin{theorem}\label{thm:top-classification}
For a COAT group $\pi$ and $w \in H^2(B\pi;\Z/2) \cup \{\infty\}$, stable homeomorphism classes of closed oriented 4-manifolds with normal $1$-type isomorphic to $(\pi,w)$ are in bijection with the following sets:
\smallskip \begin{enumerate}
\item  the set $\Z\times \Z/2$ in the totally non-spin case, where $\Z/2$ is the Kirby-Siebenmann invariant in $\Z/2$ because of the existence of the sister projective space $*\CP^2$.
\item the set $8\cdot\Z \times \left(H_2(B\pi;\Z/2)/\Out(\pi) \cup \{ odd \}\right)$ in the spin case $w=0$, because the $E_8$ manifold is a topological spin manifold with signature~$8$. The Kirby-Siebenmann invariant is the signature divided by 8.
\item In the almost spin case $w\notin\{0,\infty\}$, we have the set
\[
8\cdot\Z \times H_2(B\pi;\Z/2)/\Out(\pi)_w,
\]
on which the Kirby-Siebenmann invariant is given by the signature divided by 8 plus evaluation of $w$ on the element of $H_2(B\pi;\Z/2)$.
\end{enumerate}
\end{theorem}
\smallskip

In each normal $1$-type, the smooth classification occurs as the kernel of the Kirby-Siebenmann invariant.

\subsection{Explicit invariants for spin 4-manifolds}

Next, in the case of spin 4-manifolds with COAT fundamental group~$\pi$, we describe a complete set of invariants that are defined independently of a normal 1-smoothing.  The first invariant besides the signature is the parity of the equivariant intersection form
\[
\lambda_M \colon \pi_2(M) \times \pi_2(M) \to \Z\pi.
\]
This is a sesquilinear, hermitian form.
An intersection form $\lambda_M \colon \pi_2(M) \to \pi_2(M)^*$ is called {\em even} if there exists a $\Z\pi$-linear map $q:\pi_2(M) \to \pi_2(M)^*$ such that $\lambda_M = q + q^*$.  If no such $q$ exists then we say that $\lambda_M$ is \emph{odd}.  We refer to $\lambda$ being even or odd as its {\em parity}.

It turns out that the $\Z\pi$-module $\pi_2(M)$ is stably isomorphic to $I\pi\oplus \Z\pi^k$ for some $k \in \mathbb{N}$, where $I\pi$ denotes the augmentation ideal of $\Z\pi$. Since $\lambda_M$ admits a quadratic refinement (see \cref{defn:quadratic-refinement}), in order to determine the parity, it suffices to restrict to the (non-free) $I\pi$ summand; this assertion is explained in \cref{lem:quadratic=even-in-free}. If $\lambda_M$ is odd then the signature determines the stable diffeomorphism type of such spin manifolds.

If the  intersection form $\lambda_M$ is even, we can arrange its restriction to $I\pi$ to vanish, after some stabilisation and a change of basis. In this case, the self-intersection number $\mu_M$ also vanishes on $I\pi$ and one can compute the {\em first order intersection number} $\tau_1$ on $I\pi$, which takes values in a quotient of $\Z[\pi \times\pi]$; compare~\cite{schneiderman-teichner-tau}.

In this paper we only need a $\Z/2$-valued version of $\tau_1$, defined on spherically characteristic classes in $\pi_2(M)$. This invariant first appeared in R.~Kirby and M.~Freedman~\cite[p.~93]{Kirby-Freedman} and Y.~Matsumoto~\cite{Matsumoto}, and a similar invariant was later used by M. Freedman and F. Quinn \cite[Definition~10.8]{Freedman-Quinn}. A detailed definition is given in \cref{sec:tau}, but here is a rough outline.

Let $[S]\in\pi_2(M)$ be a spherically characteristic element (see \cref{defn:spherically-characteristic}) with $\mu_M(S)=0$.  Pair up the self-intersection points of the immersed 2-sphere $S$ with framed Whitney discs, and count the intersection points of the Whitney discs with $S$, modulo~$2$.   This defines~$\tau(S)\in\Z/2$.
We show that this descends to an invariant
\[
\tau_M\in\Hom(H^2(B\pi;\Z/2),\Z/2)\cong H_2(B\pi;\Z/2),
\]
as discussed in \cref{lem:tauinv}.  Then we obtain the following theorem, giving the promised classification for spin $4$-manifolds in terms of explicit invariants.

\smallskip
\begin{theorem}\label{thm:B}
Closed spin $4$-manifolds $M$ and $M'$ with fundamental groups isomorphic to a COAT group $\pi$ are stably homeomorphic if and only if
\smallskip \begin{enumerate}[(1)]
\item\label{item:B1} their signatures agree: $\sigma(M) = \sigma(M')$,
\item\label{item:B2} their equivariant intersection forms $\lambda_M$ and $\lambda_{M'}$ have the same parity, and
\item\label{item:B3} for even parity, their first order intersection invariants agree:
\[
[\tau_M]=[\tau_{M'}]\in H_2(B\pi;\Z/2)/\Out(\pi).
\]
\end{enumerate}
The smooth result is exactly the same.  In the smooth case, the signature lies in $16\cdot\Z$, whereas in the topological case the signature is divisible by $8$.
\end{theorem}
\smallskip

Note that the intersection and self-intersection numbers $\lambda_M, \mu_M$ are considered as lying at order zero, whereas $\tau_M$ is of order one. There are invariants of all orders (with large indeterminacies in their target groups), defined in an inductive way, as described by R.~Schneiderman and the fourth author~\cite[Definition~9]{schneiderman-teichner-ki}.
The idea is as follows: if some algebraic count of intersections vanishes, pair these intersections up by Whitney discs, and count how these new Whitney discs intersect the previous surfaces. The order of the invariant is the number of layers of Whitney discs present.

However only order one intersections are relevant in the stable setting.  It follows directly from~\cite[Theorem~2]{schneiderman-teichner-tau} that an element $[S]\in\pi_2(M)$ is represented by an embedding $S \colon S^2 \hookrightarrow M\#k(S^2\times S^2)$ for some $k$ if and only if $\mu(S)=0$ and $\tau_1(S)=0$. Here $\tau_1$ is the {\em full} first order intersection invariant of~\cite{schneiderman-teichner-tau}.

\subsection{The stable $HAN_1$-type}

We consider the following data for a closed oriented 4-manifold $M$ which we shall refer to as its Hermitian Augmented Normal $1$-type:
\[
HAN_1(M)=(\pi_1(M),w_M, \pi_2(M),\lambda_M).
\]
Here $\pi_2(M)$ is considered as a $\Z[\pi_1(M)]$-module, $\lambda_M$ is the equivariant intersection form on $\pi_2(M)$ and $w_M \in H^2(B\pi_1(M);\Z/2)\cup \{\infty\}$ gives the normal 1-type $(\pi_1(M),w_M)$ of $M$.

Connected sum with copies of $S^2 \times S^2$ leaves the normal 1-type unchanged and induces stabilisation of $(\pi_2(M),\lambda_M)$ by hyperbolic forms. There is a notion of stable isomorphism, denoted $\cong^{s}$, given by a pair of maps between fundamental groups and their modules, preserving $w$ and $\lambda$.  See \cref{section-tau-versus-int-form} for the details.

Note that in previous discussions of similar {\em quadratic} $2$-types, one needed to add an invariant $k_M\in H^3(B\pi_1(M);\pi_2(M))$, the $k$-invariant classifying the second Postnikov section of $M$. Our result implies indirectly that this $k$-invariant is determined stably by the other invariants.

\smallskip
\begin{theorem}\label{thm:diffeo-iff-int-forms-intro}
For closed oriented $4$-manifolds $M$ and $M'$ with COAT fundamental group, any stable isomorphism $HAN_1(M) \cong^s HAN_1(M')$ is realised by a stable diffeomorphism.
\end{theorem}
\smallskip

The spin case of this theorem says that the $H_2(B\pi;\Z/2)$ part of the spin classification above, which we identified with a $\tau_M$-invariant, can also be seen from the equivariant intersection form.  This is extremely surprising and reveals a new feature of stable classification that has not been previously observed.

As a corollary of \cref{thm:diffeo-iff-int-forms-intro}, we recover the following special case of J.~Davis' theorem~\cite{Davis-05} that homotopy equivalent 4-manifolds with torsion-free fundamental group satisfying the strong Farrell-Jones conjecture are stably diffeomorphic.

\smallskip
\begin{cor}
Let $M,M'$ be closed oriented 4-manifolds with COAT fundamental groups that are homotopy equivalent.  Then $M$ and $M'$ are stably diffeomorphic.
\end{cor}
\smallskip

In particular, it follows that the $\tau_M$-invariant as in \cref{thm:B} is a homotopy invariant.
We remark that this contrasts with the case of finite groups, where there are homotopy equivalent (almost spin) 4-manifolds that are not stably diffeomorphic, as detected by the $H_2(B\pi;\Z/2)$ part of the bordism group~\cite[Example 5.2.4]{teichnerthesis}.

In \cref{theorem:realisation-of-forms} we shall describe precisely which stable isomorphism classes of forms are realised as the intersection forms of $4$-manifolds with a COAT fundamental group.

\subsection{Concluding remarks}

The case of COAT groups is particularly attractive because models for all stable spin diffeomorphism classes can be constructed, starting from the given $3$-manifold. All invariants can be computed explicitly, leading to simple algebraic results.

The main new aspect of our classification, not discussed in any previously known examples, is how the order one intersection invariant $\tau_M$ enters into the picture, determining most of the finite part of the classification.

Even though \cref{thm:diffeo-iff-int-forms-intro} tells us that~$\tau_M$ is determined by the intersection form~$\lambda_M$, this can turn out to be a red herring.
For example, one does not need to know the equivariant intersection form on the entire $\Z[\pi_1(M)]$-module $\pi_2(M)$, which can be huge, in order to decide whether two $4$-manifolds are stably diffeomorphic. Instead, the $\tau_M$-invariants can be computed on the much smaller vector space $H^2(B\pi;\Z/2)$.

\subsection*{Organisation of the paper}

In \cref{sec:basics} we recall the basic definitions needed for the theory.
\cref{sec:stablediffeoclasses} computes the bordism groups and the action of the automorphisms of the normal $1$-types on the bordism groups, from which the proof of \cref{thm:A} is derived.
\cref{stable-homeo-classification} briefly describes how to adapt the computations of \cref{sec:stablediffeoclasses} to the stable homeomorphism classification of topological manifolds with COAT fundamental group.
\cref{sec:almost-spin-classificiation} uses the topological bordism groups to complete the computation of the stable classification of almost spin $4$-manifolds, first computing in the topological category, then deducing the smooth result from the topological result.
In \cref{sec:some-examples} we present some examples, computing the set of stable diffeomorphism classes of spin manifolds whose fundamental group $\pi$ is a central extension of $\Z^2$ by $\Z$.
\cref{sec:exmanifolds,sec:tau} give the proof of \cref{thm:B}.  In \cref{sec:exmanifolds} we discuss the parity of the equivariant intersection form of a spin manifold and show that it detects a $\Z/2$ invariant in the bordism group corresponding to the element $\{odd\}$ in \cref{thm:A}(\ref{item:A2}).
In \cref{sec:tau} we introduce the $\tau$-invariant of a spin manifold and show that it detects the invariants in the bordism group that come from $H_2(\pi;\Z/2)$.  We show that these same invariants from the bordism group can also be seen in the equivariant intersection form in \cref{section-tau-versus-int-form}.

\subsection*{Acknowledgements}
We are indebted to Diarmuid Crowley, Ian Hambleton and the referee for many extremely helpful comments and suggestions.  In particular Diarmuid pointed out some errors, in a previous version, in our computations of the action of the automorphisms of the normal $1$-types.
The first, third and fourth authors were supported by the Max Planck Institute for Mathematics.  The second author was supported by the GRK 1150 ``Cohomology and Homotopy'' and Wolfgang L\"uck.  The third author was also supported by an NSERC Discovery Grant.

\section{Normal $1$ type and the James spectral sequence}
\label{sec:basics}
\begin{nota}
A map has the same degree of connectedness as its homotopy cofibre and the same degree of coconnectedness as its homotopy fibre.  Concretely, a map of spaces is called $k$-connected if it induces an isomorphism of homotopy groups~$\pi_i$ for $i < k$ and a surjection on $\pi_k$. A map is called $k$-coconnected if it induces an isomorphism on $\pi_i$ for $i > k$ and is injective on $\pi_k$.
\end{nota}

\smallskip
\begin{definition}
Let $M$ be a closed manifold of dimension $n$. A \emph{normal $k$-type} for~$M$ is a fibration over $BSO$, denoted by $\xi \colon B \to BSO$ through which a map representing the stable normal bundle $\nu_M \colon M \to BSO$ factors as follows
\[\xymatrix@R=.5cm{ & B \ar[d]^\xi \\ M \ar@/^.7pc/[ur]^-{\wt{\nu}_M} \ar[r]_-{\nu_M} & BSO }\]
with $\wt{\nu}_M$ a $(k+1)$-connected map and $\xi$ a $(k+1)$-coconnected map. A choice of $\wt{\nu}_M$ is called a \emph{normal k-smoothing} of $M$.
\end{definition}
\smallskip

All the normal $k$-types of $M$ are fibre-homotopy equivalent to one another.  Frequently we only specify the normal $k$-type up to fibre-homotopy equivalence.  For example, $B= BSpin \times B\pi \to BSO$ is not a fixed space until one chooses models for the classifying spaces $B\pi$, $BSpin$ and $BSO$.
Note that the fibre-homotopy class of a normal $1$-type is an invariant of stable diffeomorphism since $S^2 \times S^2$ has trivial stable normal bundle.
\smallskip
\begin{rem}
If $M$ is a closed $n$-dimensional manifold, and $\xi\colon B \to BSO$ is its normal $k$-type, then the automorphisms $\Aut(\xi)$ of this fibration act transitively on the set of homotopy classes of normal $k$-smoothings of~$M$.
\end{rem}

\smallskip
\begin{thm}[{{\cite[Theorem C]{surgeryandduality}}}]
	\label{thm:stablediffeoclasses}
Two closed $2q$-dimensional manifolds with the same Euler
characteristic and the same normal $(q-1)$-type, admitting bordant normal $(q-1)$-smoothings, are diffeomorphic after connected sum with $r$ copies of $S^q\times S^q$ for some $r$.
\end{thm}
\smallskip

Note that two closed orientable connected 4-manifolds with the same fundamental groups have the same Betti numbers $\beta_i$ for $i \neq 2$.  A necessary condition for stable diffeomorphism of two 4-manifolds is that they have the same signature, since connect summing $S^2 \times S^2$ just adds a hyperbolic summand to the intersection form.  Two 4-manifolds with the same signature have second Betti numbers differing by a multiple of two.  It is then easy to see that the Euler characteristics can be made to coincide by stabilising one of the 4-manifolds.
\cref{cor:stablediffeoclasses} from the introduction follows from this observation,  \cref{thm:stablediffeoclasses}, and the fact that stably diffeomorphic $4$-manifolds are bordant over their normal $1$-types (see~\cite[Lemma~2.3(ii)]{CS11} for a proof).
Recall that \cref{cor:stablediffeoclasses} states that stable diffeomorphism classes of $4$-manifolds with normal $1$-type $\xi$ are in one-to-one correspondence with $\Omega_4(\xi)/\Aut(\xi)$.

Next we want to recall how to compute the bordism groups $\Omega_4(\xi)$.
For a vector bundle $E \colon Y \to BSO(n)$, let $\Th(E)$ be the Thom space given by the unit disc bundle modulo the unit sphere bundle.  Given a stable vector bundle $\eta\colon Y \to BSO$, let $M\eta$ be the Thom spectrum.  For the convenience of the reader  we recall the construction of the Thom spectrum.
Let $Y_n$ be given by the following pullback diagram
\[\xymatrix{Y_n \ar[r]^-{\eta_n} \ar[d] & BSO(n) \ar[d] \\ Y \ar[r]_-{\eta} & BSO}\]
Then the $n^{th}$ space in the spectrum $M\eta$ is given by
\[ (M\eta)_n = \Th(\eta_n) \]
It follows immediately from the definition that $\eta_{n+1}|_{Y_{n}} = \eta_{n}\oplus \underline{\R}$ and hence we obtain canonical structure maps
\[ (M\eta)_n\wedge S^1 = \Th(\eta_n) \wedge \Th(\underline{\R}) = \Th(\eta_n \oplus \underline{\R}) \to \Th(\eta_{n+1}) = (M\eta)_{n+1}.\]

 Normal $1$-types of a 4-manifold $N$ with COAT fundamental group $\pi$ are given by
\[ B = \begin{cases} BSO \times B\pi \xrightarrow{\pr_1} BSO &\text{ in the totally non-spin case,} \\ BSpin \times B\pi \xrightarrow{\gamma\circ \pr_1} BSO & \text{ in the spin case, and} \\ BSpin \times B\pi \xrightarrow{\gamma\times E} BSO & \text{ in the almost spin case} \end{cases}\]
where $E$ is a certain complex line bundle over $B\pi$ and $\gamma$ is the tautological stable vector bundle over $BSpin$, see \cref{lem:normal-totally-non-spin-case,lem:normal-spin-case,lemma:almostspin}.

Recall from above that a stable bundle over a space gives rise to a Thom spectrum, and in the case of the normal $1$-types as just described, this gives the Thom spectrum $M\xi$, where we obtain (from the construction of that spectrum) that
\[M\xi = \begin{cases} MSO \wedge B\pi_+ & \text{ in the totally non-spin case,} \\
        MSpin \wedge B\pi_+ & \text{ in the spin case, and } \\
        \Sigma^{-2}(MSpin \wedge \Th(E)) & \text{ in the almost spin case.}
        \end{cases}\]
Note that in the almost spin case there is a shift by two in the indexing of the spectrum corresponding to the dimension of the vector bundle~$E$.
The Pontrjagin-Thom construction $\Omega_4(\xi) \cong \pi_4(M\xi)$ yields isomorphisms:
\[\Omega_4(\xi) \cong \begin{cases} \pi_4(MSO\wedge B\pi_+) & \text{ in the totally non-spin case,} \\ \pi_4(MSpin\wedge B\pi_+) & \text{ in the spin case, and} \\ \pi_6(MSpin\wedge \Th(E)) &\text{ in the almost spin case.}\end{cases}\]

Recall that the homotopy groups of a spectrum $\mathbb{E}$ are defined by $\pi_n(\mathbb{E}) = \colim \pi_{n+k}(\mathbb{E}_k)$.  Note that $B\pi_+= B\pi \sqcup \{\ast\}$ is the Thom space of the canonical rank $0$ bundle over $B\pi$, and thus the spin case can be viewed as a special case of the almost spin case (in which the bundle $E$ may be chosen to be the rank $0$ bundle).

To compute the bordism group $\Omega_4(\xi)$, we apply the James spectral sequence \cite[Theorem 3.1.1.]{teichnerthesis} with homology theory being stable homotopy theory $\pi_*^s$, to the diagram
\[\xymatrix@R=.5cm{F \ar[r] & B \ar[r] \ar[d]_{\xi} & B\pi \\ & BSO & }\]
where $B$ is the normal $1$-type and $F$ is thus either $BSO$ or $BSpin$.

The $E^2$-page of the James spectral sequence reads as
\[ E^2_{p,q} = H_p(B\pi;\pi_q(M\xi|_F)) \Longrightarrow \pi_{p+q}(M\xi).\]
A priori this is to be interpreted with twisted coefficients. However it turns out that, since the fibration $F \to B \to B\pi$ is trivial (i.e. $B = F\times B\pi$), the spectral sequence is not twisted. Furthermore $M\xi|_F$ is either $MSO$ or $MSpin$.

The homotopy groups $\pi_*(M\xi)$ can also be computed by a standard Atiyah-Hirzebruch spectral sequence (for $MSpin$ or $MSO$). It turns out that the James spectral sequence is the same as the Atiyah-Hirzebruch spectral sequence in the first two cases above (totally non-spin and spin case). In the almost spin case the $E^2$-pages of the James spectral sequence and the Atiyah-Hirzebruch spectral sequence are isomorphic using the Thom isomorphism
\[ \wt{H}_{p+2}(\Th(E);A) \cong H_p(B\pi;A)\]
for all abelian groups $A$.

Denote the filtration on the abutment of an Atiyah-Hirzebruch spectral sequence by
\[ 0 \subset F_{0,n} \subset F_{1,n-1} \subset \dots \subset F_{n-q,q} \subset \dots \subset F_{n,0} = \Omega_n(\xi).\]
Recall that $F_{n-q,q}/F_{n-q-1,q+1} \cong E^{\infty}_{n-q,q}$.

Denote the restriction of the fibration $f \colon B \to B\pi$ to the $p$-skeleton of $B\pi$ by $B|_p$, and let $\xi|_p \colon B|_p \to BSO$ be the restriction of $\xi$ to $B|_p$. An element of $\Omega_n(\xi)$ lies in $F_{p,n-p}$ if and only if it is in the image of the map $\Omega_n(\xi|_p) \to \Omega_n(\xi)$.
This follows from the naturality of the spectral sequence applied to the map of fibrations induced by the inclusion of $B\pi^{(p)} \to B\pi$.


The following key lemma allows us to interpret the $E^2$ page in terms of tranverse inverse images.

Let $X$ be a CW-complex and let $X^{(p)}$ be its $p$-skeleton. Let $E\colon X\to BSO$ be a stable vector bundle and let $\xi=(E,\gamma)\colon X\times BSpin\to BSO$.  For a subset $Y$ of $X$, let $\xi|_Y$ denote the restriction of $\xi$ to $Y\times BSpin$. Denote the barycentres of the $p$-cells $\{e_i^p\}$ of $X$ by $\{b_i^p\}_{i \in I}$.
Given an element $[f\colon M \to X^{(p)}\times BSpin] \in \Omega_n^{Spin}(\xi|_{X^{(p)}})$, denote the regular preimages of the barycentre $\{b_i^p\} \in X^{(p)}$ under $\pr_1\circ f$ by $N_i \subset M$. Note that $[N_i] \in \Omega_{n-p}^{Spin}$, since the normal bundle of $N_i$ in $M$ is trivial and so is $(\pr_2\circ f)^*E$ restricted to $N_i$, and hence $N_i$ inherits a spin structure from $f\colon M\to X^{(p)}\times BSpin$.
For spin $4$-manifolds, we will use the case of the following theorem when $E$ is the trivial bundle.  For almost spin manifolds, $E$ will be a non-trivial bundle depending on the second Stiefel-Whitney class.

\smallskip
\begin{lemma}\label{lem:arf}
The canonical map $\Omega_n(\xi|_{X^{(p)}}) \to H_p(X^{(p)};\Omega_{n-p}^{Spin})$ that comes from the spectral sequence coincides with the map
\[\begin{array}{rcl} \Omega_n(\xi|_{X^{(p)}}) & \to &  H_p(X^{(p)};\Omega_{n-p}^{Spin}) \\
\lbrack M \to X^{(p)}\rbrack  &\mapsto&  \Big[\sum\limits_{i \in I} [N_i]\cdot e_i^p \Big]. \end{array}\]
\end{lemma}

The map $\Omega_n(\xi|_{X^{(p)}}) \to H_p(X^{(p)};\Omega_{n-p}^{Spin})$, which is sometimes called an edge homomorphism, arises as follows.  The abutment of the James spectral sequence $\Omega_n(\xi|_{X^{(p)}}) = F_{n,0}$ maps to its quotient by the first filtration step $F_{p,n-p}$ that differs from $F_{n,0}$.  This term is indeed $F_{p,n-p}$, since the homology of $X^{(p)}$ vanishes in degrees greater than $p$, therefore $E^2_{s,t} = E^{\infty}_{s,t} =0$ for all $s>p$ and all $r \geq 3$.   We have $F_{n,0}/F_{p,n-p} \cong E^{\infty}_{p,n-p}$.  The target, $H_p(X^{(p)};\Omega_{n-p}^{Spin})$ in the left-most column, is the $E^2_{p,n-p}$ term of the spectral sequence.  Since no differentials have image in $E^2_{p,n-p}$, we have that $E^{\infty}_{p,n-p} \subseteq E^{2}_{p,n-p} = H_p(X^{(p)};\Omega_{n-p}^{Spin})$, and so the composition
\[\Omega_n(\xi|_{X^{(p)}}) = F_{n,0} \to F_{n,0}/F_{p,n-p} \xrightarrow{\simeq} E^{\infty}_{p,n-p} \to E^{2}_{p,n-p} = H_p(X^{(p)};\Omega_{n-p}^{Spin})\]
gives the desired map.

\begin{proof}[Proof of \cref{lem:arf}]
The case where $p=0$ is trivial and therefore we can assume $p\geq 1$ and consider reduced homology instead. The case that $n<p$ is also trivial, so we assume for the rest of the proof that $n \geq p$.

Consider the following diagram, which is induced by the maps of pairs
\[ (X^{(p)},\emptyset) \to (X^{(p)},X^{(p)}\setminus \mathring{D}^p_i) \leftarrow (D_i^p,\partial D_i^{p}), \]
where the first map picks out one single $p$-cell $D^p_i$ and note that  an element in $H_p(X^{(p)};\Omega_{n-p}^{Spin})$ is determined by its image in $\Omega_{n-p}^{Spin}$, ranging over all $p$-cells.
\[\xymatrix@C-1pc{\Omega_n(\xi|_{X^{(p)}}) \ar[r] \ar[d] & \Omega_n(\xi|_{X^{(p)}},\xi|_{X^{(p)}\setminus \mathring{D}_i^p}) \ar[d] & \Omega_n(\xi|_{D_i^p},\xi|_{\partial D_i^p}) \ar[l]_-\cong\ar[d]^\cong\ar[r]^-\cong&\wt\Omega_n^{Spin}(S^p) \ar[d]^\cong\\
H_p(X^{(p)};\Omega_{n-p}^{Spin}) \ar[r] & H_p(X^{(p)},X^{(p)}\setminus \mathring{D}_i^p;\Omega_{n-p}^{Spin})  & H_p(D_i^p,\partial D_i^p;\Omega_{n-p}^{Spin})\ar[l]_-\cong \ar[r]^-\cong&\wt H_p(S^p;\Omega_{n-p}^{Spin})\ar[d]^\cong \\
& & & \Omega_{n-p}^{Spin}}\]
An element of the relative bordism group $\Omega_n(\xi|_{X^{(p)}},\xi|_{X^{(p)}\setminus \mathring{D}_i^p})$ is represented by an $n$-dimensional manifold with boundary $M$, together with a diagram
\[\xymatrix @R-0.3cm @C-0.3cm{\partial M \ar[r] \ar[d] & (X^{(p)}\setminus \mathring{D}_i^p) \times BSpin \ar[d]\\
M \ar[r] \ar[dr]_{\nu_M} & X^{(p)} \times BSpin \ar[d]^{\xi|_{X^{(p)}}} \\ & BSO.}\]
The addition is, as ever, disjoint union, and we quotient by bordisms respecting the bundle structure.
This is just the unreduced homology theory arising from the reduced theory corresponding to the spectrum $MSpin \wedge \Th(E)$ discussed above.

Now we explain the maps in the diagram above.
The first and second horizontal maps are from the long exact sequences of the appropriate pairs.  The second horizontal maps are isomorphisms by excision.  The third isomorphism follows since $\partial D^p \to D^p$ is a cofibration, so homology of a pair is isomorphic to the corresponding reduced homology of the quotient $S^p$.  The vertical maps (not including the bottom-right vertical map) are edge homomorphisms that arise in the James spectral sequence, analogous to the map in the statement of the theorem (which is the left-most vertical map).   The diagram commutes by naturality of the James spectral sequence.

By commutativity, it suffices to check that the right-then-down composition is determined by inverse images as described above the statement of the lemma. Note that the right vertical composite is the (inverse of the) suspension isomorphism in spin bordism (because suspension isomorphisms are natural in the homology theory, and the bottom right vertical map is, by definition, the suspension isomorphism in singular homology, after identifying $\Omega_{n-p}^{Spin} \cong \wt{H}_0(S^0;\Omega_{n-p}^{Spin})$). But suspension isomorphisms in bordism theories are are given by transverse inverse images.
This follows from the description of the suspension isomorphism $\wt{\Omega}_n(S^q) \cong \wt{\Omega}_{n-1}(S^{q-1})$ as the boundary map of the Mayer-Vietoris sequence in reduced bordism theory associated to the decomposition $S^q= \Sigma S^{q-1} = D^q \cup_{S^{q-1}} D^q$.  A proof that this boundary map can be described in terms of inverse images may be found in~\cite[Section~II.3]{brocker-tom-dieck}.
This completes the proof of the lemma.
\end{proof}

\section{Stable diffeomorphism classification from bordism groups}\label{sec:stablediffeoclasses}

Throughout this section, all results will hold for $\pi$ a COAT group.  With potential future use in mind, for many lemmas we will try to give the most general hypotheses under which the given proof holds.
All manifolds called either $M$ or $X$ will be smooth, oriented and have fundamental group $\pi$.

Recall that a manifold $M$ is called \emph{totally non-spin} if its universal cover $\widetilde{M}$ is not spin, and $M$ is called \emph{almost spin} if $M$ is not spin but its universal cover is.  The normal 1-type of $M$ is determined by $w_2(M)$ and $w_2(\wt{M})$.  We investigate the totally non-spin case in \cref{sec:totallynonspin}, the spin case in \cref{sec:spincase} (and also in \cref{sec:exmanifolds,sec:tau}), and the almost spin case in \cref{sec:almostspin}.  In each case we compute the bordism group of the relevant vector bundle $\xi$, and the action of $\Aut(\xi)$ on the bordism group.

\subsection{Totally non-spin 4-manifolds}
\label{sec:totallynonspin}

\smallskip \begin{lemma}\label{lem:normal-totally-non-spin-case}
Let $\pi$ be a finitely presented group.  The normal $1$-type of a totally non-spin manifold with fundamental group $\pi$ is given by
\[\xymatrix{\xi \colon B\pi \times BSO \ar[r]^-{pr_2} & BSO,}\]
where the map is given by the projection onto $BSO$.
\end{lemma}

\begin{proof}
Since $M$ has fundamental group $\pi$ there is a canonical map $c\colon M \to B\pi$ classifying the universal cover of $M$. The orientation of $M$ gives a factorisation
\[\xymatrix{M \ar[r]^-{c\times \nu_M} & B\pi \times BSO \ar[r]^-{pr_2} & BSO}.\]
The map $B\pi \times BSO \to BSO$ is $2$-coconnected since $B\pi$ has no higher homotopy groups. Moreover the map $M \to B\pi \times BSO$ induces an isomorphism on fundamental groups. It remains for us to verify that the map $M \to BSO$ induces a surjection on~$\pi_2$. For this we note that $w_2\colon BSO \to K(\Z/2,2)$ induces an isomorphism on~$\pi_2$, so it suffices to see that $\pi_2(M) \to \pi_2(K(\Z/2,2))$ is surjective.

The composition $\wt{M} \to BSO \to K(\Z/2,2)$ determines a cohomology class equal to $w_2(\wt{M})$ in $H^2(\wt{M};\Z/2) \cong \Hom(H_2(\wt{M};\Z),\Z/2)$. The isomorphism here is given by the universal coefficient theorem, and uses that $H_1(\wt{M};\Z) =0$.  Consider the following diagram.
\[\xymatrix{\pi_2(M) \ar[r]^-{\cong}  & \pi_2(\wt{M}) \ar[r] \ar[d]^-{\cong} & \pi_2(K(\Z/2,2)) \ar[r]^-{\cong} \ar[d]^-{\cong} & \Z/2 \\
& H_2(\wt{M};\Z) \ar[r] & H_2(K(\Z/2,2);\Z) & }\]
The right-up-right composition starting at $H_2(\wt{M};\Z)$ is $w_2(\wt{M})$, according to the identification
\[[\wt{M},K(\Z/2,2)] \cong H^2(\wt{M};\Z/2)\]
and the universal coefficient theorem.  The vertical maps are isomorphisms by the Hurewicz theorem and the diagram commutes because the Hurewicz homomorphism is a natural transformation.  Since $\wt{M}$ is not spin, $w_2(\wt{M})\neq 0$, from which it follows that $\pi_2(M) \to \pi_2(K(\Z/2,2))$ is surjective.
\end{proof}

\smallskip
\begin{lemma}
Let $\pi$ be a finitely presented group and let $\xi \colon B\pi \times BSO \xrightarrow{pr_2} BSO$.  Then the automorphisms of $\xi$ are given by $\Aut(\xi) \cong \Out(\pi)$.
\end{lemma}
\smallskip

\begin{proof}
An automorphism of $\xi$ is given by a map $B\pi\times BSO\to B\pi$.
Since $BSO$ is simply connected, we have
\[[B\pi\times BSO,B\pi]\cong [B\pi,B\pi].\]
Restrict to the homotopy equivalences $B\pi \to B\pi$, the (unbased) homotopy classes of which are in one to one correspondence with the outer automorphisms $\Out(\pi)$ of $\pi$.
This is because inner automorphisms correspond to base point changes and an element of $[B\pi, B\pi]$ is independent of base points.
\end{proof}

\smallskip
\begin{thm}
\label{thm:3.3}
Let $\pi$ be a COAT group.  For $\xi\colon B\pi \times BSO \to BSO$ as above we have
\[ \Omega_4(\xi) \cong \Z \]
detected by the signature.  Moreover the action of $\Out(\pi)$ on $\Omega_4(\xi)$ is trivial.
\end{thm}

\begin{proof}
Since the oriented bordism groups $\Omega_q^{SO}$ are trivial for $q=1,2,3$, this follows from the James spectral sequence for the fibration $BSO \to B\pi \times BSO \to B\pi$, so that $\Omega_q(\xi|_F) = \Omega_q^{SO}$ is oriented bordism. The result is true for all groups $\pi$ with $H_4(B\pi;\Z)=0$, in particular for aspherical $3$-manifold groups.  The assertion that the action of $\Out(\pi)$ is trivial is straightforward.
\end{proof}

\noindent From this we obtain the following corollary, which is \cref{thm:A}~(\ref{item:A1}).

\smallskip
\begin{cor}\label{cor:stable-class-non-spin}
Two oriented, totally non-spin $4$-manifolds with COAT fundamental group $\pi$ are stably diffeomorphic if and only if their signatures are equal.
\end{cor}
\smallskip

Thus the signature of the ordinary intersection form is a complete invariant for totally non-spin 4-manifolds. Note that we do not need to look at equivariant intersection forms in this case.

\subsection{Spin $4$-manifolds}
\label{sec:spincase}

\smallskip \begin{lemma}\label{lem:normal-spin-case}
Let $\pi$ be a finitely presented group. A normal $1$-type of a spin manifold $M$ with fundamental group $\pi$ is given by
\[\xymatrix{B\pi\times BSpin\ar[r]^-{\gamma \circ pr_2}&BSO},\]
where $pr_2$ is the projection onto $BSpin$ and $\gamma$ is the canonical map $BSpin\to BSO$.
\end{lemma}

\begin{proof}
The map $\gamma \circ \pr_2$ is $2$-coconnected since $B\pi$ has trivial higher homotopy groups $\pi_i(B\pi) =0$ for $i \geq 2$ and $BSpin \to BSO$ is 2-coconnected.

Since $M$ has fundamental group $\pi$ there is a canonical map $c\colon M \to B\pi$ classifying the universal cover of $M$. Let $\wt\nu_M$ be the lift of $\nu_M\colon M\to BSO$ to $BSpin$ given by the spin structure on $M$. Then a normal $1$-smoothing of $M$ is given by
\[\xymatrix{M\ar[r]^-{c\times \wt\nu_M}&B\pi\times BSpin}.\]
By definition of $c$ the map $c\times \wt\nu_M$ is an isomorphism on $\pi_1$ and since we have $\pi_2(B\pi\times BSpin)=0$, it is therefore $2$-connected.
\end{proof}

\smallskip \begin{lemma}
Let $\pi$ be a finitely presented group and let $\xi\colon B\to BSO$ be $\gamma\circ pr_2\colon B\pi\times BSpin\to BSO$. Then
\[\Aut(\xi)\cong H^1(B\pi;\bbZ/2)\rtimes \Out(\pi),\]
where the action of $\Out(\pi)$ on $H^1(B\pi;\Z/2)$ in the definition of the multiplication in the semi-direct product is the canonical one, obtained as follows.   An element of $\Out(\pi)$ determines a homotopy class of maps $B\pi \to B\pi$.  An element of $H^1(B\pi,\Z/2)$ determines a homotopy class of maps $B\pi \to K(\Z/2,1)$.  Then $[B\pi,B\pi]$ acts on $[B\pi,K(\Z/2,1)]$ by precomposition.
\end{lemma}

\begin{proof}
Any automorphism of $\xi$ gives a map $B\pi\times BSpin\to B\pi$ and a lift of $B\pi\times BSpin\to BSO$ to $BSpin$.
Let us first consider the map $[B\pi\times BSpin ,B\pi]$.
Since $BSpin$ is simply connected, we have
\[[B\pi\times BSpin,B\pi]\cong [B\pi,B\pi].\]
The homotopy classes of homotopy equivalences in $[B\pi,B\pi]$ are in bijective correspondence with $\Out(\pi)$.  This determines a map $\Aut(\xi) \to \Out(\pi)$.

Since a possible lift of $\gamma\circ pr_2 \colon B\pi\times BSpin\to BSO$ to $BSpin$ is given by the projection to the second factor, any other lift is determined by a map the homotopy fibre of $\gamma\colon BSpin\to BSO$, which is a $K(\Z/2,1)$. Thus a lift corresponds to an element of
\[H^1(B\pi\times BSpin,\Z/2)\cong H^1(B\pi;\Z/2).\]
Thus the kernel of the map $\Aut(\xi) \to \Out(\pi)$ is identified with $H^1(B\pi;\Z/2)$ and so we have a short exact sequence
\[1 \to H^1(B\pi;\Z/2) \to \Aut(\xi)\to \Out(\pi)\to 1.\]
It remains to to prove that $\Aut(\xi)$ is a semi-direct product as claimed.
First, that the sequence splits is straightforward. This can be seen as follows:  an outer automorphism $\rho \in \Out(\pi)$ determines a homotopy class of maps $\rho\colon B\pi \to B\pi$, by a slight abuse of notation, and so gives rise to a homotopy class of maps $(\rho,\id_{BSpin})\colon B\pi \times BSpin \to B\pi\times BSpin$, and thus produces an element of $\Aut(\xi)$.  It is not too hard to see that this map is a group homomorphism. Thus the sequence splits and $\Aut(\xi)$ is indeed a semi-direct product.

Finally, we argue that the action, in the group law of the semi-direct product, of $\rho \in \Out(\pi)$ on $H^1(B\pi;\Z/2) \cong [B\pi , B \Z/2]$, is that given by precomposition with the map in $[B\pi,B\pi]$ determined by $\rho$.  To see this, consider the following diagram, where $m_1$ and $m_2$ are maps $B\pi\to BSpin$ such that $\gamma\circ m_i\colon B\pi\to BSO$ is null homotopic, corresponding to elements of $H^1(B\pi;\Z/2)$.
\[\begin{tikzcd}[row sep = small]
B\pi\ar[d, phantom, "\times"]\ar[r, "\rho_1"]\ar[rd, "m_1" description]&B\pi\ar[d, phantom, "\times"]\ar[r, "\rho_2"]\ar[rd, "m_2" description]&B\pi\ar[d, phantom, "\times"]\\
BSpin\ar[r, "\id"']&BSpin\ar[r, "\id"']&BSpin
\end{tikzcd}\]
The composition is precisely the claimed product on $\Aut(\xi)$.
\end{proof}

\smallskip
\begin{thm}\label{thm:bordism-group-spin-case}
Let $\pi$ be a COAT group and let $\xi\colon B\to BSO$ be $p\circ pr_2\colon B\pi\times BSpin\to BSO$. Then
\[\Omega_4(\xi)\cong H_0(B\pi;\Omega_4^{Spin})\oplus H_2(B\pi;\Omega_2^{Spin})\oplus H_3(B\pi;\Omega_1^{Spin})\cong 16\cdot\Z \oplus \Hom(\pi;\Z/2)\oplus \Z/2.\]
Here, the $16 \cdot \Z$-factor is given by the signature.
\end{thm}

\begin{proof}
Consider the James spectral sequence associated to the fibration
\begin{equation}\label{fibration}
\xymatrix{BSpin \ar[r] & B\pi \times BSpin \ar[r] & B\pi,}
\end{equation}
with generalised homology theory $h_*=\pi^s_*$, the stable homotopy groups.  The $E^2$ page consists of the groups $H_p(B\pi;\Omega_q^{Spin})$.  There are nontrivial terms with $p+q=4$ for $p=0,2,3$, namely $H_0(B\pi;\Omega_4^{Spin}) \cong \Z$, $H_3(B\pi;\Omega_1^{Spin}) \cong \Z/2$ and
\[H_2(B\pi;\Omega_2^{Spin}) \cong H_2(B\pi;\Z/2) \cong H^1(B\pi;\Z/2) \cong \Hom(\pi,\Z/2),\]
with the latter two isomorphisms given by Poincar\'{e} duality and universal coefficients.

There can be at most two non-zero differentials that contribute to the $4$-line, i.e.\ the terms $E^{\infty}_{p,q}$ with $p+q=4$.  All other possible differentials start or end at $0$ since (i) $H_p(B\pi;A) = 0$ for $p >3$, for any choice of coefficient group $A$, (ii) $\Omega_3^{Spin} =0$, and (iii) it is a first quadrant spectral sequence.  One of the possibly  nontrivial differentials is
\[ d_2\colon H_3(B\pi;\Z/2) \to H_1(B\pi;\Z/2).\]
However, this differential is dual to $\Sq^2$, according to~\cite[Theorem 3.1.3]{teichnerthesis}, and hence vanishes since $\Sq^q \colon H^n \to H^{n+q}$ is zero whenever $n <q$.
The other potentially nontrivial differential is
\[d_3 \colon H_3(B\pi;\Omega_2^{Spin}) \to H_0(B\pi;\Omega_4^{Spin}).\]
However $H_3(B\pi;\Omega_2^{Spin}) \cong \Z/2$ and $H_0(B\pi;\Omega_4^{Spin}) \cong 16\cdot \Z$, so there can be no nontrivial homomorphism.
(The vanishing of this differential is also a consequence of the claim below.)
Thus all of the 4-line on $E^2$ page survives to the $E^{\infty}$ page, and we obtain a filtration
\[0 \subset F_1 \subset F_2 \subset F_3 = \Omega_4(\xi)\]
with $F_1 \cong 16\cdot \Z$, $F_2/F_1 \cong H_2(B\pi;\Z/2)$ and $F_3/F_2 \cong \Z/2$.
\smallskip
\begin{claim}
The subset $F_1 \cong 16\cdot \Z$ is a direct summand of $\Omega_4(\xi)$.
\end{claim}
\smallskip

To prove the claim we argue as follows.
We can restrict the fibration~(\ref{fibration}) to a basepoint in $B\pi$. The resulting fibration is a retract of~(\ref{fibration}) which commutes with the maps to $BSO$, and hence the naturality of the James spectral sequence implies that in the James spectral sequence for~(\ref{fibration}), the $y$-axis splits as a direct summand of $\Omega_*(\xi)$.
This completes the proof of the claim.

The intersection of the $4$-line and the $y$-axis is precisely $\Omega_4^{Spin}$, which is isomorphic to $16\cdot \Z$ by taking the signature.
In particular, as noted above, the claim implies that all differentials with image in $H_0(B\pi;\Omega_4^{Spin})$ are trivial.

It remains to argue why $F_2$ is also a direct summand in $\Omega_4(\xi)$.  This will follow from the next claim.
Denote the quotient $\Omega_4(\xi)/F_1$ by $\widetilde{\Omega}_4(\xi)$; this is sometimes called the reduced bordism group.

\smallskip
\begin{claim}
The subset $F_2/F_1\cong H_2(B\pi;\Z/2)$ is a direct summand of $\widetilde{\Omega}_4(\xi)$.
\end{claim}
\smallskip

We have a short exact sequence
\begin{equation}\label{split-exact-sequence}
\xymatrix{0 \ar[r] & H_2(B\pi;\Z/2) \ar[r] & \widetilde{\Omega}_4(\xi) \ar[r] & H_3(B\pi;\Z/2) \ar[r] & 0. }
\end{equation}
We will construct a splitting of this sequence in \cref{lemma:split}, but one can also abstractly see that this sequence must split, which will prove the claim. We have seen that $\Omega_*(\xi) \cong \Omega_*^{Spin}(B\pi)$, so the bordism group we want to compute is the ordinary spin bordism of $B\pi$.
Since $B\pi$ has a model which is a closed orientable $3$-manifold $X$, and orientable $3$-manifolds are parallelisable, it follows that the stable normal bundle is trivial. In particular the Spivak fibration of $X$ is trivial.
Also from the fact that $X$ is a manifold, it follows that $X$ has a CW structure with a unique $3$-cell.
From \cref{lemma:SpivakFibrationtrivial} below, it follows that the top (3-dimensional) cell of $X$, in a CW-structure on $X$ with only one 3-cell, splits stably, by which we mean that
 the attaching map $S^2 \to X^{(2)}$ is stably null homotopic.  Here stably means after suspending the attaching map sufficiently many times.
The naturality of the Atiyah-Hirzebruch spectral sequence for spin bordism thus implies that the contribution of the $3$-cell of $X$ is a direct summand because reduced homology theories (such as $\widetilde{\Omega}_*^{Spin}$) satisfy $\widetilde{\Omega}_i^{Spin}(Y) \cong \widetilde{\Omega}_{i+1}^{Spin}(\Sigma Y)$ and send wedges of spaces to direct sums of abelian groups.  This completes the proof of the claim that  $F_2/F_1$ is a direct summand of $\widetilde{\Omega}_4(\xi)$.  Since $F_2/F_1$ is identified with $\Hom(\pi,\Z/2)$, this completes the proof of \cref{thm:bordism-group-spin-case}.
\end{proof}

The next lemma may be of some independent interest.

\smallskip \begin{lemma}\label{lemma:SpivakFibrationtrivial}
Suppose $X$ is an $n$-dimensional Poincar\'e complex with a CW structure that has precisely one $n$-dimensional cell. Let $\varphi\colon S^{n-1} \to X^{(n-1)}$ be the attaching map of this cell.
Then $\varphi$ is stably null homotopic if and only if the Spivak normal fibration of $X$ is trivial.
\end{lemma}

\begin{proof}
Denote the Spivak normal fibration of $X$ by $SF(X)$. From the uniqueness property of the Spivak normal fibration of an $n$-dimensional Poincar\'e complex \cite{Spivak} and \cite[Definition 3.57 and Theorem 3.59]{luecksurgery}, it follows that $SF(X)$ is trivial if and only if there exists a $k \geq 0$ and a map
$e\colon S^{k+n} \to S^k \wedge X_+$ such that the composite
\[\xymatrix@C=1.7cm{S^{k+n} \ar[r]^-{e} & S^k \wedge X_+ \ar[r]^-{S^k\wedge \mathrm{collapse}} & S^{k+n}}\]
has degree one. Here $\mathrm{collapse}$ denotes the map that collapses the $(n-1)$-skeleton of $X$.


Assume that the Spivak normal fibration $SF(X)$ is trivial.
Observe that we have a factorisation
\[S^k \wedge X_+ \to S^k \wedge X \xrightarrow{S^k\wedge \mathrm{collapse}}  S^{k+n},\]
where the first map is the quotient by $S^k \times \{\ast_X\}$, with $\ast_X$ the basepoint of~$X$.  Therefore triviality of the Spivak normal fibration implies the existence of a map $e' \colon S^{k+n} \to S^k \wedge X$ that yields a degree one map when composed with the collapse map $S^k \wedge X \to S^{k+n}$.

Recall that having precisely one $n$-cell in~$X$ amounts to the fact that there is a cofibration sequence
\[\xymatrix@C=1.2cm{S^{n-1} \ar[r]^-{\varphi} & X^{(n-1)} \ar[r] & X \ar[r]^-{\mathrm{collapse}} & S^n }\]
which fits (after suspending $k$~times) into a diagram
\[\xymatrix{S^{k+n-1} \ar[r]^-{S^k\wedge\varphi} & S^k \wedge X^{(n-1)} \ar[r] & S^k \wedge X \ar[r] & S^{k+n} \ar[r]^-{S^{k+1} \wedge \varphi} & S^{k+1} \wedge X^{(n-1)} \\ & & S^{k+n} \ar[u]_-{e'} \ar@/_1pc/[ur]_{\deg \pm 1} & & }\]
in which the composition of any two horizontal maps  is null homotopic.  In particular the composition
\[\xymatrix{S^{k+n} \ar[r]^-{e'} & S^k \wedge X \ar[r] & S^{k+n} \ar[r]^-{S^{k+1} \wedge \varphi} & S^{k+1} \wedge X^{(n-1)} }\]
is null homotopic. Since the composition of the first two maps has degree one, it follows that $S^{k+1}\wedge \varphi$ is null homotopic as claimed.

For the converse, suppose now that $\varphi$ is stably null homotopic. Then there is a $k \geq 0$ such that $S^k \wedge X$ is homotopy equivalent to $S^k \wedge X^{(n-1)} \vee S^{n+k}$. We thus obtain a map $S^{n+k} \to S^k \wedge X$ whose composite with the suspended collapse map
\[S^k\wedge X \to S^k \wedge X^{(n-1)} \vee S^{n+k} \to  S^{n+k}\] has degree one, possibly after precomposing with a degree $-1$ map $S^{n+k} \to S^{n+k}$ to arrange that the degree be positive. Now observe that there is a homotopy equivalence
\[ S^k\wedge X_+ \simeq (S^k \wedge X) \vee S^k.\]
Use this equivalence to obtain a map \[e \colon S^{k+n} \to S^k \wedge X \to (S^k \wedge X) \vee S^k \to S^{k} \wedge X_+\]
 as desired.
\end{proof}

\smallskip
\begin{rem}
Sometimes it is written in the literature that the top cell of a framed manifold splits off stably. This lemma tells us that this is true, but the proof does not require the full tangential structure of a framed manifold. Really only the underlying Poincar\'e complex is relevant.
\end{rem}
\smallskip

For the rest of this subsection we restrict our attention to COAT fundamental groups.
We can say more than asserting the existence of an abstract direct sum decomposition of $\wt{\Omega}_4(\xi)$.  A better understanding of the invariants representing the $H_2(B\pi;\Z/2)$ summand will be crucial for computing the action of $\Aut(\xi)$.
The Kronecker evaluation map $\kappa\colon H_2(B\pi;\Z/2) \to \Hom(H^2(B\pi;\Z),\Z/2)$
is an isomorphism since $H^3(B\pi;\Z)\cong \Z$ is free.  Next we will define a map \[\Phi\colon \wt{\Omega}_4^{Spin}(B\pi) \to \Hom(H^2(B\pi;\Z),\Z/2).\]

Let $X$ be an aspherical $3$-manifold such that $X \simeq B\pi$. In fact, by JSJ decomposition and the geometrization theorem, any two such manifolds are diffeomorphic, but we do not need this fact.
Let $\lbrack M \xrightarrow{c} X \rbrack$ be an element of $\wt{\Omega}_4^{Spin}(B\pi)$ and let $\sigma$ be a spin structure on $X$.
We define a map $\psi_{c,\sigma}\colon H_1(X;\bbZ)\to \bbZ/2$ in the following way. Represent $x\in H_1(X;\bbZ)$ by an embedding $S=\coprod S^1\to X$ and consider the spin structure on $S$ that makes  each connected component of $S$ a spin null-bordant surface.  Let $F\subseteq M$ be a regular preimage of $S$ under $c$. The spin structures on $X$ and $S$ induce a spin structure on the normal bundle of~$S$ in~$X$, and this pulls back to a spin structure on the normal bundle of~$F$ in~$M$. Together with the spin structure on~$M$ this determines a spin structure on~$F$, so we can view $\lbrack F_x \rbrack \in \Omega_2^{Spin}(\ast)$.
A spin structure on a surface $F$ determines a quadratic refinement $\mu \colon H_1(F;\Z/2) \to \Z/2$ of the intersection form on $H_1(F;\Z/2)$ (see \cref{defn:quadratic-refinement}).  The Arf invariant of a quadratic form is an element of $\Z/2$.  See R.~Kirby \cite[Appendix]{Kirby-4-manifold-book} for a concise treatment.
We define $\psi_{c,\sigma}(x) := \Arf(\lbrack F\rbrack)\in \bbZ/2$.

\smallskip
\begin{lemma}\label{lem:split-homomorphism}
The map $\psi_{c,\sigma}\colon H_1(X;\bbZ)\to \bbZ/2$ is a well-defined homomorphism.
\end{lemma}

\begin{proof}
First we will show that $\psi_{c,\sigma}$ only depends on the bordism class of \mbox{$\lbrack M \xrightarrow{c} X \rbrack$.} For any spin bordism $g\colon W\to X$, the regular preimage of an embedding $S\to X$ is a spin bordism between the regular preimages in the two boundaries of the cobordism, and the Arf invariant is an isomorphism from two-dimensional spin bordism to $\Z/2$. 

To see that $\psi$ is well defined, we also have to check that $\psi_{c,\sigma}(x)$ does not depend on the choice of the embedding $S \hookrightarrow X$. Any two choices $S_0, S_1$ are bordant, since they represent the same homology class in a 3-manifold,
 and the component-wise null bordant spin structure on both ends can be extended over the cobordism. Embed the cobordism in $X\times[0,1]$ and take a regular preimage in $M\times[0,1]$ under $c\times \id_{[0,1]}$, to yield a spin cobordism between the preimages $F_0$ and $F_1$ of $S_0$ and $S_1$. Therefore, $\Arf(\lbrack F_0\rbrack)=\Arf(\lbrack F_1\rbrack)$ and $\psi_{c,\sigma}(x)$ is well defined.

It remains to check that $\psi_{c,\sigma}$ is a homomorphism. A class $x+y\in H_1(X;\bbZ)$ can be represented by the union of disjoint embeddings $S_x\to X$ and $S_y\to X$ which represent $x$ and $y$ respectively. Taking null bordant spin structures on $S_x$ and $S_y$ also gives a null bordant spin structure on the union. Let $F_x$ and $F_y$ be the preimages of $S_x$ and $S_y$ respectively.  By the additivity of the Arf invariant we obtain
\[\psi_{c,\sigma}(x+y)=\Arf(\lbrack F_x+F_y\rbrack)=\Arf(\lbrack F_x\rbrack)+\Arf(\lbrack F_y\rbrack)=\psi_{c,\sigma}(x)+\psi_{c,\sigma}(y).\]
This completes the proof of \cref{lem:split-homomorphism}.
\end{proof}

Now we can define
\[ \begin{array}{rcl} \widetilde{\Omega}_4^{Spin}(B\pi) & \xrightarrow{\Phi} &  \Hom(H^2(B\pi;\Z),\Z/2) \\
  \lbrack c \colon M \to X \rbrack &\mapsto& \psi_{c,\sigma} \circ PD \end{array}\]
where $PD$ denotes Poincar\'{e} duality.  In the next lemma we show that the map $\Phi$ gives us the desired splitting.

\smallskip
\begin{rem}
The construction of $\Phi$ depends on the choice of a spin structure on $X$.  We remark that set of the spin structures on $X$ are in bijective correspondence with the possible splittings of the sequence under consideration, since both sets are (non-canonically) isomorphic to $H^1(X;\Z/2)$.  We conjecture that the $\Phi$ construction gives rise to an explicit such correspondence.
\end{rem}

\smallskip
\begin{lemma}\label{lemma:split}
The map
\[\Phi \colon \widetilde{\Omega}_4^{Spin}(B\pi) \to \Hom(H^2(B\pi;\Z),\Z/2)\]
splits the short exact sequence (\ref{split-exact-sequence}), where we identify \[H_2(B\pi;\Z/2) \toiso \Hom(H^2(B\pi;\Z);\Z/2)\] via the Kronecker evaluation map~$\kappa$.
\end{lemma}

\begin{proof}
We have a diagram
\[\xymatrix{0 \ar[r] & H_2(B\pi;\Z/2) \ar[d]_{\cong}^{\kappa} \ar[r]^-j & \wt{\Omega}_4^{Spin}(B\pi) \ar[r] \ar[dl]^{\Phi} & H_3(B\pi;\Z/2) \ar[r] & 0 \\ & \Hom(H^2(B\pi;\Z);\Z/2) & \wt{\Omega}_4^{Spin}(B\pi^{(2)}) \ar[u]_-{i_*} \ar@{->>}[l]_-{p} }\]

The map $p$ is described via \cref{lem:arf} as follows. In a CW structure on $B\pi$ as a 3-complex with only one $3$-cell, the differential in the cellular cochain complex
\[C_{cell}^2(B\pi) \xrightarrow{\delta_2} C_{cell}^3(B\pi) \cong \Z\] is trivial.
A $2$-cell determines a 2-cochain,  $e_k^* \in C_{cell}^2(B\pi) = \Hom_{\Z}(C_2^{cell}(B\pi),\Z)$ by $e_k^*(e^2_j) = \delta_{kj}$.
Since the coboundary map $\delta_2=0$, every $2$-cell $e^2_k$ determines an element $[e_k^*]$ in $H^2(B\pi;\Z)$ and the inclusion $B\pi^{(2)} \subset B\pi$ induces an isomorphism on second cohomology.

The map $p$ sends a class $\lbrack M \xrightarrow{c} B\pi^{(2)} \rbrack$ to the map in $\Hom_{\Z/2}(H^2(B\pi;\Z/2),\Z/2)$ that sends $\lbrack e^*_k \rbrack$ to the Arf invariant $\Arf(c^{-1}(b_k^2))$, where $b_k^2 \in e_k^2$ denotes the barycentre of the $k^{th}$ $2$-cell (which we can assume after a small homotopy of $c$ to be a regular point).
Since $p$ is surjective the lemma follows if we can show the following

\smallskip
\begin{claim}
We have
\[ p = \Phi \circ i_* \colon \wt{\Omega}_4^{Spin}(B\pi) \to \Hom(H^2(B\pi;\Z),\Z/2).\]
\end{claim}
\smallskip

For each cell $e^2_k$ there is an embedding $\alpha_k\colon S^1\to X$ that intersects the 2-skeleton only in $b^2_k$ and there only once. To see this, join up two of the intersection points of the boundary of the 3-cell with $b^2_k$ using a path in the 3-cell.  Thus, $PD^{-1}([\alpha_k])=[e_k^*]$ and $c^{-1}(\alpha_k(S^1))=c^{-1}(b_k^2)$. Furthermore, the spin structure on the normal bundle of the (equal) preimages agree and we have:
\begin{align*}p(\lbrack M \xrightarrow{c} B\pi^{(2)} \rbrack)([e_k^*])&=\Arf(c^{-1}(b_k^2))=\Arf(c^{-1}(\alpha_k(S^1)))=\psi_{c,\sigma}([\alpha_k])\\& = (\psi_{c,\sigma}\circ PD)(e^*_k) =\Phi(i_*\lbrack M \xrightarrow{c} B\pi^{(2)} \rbrack)([e_k^*]).\end{align*}
This completes the proof of \cref{lemma:split}.
\end{proof}


To describe the action of $\Aut(\xi)$ on $\Omega_4(\xi)$ we need the following lemma.

\smallskip
\begin{lemma}
\label{lem:spinaction}
Given a surface $F$ with a spin structure and a map $f\colon F\to S^1$, the map $f$ induces a spin structure on a regular preimage $f^{-1}(\ast)$. We denote the spin bordism class of $f^{-1}(\ast)$ by $\mu(f^{-1}(\ast))\in\Omega_1^{Spin}\cong \bbZ/2$. Let $0\neq x\in H^1(S^1;\bbZ/2)$.  Then
\[\Arf(F)+\Arf(f^*(x)\cdot F)=\mu(f^{-1}(\ast))\in \bbZ/2,\]
where $f^*(x)\cdot F$ denotes the surface $F$ with the spin structure changed by $f^*(x)$.
\end{lemma}

\begin{proof}
First note that $[f^{-1}(\ast)]=PD(f^*(x))\in H_1(F;\bbZ/2)$.\\
\\
\textbf{Case 1:} The map $f_*\colon H_1(F;\bbZ/2)\to H_1(S^1;\bbZ/2)$ is trivial.
Then for any $y\in H_1(F;\bbZ/2)$ we have
\[f^*(x) \cap y= x \cap f_*y = 0 \in H_0(F;\Z/2) =\Z/2\] and thus $f^*(x)=0$ and $[f^{-1}(\ast)]=PD(f^*(x))=0$. This implies that \[\mu(f^{-1}(\ast))=0=\Arf(F)+\Arf(F)=\Arf(F)+\Arf(f^*(x)\cdot F).\]
~\\
\textbf{Case 2:} The map $f_*\colon H_1(F;\bbZ/2)\to H_1(S^1;\bbZ/2)$ is nontrivial.
Let $\alpha:=PD(f^*(x))\in H_1(F;\bbZ/2)$ and choose $\beta\in H_1(F;\bbZ/2)$ with $f_*(\beta)\neq 0$. Then, again identifying $H_0(F;\Z/2)$ with $\Z/2$, we have
\[\lambda(\alpha,\beta)= f^*(x) \cap \beta= x \cap f_*\beta = 1 \in\bbZ/2.\]
We can extend $\{\alpha,\beta\}$ to a basis $\{\alpha,\beta,\gamma_1,\delta_1,\dots,\gamma_{g-1},\delta_{g-1}\}$ with
\[\lambda(\gamma_i,\delta_j)=\begin{cases}1&i=j\\0&\text{else}\end{cases}\]
and all other intersections being zero. With these choices we have see that the action of~$f^*(x)$ on the spin structure gives $\mu(f^*(x)\cdot \beta)=1+\mu(\beta)\in\Omega_1^{Spin}$, and~$f^*(x)$ does not change the spin bordism classes of the other basis elements. Therefore,
\begin{align*}
\Arf(f^*(x)\cdot F)&=\mu(f^*(x)\cdot \alpha)\mu(f^*(x)\cdot \beta)+\sum_i\mu(f^*(x)\cdot \gamma_i)\mu(f^*(x)\cdot \delta_i)\\&=\mu(\alpha)+\mu(\alpha)\mu(\beta)+\sum_i\mu(\gamma_i)\mu(\delta_i)=\mu(\alpha)+\Arf(F).
\end{align*}
The proof is completed by noting that  $\mu(\alpha)=\mu(f^{-1}(\ast))$ by definition.
\end{proof}

Next we use our understanding of the splitting map $\Phi$ to compute the action of $\Aut(\xi)$ on $\Omega_4(\xi)$.  Let $\sigma\colon X\to BSpin$ denote the spin structure on $X$ used for the definition of the splitting $\Phi$. For an element $\rho\in\Out(\pi)$, view $\rho$ as a homotopy equivalence $X\to X$, and denote the difference between the spin structures $\sigma$ and $\sigma\circ\rho$ by $m_\rho\in H^1(X;\Z/2)$.

\smallskip \begin{thm}\label{thm:action-spin-case}
The action of $\Aut(\xi)$ on $\Omega_4(\xi)$ is given in the following way. Let $(z,\phi,\epsilon)\in 16\cdot\Z\oplus \Hom(\pi;\bbZ/2)\oplus\bbZ/2$ be given.
\smallskip
\begin{enumerate}[(i)]
\item An element $m\in H^1(B\pi;\bbZ/2)\cong \Hom(\pi;\bbZ/2)$ acts on $(z,\phi,\epsilon)$ by
\[m\cdot (z,\phi,\epsilon)=(z,\phi+\epsilon m,\epsilon).\]
\item An outer automorphism $\rho \in \Out(\pi)$ acts by
\[\rho\cdot(z,\phi,\epsilon)=(z,\phi\circ \rho^{-1}+\epsilon m_\rho,\epsilon).\]
\end{enumerate}
\end{thm}

\begin{proof}
The elements in the $16\cdot \Z$ summand can be represented by connected sums of $K3$ surfaces. On these the action of $\Aut(\xi)$ is trivial, since they have a unique spin structure and the map to $B\pi$ factors through a point up to homotopy.
\smallskip
\begin{enumerate}[(i)]
\item \label{item-i-action-aut-xi-spin}
Recall that $X$ denotes a $3$-manifold model for $B\pi$, and recall that the splitting \[\Phi \colon \widetilde{\Omega}_4^{Spin}(B\pi) \to \Hom(H^2(B\pi;\Z),\Z/2)\] from \cref{lemma:split} is given in the following way.
Consider a diagram
\[\xymatrix{M\ar[r]^c & X\\F\ar[u]^j\ar[r]^f&S^1\ar[u]^i}\]
where $i\colon S^1\to X$ is an embedding, $F_i$ is its regular preimage under $c$, and $f = c|_{F_i}$.  The embedding $i \colon S^1 \to X$ represents an element of $H_1(X;\Z) \cong H^2(X;\Z)$.  Then $\Phi(\lbrack M \xrightarrow{c} B\pi \rbrack)$ sends $[i\colon S^1\to X]$ to $\Arf(F_i)$.

Changing the spin structure of $M$ by $c^*(m)\in H^1(M;\Z/2)$ changes the induced spin structure on $F_i$ by $(c|_{F_i})^*(m)=j^*c^*(m)$. On the other hand, changing the spin structure $\sigma$ of $X$ by $m$ changes the spin structure on the normal bundle of $S^1\subseteq X$ by $m|_{S^1}=i^*(m) \in H^1(S^1;\Z/2)$ and hence this change also alters the induced spin structure on $F_i$ by $f^*i^*(m)=j^*c^*(m)$. Therefore, the action of $m$ on the bordism group can be described by letting it act on the spin structure of $X$.

By \cref{lem:spinaction}, this action of $m \in H^1(X;\Z/2)$ on the spin structure of $X$ changes the Arf invariant by $[f^{-1}(\ast)] \in \Omega_1^{Spin} \cong \Z/2$ if $i^*(m)\neq 0$. By \cref{lem:arf}, we have that $\epsilon=[f^{-1}(\ast)]\in \Omega_1^{Spin}$, and thus if $\epsilon=0$ the element $m\in H^1(B\pi;\bbZ/2)$ acts trivally. On the other hand if $\epsilon=1$, then $m$ changes the Arf invariant associated to the element $[i\colon S^1\to X]\in H_1(X;\bbZ)$ if and only if $m(i)
=i^*(m) \neq 0$.
\item
An automorphism of $\pi$ induces an automorphism of $H_3(B\pi;\Z/2) \cong \Z/2$. However there is only one automorphism of the group~$\Z/2$, hence $\epsilon$ is unchanged by $\rho$.

As in (\ref{item-i-action-aut-xi-spin}), the element in $\Hom(\pi;\Z/2)$ associated to $M$ is computed by considering the Arf invariants of surfaces $F_i=c^{-1}(i(S^1))$.  For an element $g\in \pi$, represent $g$ by an embedding $i\colon S^1\to X$, and compute $\Arf(F_i)$. When applying $\rho$, for an embedding $i\colon S^1\to X$, we have to compute the Arf invariant of the surface $(\rho\circ c)^{-1}(i(S^1))=c^{-1}((\rho^{-1}\circ i)(S^1))=F_{\rho^{-1}\circ i}$. Hence one might suspect that $\rho$ acts by sending $\phi$ to $\phi\circ \rho^{-1}$. But applying $\rho$ also changes, by $m_\rho$, the spin structure on $X$ that is used to compute the Arf invariant of $F_{\rho^{-1}\circ i}$.
Therefore, the argument of (\ref{item-i-action-aut-xi-spin}) applies, with $m=m_{\rho}$, to show that we have an extra summand $\epsilon m_\rho$.
%
%
\end{enumerate}
\end{proof}

From the results of this section we obtain the following corollary, which is \cref{thm:A}~(\ref{item:A2}).
Before stating the corollary we collect the notation that will appear in the statement.
As above, let $X$ be a closed oriented aspherical $3$-manifold with fundamental group $\pi$.
For a $4$-manifold $M$, an isomorphism $\pi_1(M) \to \pi$ determines, up to homotopy, a map $c \colon M \to X$.  The following two inverse image constructions, together with the signature, will be used to state the spin classification in \cref{cor:stable-class-spin}.

The inverse image of a regular point $c^{-1}(\pt) \in M$ determines an element  $S \in \Omega_1^{Spin} \cong \Z/2$.
Now choose a spin structure $\sigma$ on $X$.
The map $\Phi\colon H_1(X;\bbZ)\to \bbZ/2$, defined above using the  Arf invariants of certain inverse images, determines an element of $H_2(B\pi,\Z/2)$ by universal coefficients and Poincar\'{e} duality.

\smallskip \begin{cor}\label{cor:stable-class-spin}
The stable diffeomorphism classes of spin $4$-manifolds with COAT fundamental group $\pi$ are in one-to-one correspondence with
$$16\cdot \Z \times \left(H_2(B\pi;\Z/2)/\Out(\pi) \cup \{ \ast\}\right).$$
The $16\cdot\Z$ entry is detected by the signature.
The extra element $\{ \ast\}$ corresponds to the case that $S=1 \in \Omega_1^{Spin}$.
If $S=0$, the element in $H_2(B\pi;\Z/2)/\Out(\pi)$ is determined by the Arf invariants via the map $\Phi$.
\end{cor}
\smallskip

\noindent The element $\{\ast\}$ was called $\{odd\}$ in \cref{thm:A} in the introduction.

\smallskip

\begin{example}\label{example:Z-hoch-drei}
For $\pi\cong\Z^3$ two elements $(n,\phi,\epsilon),(n',\phi',\epsilon')\in\Omega_4(\xi)\cong 16\cdot\Z\oplus \Hom(\Z^3;\Z/2)\oplus \Z/2$ with the same signature, i.e.\ $n=n'$, correspond to stably diffeomorphic 4-manifolds if and only if
\[\begin{cases}\epsilon=\epsilon'=1&\text{or}\\\epsilon=\epsilon'=0,~\phi=\phi'=0&\text{or}\\\epsilon=\epsilon'=0,~\phi\neq0\neq \phi'.&\end{cases}\]
We used the fact that the canonical map $\GL_n(\Z) \to \GL_n(\Z/2)$ is surjective, in particular for $n=3$.
\end{example}

\subsection{Almost spin $4$-manifolds}\label{sec:almostspin}

We begin our investigation of almost spin 4-manifolds by producing a unique lift $w \in H^2(B\pi;\Z/2)$ of $w_2(M)$.  The first part of this section applies for a larger class of groups than 3-manifold groups; we will point out when we restrict to COAT groups.

\smallskip \begin{lemma}\label{lem:w}
Let $\pi$ be a group, let $M$ be an almost spin $4$-manifold and let $c\colon M \to B\pi$ induce an isomorphism on fundamental groups. Then there exists a unique element $w \in H^2(B\pi;\Z/2)$ such that $c^*(w) = w_2(M)$. If $\pi$ is such that $H^3(B\pi;\Z)$ is $2$-torsion free, then $w = w_2(E)$ for some complex line bundle $E$ over $B\pi$.
\end{lemma}

\begin{proof}
The first part follows if we can establish the following exact sequence
\[\xymatrix{0 \ar[r] & H^2(B\pi;\Z/2) \ar[r]^-{c^*} & H^2(M;\Z/2) \ar[r]^-{p^*} & H^2(\widetilde{M};\Z/2)^{\pi}}\]
where the superscript $\pi$ denotes the fixed point set of the $\pi$-action. This is because by assumption $0 = w_2(\wt{M}) = p^*(w_2(M))$ since $p^*(TM) = T\wt{M}$.

To see why this sequence is exact, consider the Serre spectral sequence applied to the fibration
\[\xymatrix{\widetilde{M} \xrightarrow{p} M \xrightarrow{c} B\pi.}\]
Its $E^2$-term is
\[H^p(B\pi;H^q(\wt{M};\Z/2)) \Longrightarrow H^{p+q}(M;\Z/2)\]
where $H^q(\wt{M};\Z/2)$ is to be understood as module over $\pi$.
On the 2-line the non-vanishing terms are  $H^2(B\pi;\Z/2)$ and $H^0(B\pi;H^2(\wt{M};\Z/2)) \cong H^2(\wt{M};\Z/2)^\pi$.
Since $H^1(\wt{M};\Z/2) =0$, the only potentially nonzero differential which can affect the $E^{\infty}$-page is $d^3 \colon H^2(\wt{M};\Z/2)^\pi \to H^3(B\pi;H^0(\wt{M};\Z/2))$. Thus the exact sequence exists as claimed.

From the Bockstein sequence associated to $0 \to \Z \xrightarrow{2} \Z \to \Z/2 \to 0$, we see that the map
\[\xymatrix{H^2(B\pi;\Z) \ar[r]^-{\red_2} & H^2(B\pi;\Z/2)}\]
is surjective, because multiplication by two on coefficients induces an injection on $H^3(B\pi;\Z)$ by assumption.

To prove the second statement of the lemma, choose a complex line bundle $E \to B\pi$ whose first Chern class $c_1(E)$ is a lift of $w$ to $H^2(B\pi;\Z)$ (recall that complex line bundles are classified by their first Chern class). Furthermore the second Stiefel-Whitney class of the underlying 2-dimensional real vector bundle is the reduction of the first Chern class: $w_2(E) = \red_2(c_1(E))$.
\end{proof}

Now let $\pi$ be a group for which $H^3(B\pi;\Z)$ is 2-torsion free, fix a choice of complex line bundle $E$ provided by \cref{lem:w} and consider it as a 2-dimensional real vector bundle.

\smallskip \begin{lemma}\label{lemma:almostspin}
The normal $1$-type of an almost spin manifold $M$ with fundamental group $\pi$ is given by
\[\xymatrix{\xi\colon B\pi \times BSpin \ar[r]^-{E\times p} & BSO\times BSO \ar[r]^-{\oplus} & BSO }\]
where $E \to B\pi$ is a stable vector bundle such that $c^*(w_2(E)) = w_2(M)$ and $\oplus$ refers to the $H$-space structure on $BSO$ that comes from the Whitney sum of stable vector bundles.
\end{lemma}

\begin{proof}
To see that the map $\xi$ is $2$-coconnected note that since $B\pi$ has vanishing higher homotopy groups, $\pi_i(B\pi \times BSpin) \cong \pi_i(BSpin)$ for $i >1$, and $\xi$ restricted to $BSpin$ is the canonical map, which is $2$-coconnected.

For simplicity denote the bundle over $M$ given by $\nu_M \oplus c^*(-E)$ by $\nu(E)$. Here $-E$ is the stable inverse bundle to $E$.
The bundle $\nu(E)$ has a spin structure as, by design, $w_2(\nu(E)) = 0$. Denote some choice of lift of the classifying map $\nu(E)\colon M \to BSO$ to $BSpin$ by $\wt{\nu}(E) \colon M \to BSpin$. Now consider the following diagram
\[\xymatrix{ & B\pi \times BSpin \ar[d]^{\xi} \\
    M \ar[r]_-{\nu_M} \ar@/^1pc/[ur]^-{c\times\wt{\nu}(E)} & BSO}\]
This diagram commutes because it commutes up to homotopy (the composition $\xi \circ (c \times \wt{\nu}(E))$ classifies the bundle $\nu_M \oplus c^*(-E) \oplus c^*(E) \cong \nu_M$). Since $\xi$ is a fibration we can use the homotopy lifting property to change the map $c \times \wt{\nu}(E)$ in its homotopy class to make the diagram commute strictly.
The map $M \to B\pi \times BSpin$ described is $2$-connected since $c$ induces an isomorphism on $\pi_1$ and $\pi_2(B\pi \times BSpin) =0$. This completes the proof of the lemma.
\end{proof}

\smallskip

\begin{defi}\label{defn:out-pi-w}
Let $\Out(\pi)_w$ be the subgroup of $\Out(\pi)$ given by those elements $f \in \Out(\pi)$ such that $f^*(w) = w \in H^2(\pi;\Z/2)$, where $w$ is as in \cref{lem:w}.
\end{defi}

\smallskip \begin{lemma}
Let $\xi\colon B\pi \times BSpin \to BSO$ be as in \cref{lemma:almostspin}. We have a short exact sequence
\[0 \to H^1(B\pi;\bbZ/2) \to \Aut(\xi)\to \Out(\pi)_w \to 1.\]
\end{lemma}

\begin{proof}
Consider an automorphism $\Phi\in \Aut(\xi)$ which is, in particular, a pair of maps $(\varphi,\psi):= (p_1 \circ \Phi, p_2\circ \Phi)$, where $p_1$ and $p_2$ are the projections, making the diagram
\[\xymatrix@C=1.5cm{B\pi\times BSpin \ar[r]^-{(\varphi,\psi)} \ar[d]_-{E\times \gamma} & B\pi\times BSpin \ar[d]^-{E\times \gamma} \\ BSO \ar@{=}[r] & BSO }\]
commute up to homotopy, where $E$ is the 2-dimensional real vector bundle associated to the complex line bundle from \cref{lem:w} and $\gamma$ denotes the tautological oriented bundle over $BSpin$. We again denote $w_2(E)$ by $w$.
Since $BSpin$ is simply connected, we can factor $\varphi$ as follows
\[\xymatrix{B\pi\times BSpin \ar[r]^-\varphi \ar[d]_-{p_1} & B\pi \\ B\pi \ar@/_1pc/[ur]_-{\widehat{\varphi}} & }\]
The commutativity of the above two diagrams give rise to the following isomorphisms of \emph{stable} bundles.  The first diagram above gives the first isomorphism in the sequence below, while the second diagram gives the translation between the second and third isomorphisms.
\begin{align}\label{sequence-vector-bundle-isos}
& E\times \gamma  \cong (\varphi,\psi)^*(E\times \gamma) \\
& \Leftrightarrow p_1^*(E)\oplus p_2^*(\gamma) \cong \varphi^*(E)\oplus \psi^*(\gamma) \nonumber\\
& \Leftrightarrow p_1^*(E-\widehat{\varphi}^*(E)) \oplus p_2^*(\gamma) \cong \psi^*(\gamma)\nonumber \\
& \Leftrightarrow \left( E - \widehat{\varphi}^*(E) \right) \times \gamma \cong \psi^*(\gamma) \nonumber
\end{align}
This just says that $\psi$ is a spin structure on the stable vector bundle $\left(E - \widehat{\varphi}^*(E) \right) \times \gamma$ over $B\pi\times BSpin$.  That is, we have a commutative triangle
\[\xymatrix @C+1cm { & BSpin \ar[d]^-{\gamma} \\
B\pi \times BSpin \ar@/^1pc/[ur]^-{\psi} \ar[r]_-{\left(E-\widehat{\varphi}^*(E)\right) \times \gamma } & BSO. }\]
In particular it follows that
\begin{align*}
0  = w_2(\left(E - \widehat{\varphi}^*(E) \right) \times \gamma)  = w_2(\left(E - \widehat{\varphi}^*(E) \right))\times 1  = ( w-\widehat{\varphi}^*(w) )\times 1,
\end{align*}
which precisely means that $\widehat{\varphi} \in \Out(\pi)_w$.

The map $\Aut(\xi)\to \Out(\pi)_w$ given by $(\varphi,\psi)\mapsto \widehat\varphi:=\varphi\circ p_1$ is a group homomorphism. It is surjective by the following argument.

Starting with $\widehat\varphi\in \Out(\pi)_w$, choose a spin structure $m\colon B\pi\to BSpin$ on $E-\widehat\varphi^*(E)$.
The maps \[\varphi=\widehat\varphi\circ p_1\colon B\pi\times BSpin\to B\pi\] and \[\psi\colon B\pi\times BSpin\xrightarrow{(m,\id)} BSpin\times BSpin\xrightarrow{\oplus}BSpin\]
define an element $(\varphi,\psi)\in \Aut(\xi)$, which is a pre-image of $\widehat\varphi$.

The kernel of the above homomorphism $\Aut(\xi)\to \Out(\pi)_w$ can be identified with $H^1(B\pi;\Z/2)$ as follows. By the argument at the beginning of the proof, an element in $\Aut(\xi)$ is determined by an element $\widehat\varphi\in \Out(\pi)_w$ and a spin structure $\psi$ on $(E-\widehat\varphi^*(E))\times \gamma$. When $\widehat\varphi$ is the identity, $(E-\widehat\varphi^*(E))$ is the trivial bundle and the projection $p_2\colon B\pi\times BSpin\to BSpin$ is a spin structure on $(E-\widehat\varphi^*(E))\times \gamma$. Hence we can identify the kernel of $\Aut(\xi)\to \Out(\pi)_w$ with $H^1(B\pi;\Z/2)$ by comparing the spin structure $\psi$ to $p_2$.
\end{proof}

\noindent From now on in this section $\pi$ will be a COAT group.

\smallskip \begin{thm}
\label{thm:almostspinbord}
Let $\pi$ be a COAT group and let $\xi\colon B\pi \times BSpin \to BSO$ be as in \cref{lemma:almostspin}. Then we have a non split short exact sequence
\[0 \to 16\cdot \Z \to \Omega_4(\xi) \to H_2(B\pi;\Z/2) \to 0.\]
\end{thm}

\begin{proof}
Consider the following morphism of fibrations
\[\xymatrix{BSpin \ar[r] \ar@{=}[d] & BSpin\times B\pi \ar[r]^-p \ar[d]^-{\xi} & B\pi \ar[d]^w \\
    BSpin \ar[r] & BSO \ar[r] & K(\Z/2,2). }\]
\begin{claim}
The bordism group $\Omega_4(\xi)$ sits in a short exact sequence
\[\xymatrix{0 \ar[r] & \Omega_4^{Spin} \ar[r] & \Omega_4(\xi) \ar[r] & H_2(B\pi;\Z/2) \ar[r] & 0.}\]
\end{claim}
To prove the claim apply the James spectral sequence to the upper fibration.  We need to see that the surviving terms in the $E^{\infty}$ page of the $4$-line are $\Omega_4^{Spin}$ and $H_2(B\pi;\Z/2)$.  First, all differentials with $\Omega_4^{Spin}$ as target have a torsion group as domain.  Moreover there is a differential
\[\xymatrix{H_3(B\pi;\Z/2) \ar[r]^-{d_2} & H_1(B\pi;\Z/2),}\]
which according to \cite[Theorem 3.1.3]{teichnerthesis} is dual to the map
\[\begin{array}{rcl} H^1(B\pi;\Z/2) &\xrightarrow{\Sq^2_w}&  H^3(B\pi;\Z/2) \\
 x &\mapsto & \Sq^2(x) + x \cup w. \end{array}\]
The $\Sq^2$ summand vanishes, since as in the previous section $\Sq^n$ is trivial on $H^m$ for $m<n$.  Then as $0 \neq w \in H^2(B\pi;\Z/2)$, it follows from Poincar\'e duality on $B\pi$ that this differential is not trivial. Hence the $E^\infty$-terms in the spectral sequence on the $4$-line are exactly as claimed.

\smallskip
\begin{claim} The short exact sequence from the previous claim does not split.
\end{claim}
\smallskip

This is an immediate consequence of \cite[Main Theorem (3)]{teichnersignatures}, but for the convenience of the reader we give a proof here.
For this we see that the above sequence of fibrations induces a map of James spectral sequences.
Then we observe that the James spectral sequence for the lower fibration  $BSpin \to BSO \to K(\Z/2,2)$ has $E^2$-page
\[H_p(K(\Z/2,2);\Omega_q^{Spin}) \Longrightarrow \Omega^{SO}_{p+q} \]
Looking at the $4$-line of the spectral sequence, we obtain a short exact sequence
\[\xymatrix{0 \ar[r] & \Omega_4^{Spin} \ar[r] & F_2 \ar[r] & H_2(K(\Z/2,2);\Z/2) \cong \Z/2 = E^\infty_{2,2} \ar[r] & 0}\]
Since $F_2 \subseteq \Z \cong \Omega_4^{SO}$, and is nontrivial, $F_2$ is therefore itself isomorphic to $\Z$.
Thus this short exact sequence does not split.
From the morphism of spectral sequences we obtain a morphism of sequences
\[\xymatrix{0 \ar[r] & \Omega_4^{Spin} \ar[r] \ar@{=}[d] & \Omega_4(\xi) \ar[r] \ar[d] & H_2(B\pi;\Z/2) \ar[r] \ar@{->>}[d]^{w_*} & 0 \\
0 \ar[r] & \Omega_4^{Spin} \ar[r] & F_2 \ar[r] & H_2(K(\Z/2,2);\Z/2) \ar[r] & 0 }\]
The morphism $w_*$ from \cref{lem:w} is surjective because $0 \neq w \in H^2(B\pi;\Z/2)$. This prevents the upper sequence from splitting, since a choice of lift of $w_*$ and a splitting of the upper sequence would induce a splitting of the lower sequence. This completes the proof of the claim.

From this diagram it also follows that, as an abstract abelian group, we have a decomposition
$\Omega_4(\xi) \cong 8\cdot \Z \oplus \ker(\langle w,- \rangle)$
and the map $\Omega_4^{Spin} \to \Omega_4(\xi)$ is identified with multiplication by $2$ on the $\Z$ summand and is zero on the other summand. Hence the $8\cdot \Z$ summand in $\Omega_4(\xi)$ is given by the signature.  Note however that the splitting of $\Omega_4(\xi)$ into the direct sum is not canonical.
\end{proof}

\noindent In particular we obtain the following corollary.

\smallskip
\begin{cor}\label{cor:sign-div-eight}
An almost spin $4$-manifold $M$ with COAT fundamental group $\pi$ has signature divisible by $8$.
\end{cor}
\smallskip

\begin{rem}\label{rem:almost-spin-even}
For a manifold with $H_1(M;\Z)$ 2-primary torsion free, for example when $\pi \cong \Z^3$, this is rather interesting.  An orientable 4-manifold $M$ has even intersection form if and only if $w_2$ maps to zero in $\Hom(H_2(M,\Z),\Z/2)$, i.e.\ if it lies in $\Ext^1_{\Z}(H_1(M,\Z),\Z/2)$ \cite[p.~754, part (4)]{teichnersignatures}.
But if $H_1(M;\Z) \cong H_1(B\pi;\Z)$ contains no 2-primary torsion, then this $\Ext$-group vanishes, so the intersection form cannot be even in the case of almost spin manifolds (where $w_2 \neq 0$).  So this is ruled out as an explanation for the divisibility of the signature.
Contrast \cref{cor:sign-div-eight} with the existence of almost spin 4-manifolds with fundamental group $\Z/2 \times \Z/2$ with signature~$4$ (see \cite{teichnersignatures}) which arise as a quotient of an Enriques surface by a free antiholomorphic involution. Certainly such a manifold is almost spin (its universal cover is a $K3$ surface) and has signature $4$ (because $4\cdot 2\cdot 2 = 16 = \mathrm{sign}(K3)$).
\end{rem}
\smallskip

We postpone the discussion of the action of the automorphisms on the normal $1$-type on the bordism set until after the treatment of the stable homeomorphism question in the next section, since we make use of the action in the topological case to understand the action in the smooth case.


\section{Stable homeomorphism classification}\label{stable-homeo-classification}

The topological classification runs along similar lines to the smooth classification.  First we need to identify the possible normal $1$-types of closed topological $4$-manifolds with fundamental group $\pi$ and then calculate their respective automorphism and bordism groups, together with the action of the automorphisms on the bordism group.

\smallskip \begin{prop}\label{prop:top-1-types}
Let $M$ be a closed oriented topological $4$-manifold with fundamental group~$\pi$.
\smallskip
\begin{enumerate}
\item[(1)] If $M$ is totally non-spin, then its normal $1$-type is given by
\[\xymatrix{B\pi\times BSTOP \ar[r] & BSTOP}\]
where the map projects onto the second factor.
\item[(2)] If $M$ is spin, then its normal $1$-type is given by
\[\xymatrix{B\pi \times BTOPSpin \ar[r] & BSTOP}\]
where the map is given by projecting $BTOPSpin$ to $BSTOP$.
\item[(3)] If $M$ is almost spin and $H^3(B\pi;\Z)$ is $2$-torsion free, then its normal
$1$-type is given by
\[\xymatrix{B\pi \times BTOPSpin \ar[r]^-{E\times p} & BSTOP\times BSTOP \ar[r]^-\oplus & BSTOP}\]
Here again $\oplus$ refers to the $H$-space structure on $BSTOP$ that corresponds to the Whitney sum of $TOP$-bundles and $E$ is again a complex line bundle with $c^*(w_2(E)) = w_2(M)$.
\end{enumerate}
\end{prop}

\begin{proof}
Mainly all the arguments of the smooth case go through in the topological case. The following points need to be observed.
\smallskip
\begin{enumerate}
\item[(1)] We have $\pi_2(BSTOP) \cong \Z/2$, and the map $M \to BSTOP$ that classifies the normal bundle induces a surjection on $\pi_2$, since $M$ is assumed to be totally non-spin. This is because $H^2(BSTOP;\Z/2) = \Z/2\langle w_2\rangle$ and $w_2$ detects the non-zero element in $\pi_2(BSTOP)$, as in the smooth case.
\item[(2)] The classifying space $BTOPSpin$ is $2$-connected, hence in the latter two cases the $1$-smoothing of $M$ is automatically surjective on $\pi_2$.
\item[(3)] The proof of the existence of the bundle $E$ is the same as in the smooth case; we just consider the complex line bundle as a $TOP$-bundle. 
\end{enumerate}
\end{proof}

We can therefore compute the bordism groups relevant for the stable homeomorphism classification.  This will prove \cref{thm:top-classification} in the totally non-spin and spin cases.  We will deal with the almost spin case in the next section.

\smallskip \begin{prop}\label{top-bordism-groups}
Let $\pi$ be a COAT group. The bordism groups $\Omega_4(\xi)$ are given as follows.
\smallskip
\begin{enumerate}[(1)]
\item\label{item:prop-4.2-1} For totally non-spin, $\Omega_4^{STOP}(B\pi) \cong \Z\oplus \Z/2$, where the $\Z$ factor is given by the signature and the $\Z/2$ factor is given by the Kirby-Siebenmann invariant.
\item\label{item:prop-4.2-2} For spin, $\Omega_4^{TOPSpin}(B\pi) \cong 8\cdot\Z \oplus H_2(B\pi;\Z/2)\oplus H_3(B\pi;\Z/2)$.  The Kirby-Siebenmann invariant is given by the signature divided by $8$.
\item\label{item:prop-4.2-3} For almost spin, $\Omega_4(\xi) \cong 8\cdot\Z\oplus H_2(B\pi;\Z/2)$. The Kirby-Siebenmann invariant is given by the signature divided by $8$ plus evaluation of $w$ on the element of $H_2(B\pi;\Z/2)$.
\end{enumerate}
\end{prop}

\begin{proof}
The James spectral sequence also exists in the topological case. The relevant bordism theories are no longer
$\Omega^{SO}$ and $\Omega^{Spin}$, but $\Omega^{STOP}$ and $\Omega^{TOPSpin}$ respectively. We have that
\[ \Omega^{STOP}_i \cong \begin{cases}
    \Z & i=0 \\
    0 & i \in \{1,2,3\} \\
    \Z\oplus \Z/2 & i=4
    \end{cases} \]
and the $\Z \oplus \Z/2$ in degree $4$ is given by the signature and the Kirby-Siebenmann invariant.
Furthermore we have
\[ \Omega_i^{Spin} \cong \Omega_i^{TOPSpin} \text{ for } i < 4 \]
and the forgetful map $16\cdot \Z \cong \Omega_4^{Spin} \to \Omega_4^{TOPSpin} \cong 8\cdot \Z$ is the canonical inclusion.
The Kirby-Siebenmann invariant does not enter as a separate $\Z/2$ summand in $\Omega_4^{TOPSpin}$, as Kirby and Siebenmann~\cite[p.~325,~Theorem~13.1]{KirbySiebenmann} have proven the formula
\[ ks(M) = \tfrac{\mathrm{sign}(M)}{8} \;\mathrm{mod}\; 2. \]

Since the signature is always divisible by $8$ in the smooth case, in the topological case the signature is still divisible by $8$ by \cite[Main~Theorem~(9)]{teichnersignatures}.
Therefore, the signature provides a splitting of the extension
\[0 \to 8\cdot \Z \to \Omega_4(\xi) \to H_2(B\pi;\Z/2) \to 0\]
that occurs in the spectral sequence for the topological almost spin case. To see this note that the map $8\cdot \Z \to \Omega_4(\xi)$ sends $8\cdot m$ to the bordism class of $m$ copies of $E_8$.


The statement about the Kirby-Siebenmann invariant in (\ref{item:prop-4.2-2}) is Rochlin's theorem and in (\ref{item:prop-4.2-3}) it follows from \cite[Theorem 6.11]{HKT}. Note that this theorem also holds if the intersection form $\lambda_M$ is not even.
\end{proof}

The stable homeomorphism classifications of 4-manifolds with COAT fundamental group differ in the totally non-spin and spin cases from the smooth case as follows.

\smallskip \begin{enumerate}
\item  In the totally non-spin case, the topological classification is altered from the smooth classification by the introduction of the $\Z/2$ Kirby-Siebenmann invariant.
\item In the spin case, the signature can be any multiple of 8 in the topological case, instead of a multiple of 16 in the smooth case.  The rest of the classification is unchanged.  In particular the material of \cref{sec:exmanifolds,sec:tau} is independent of categories.
\end{enumerate}
\smallskip

The almost spin classification, involving the action of the automorphisms $\Aut(\xi)$ on the bordism group $\Omega_4(\xi)$, will be considered, in both the smooth and topological cases, in the next section.

\section{The almost spin classification}\label{sec:almost-spin-classificiation}

Recall that we have short exact sequences, in both the smooth case
\[0 \to 16\cdot \Z \to \Omega_4(\xi) \to H_2(B\pi;\Z/2) \to 0\]
and in the topological case
\[0 \to 8\cdot \Z \to \Omega_4(\xi) \to H_2(B\pi;\Z/2) \to 0.\]
In both case we have the exact sequence
\[0 \to H^1(B\pi;\bbZ/2) \to \Aut(\xi)\to \Out(\pi)_w \to 1.\]
Moreover, in the previous section we saw that in the topological case
\[\Omega_4^{TOP}(\xi) \cong 8\cdot\Z\oplus H_2(B\pi;\Z/2),\]
whereas in the smooth case the sequence does not split.
We look at the topological case first, since this will be easier.

\smallskip
\begin{thm}
Let $\pi$ be a COAT group and let $\xi$ be as in \cref{lemma:almostspin}, an almost spin normal $1$-type.
The action of $\Aut(\xi)$ on $\Omega_4^{TOP}(\xi)$ is given as follows.
\smallskip
\begin{enumerate}
\item The action of $H^1(B\pi;\bbZ/2)$ on $\Omega_4^{TOP}(\xi)$ is trivial, so the action factors through the map $\Aut(\xi) \to \Out(\pi)_w$.
\item An element $\rho$ in
the subgroup $\Out(\pi)_w$ of the outer automorphisms acts on $(z,\phi)\in 8\cdot\Z \oplus H_2(B\pi;\Z/2)$ by
\[\rho \cdot (z,\phi) \mapsto (z,\rho\cdot \phi),\]
where $\Out(\pi)_w$ acts by functoriality on $H_2(B\pi;\Z/2)$
\end{enumerate}
\end{thm}

\begin{proof}
First we prove that the action of $H^1(B\pi;\bbZ/2)$ is trivial. Recall from the James spectral sequence, that every class $[M\xrightarrow{c} B\pi] \in \Omega_4(\xi)$ is represented by a map $c$ which factors through the $2$-skeleton of $B\pi$. First assume that $M$ is smooth.  By \cref{lem:arf}, the preimage in $H_2(B\pi^{(2)};\bbZ\pi)$ is given by $\sum_i\mu(F_i)[e_i]$, where $e_i$ ranges over the $2$-cells of $B\pi$, $F_i$ is a regular preimage of the midpoint of $e_i$ and $\mu(F_i) = \Arf(F_i)$ denotes the class of $F_i$ in $\Omega_2^{Spin}$. The action of $x\in H^1(B\pi;\bbZ/2)$ on $\mu(F_i)$ is given by pulling the element $x$ back to $H^1(F_i;\bbZ/2)$ using $F_i \to M \xrightarrow{c} B\pi$, and changing the spin structure with the resulting element of $H^1(F_i;\Z/2)$. But since the map $F_i\to B\pi^{(2)}$ factors through a point, $x$ pulls back to $0\in H^1(F_i;\bbZ\pi)$.  Therefore the action of $H^1(B\pi;\bbZ/2)$ on $[M\xrightarrow{c} B\pi]$ is trivial.

The bordism class represented by the $E_8$ manifold is also invariant under the action of $H^1(B\pi;\bbZ/2)$ since the map $E_8\to B\pi$ is null-homotopic. Every element in the topological bordism group can be represented by a smooth manifold or a smooth manifold connect summed with $E_8$, therefore the action of $H^1(B\pi;\bbZ/2)$ is trivial.


It now follows that the action of $\Aut(\xi)$ on $H_2(B\pi;\Z/2)$ factors through the map $\Aut(\xi) \to \Out(\pi)_w$.
Since the entry in the $8\cdot \Z$-summand can be changed by connected sums with the $E_8$ manifold together with the trivial map to $B\pi$, it follows that the action of $\Out(\pi)_w$ on the $8\cdot \Z$-summand is trivial.


We now compute the action of $\rho\in \Out(\pi)_w$ on the $H_2(B\pi;\Z/2)$ summand.
Taking connected sum with $E_8$ if necessary, we can again assume that $M$ is smooth. As above, the entry in the $H_2(B\pi;\Z/2)$ summand is given by Arf invariants of point preimages. The action of $\rho$ on $c\colon M\to B\pi$ only permutes these preimages. Thus the action of $\Out(\pi)_w$ is the canonical action of $\Out(\pi)$ on $H_2(B\pi;\Z/2)$.

\end{proof}

We have proved the following corollary, which is \cref{thm:top-classification} (\ref{item:A3}).

\smallskip \begin{cor}
The stable homeomorphism classes of almost spin 4-manifolds with COAT fundamental group $\pi$ are in one to one correspondence with
$$8\cdot \Z \times \big(H_2(B\pi;\Z/2)/\Out(\pi)_w\big).$$
The $8\cdot \Z$ is detected by the signature and the second part is detected by $\Arf$ invariants computed using \cref{lem:arf}.
\end{cor}
\smallskip


Now we turn to the stable diffeomorphism classification of almost spin manifolds with COAT fundamental group.  We describe the set of stable diffeomorphism classes as the kernel of the Kirby-Siebenmann invariant.

\smallskip \begin{cor}
The stable diffeomorphism classes of almost spin 4-manifolds with COAT fundamental group $\pi$ are in one to one correspondence with
\begin{align*}
\ker\big(KS \colon 8\cdot\Z \times \big(H_2(B\pi;\Z/2)/\Out(\pi)_w\big) &\to \Z/2 \big) \\
(n,\phi) & \mapsto \frac{n}{8} + w(\phi),\end{align*}
The $8\cdot \Z$ is detected by the signature and the second part is detected by $\Arf$ invariants computed using \cref{lem:arf}.
\end{cor}
\smallskip

\section{Some Examples}\label{sec:some-examples}

In this section we calculate the stable classification for the class of $3$-manifold groups $\pi$ arising as a central extension
\[\xymatrix{1 \ar[r] & \Z \ar[r] & \pi \ar[r] & \Z^2 \ar[r] & 1. }\]
Such extensions are classified by an element $z \in H^2(\Z^2;\Z) \cong \Z$. Geometrically these arise as the fundamental groups of the total spaces of the principal $S^1$-bundles over $T^2$ with first Chern class $z \in H^2(T^2;\Z)$. It follows from the long exact sequence in homotopy groups that these total spaces are aspherical, since $S^1$ and $T^2$ are aspherical.  In particular the groups we consider are aspherical $3$-manifold groups.

\smallskip \begin{lemma}\label{lem:center}
If $z \neq 0$ then we have that $\Z = Z(\pi)$, the centre of $\pi$. In particular, every automorphism of $\pi$ descends to an automorphism of $\Z^2$.  This defines a map $(\widehat{-})\colon \Aut(\pi) \to \GL_2(\Z)$.
\end{lemma}

\begin{proof}
This follows from the fact that $\pi$ has the following presentation
\[\mathcal{P} = \langle a,x,y \;\vline\;\; xax^{-1}a^{-1}, yay^{-1}a^{-1}, xyx^{-1}y^{-1}a^{-z} \rangle.\]
\end{proof}

\smallskip \begin{lemma}
The map $\Aut(\pi) \to \GL_2(\Z)$ defined by \cref{lem:center} is surjective.
\end{lemma}

\begin{proof}
We claim that we can lift elements of $\SL_2(\Z)$ to $\Aut(\pi)$ and that there exists an automorphism of $\pi$ that is sent to the matrix $A = \begin{pmatrix}0 & 1 \\ 1 & 0\end{pmatrix}$. Since any element $\varphi \in \GL_2(\Z)$ has the property that either $\varphi$ or $A\cdot \varphi$ is in $\SL_2(\Z)$, the lemma follows once we establish the above claims.
So let $\varphi \in \SL_2(\Z)$. Consider the following diagram
\[\xymatrix{1 \ar[r] & \Z \ar[d]_{=} \ar[r] & \pi^\prime \ar[d]^{\psi}_\cong \ar[r] & \Z^2 \ar[d]^\varphi_\cong \ar[r] & 1  \\
    1 \ar[r] & \Z \ar[r] & \pi \ar[r] & \Z^2 \ar[r] & 1}\]
where $\pi^\prime$ is by definition the pullback of $\pi$ along $\varphi$. The upper row is again a central extension with invariant $\varphi^*(z) \in H^2(\Z^2;\Z)$. Since $\varphi \in \SL_2(\Z)$ it follows that $\varphi^*(z) = z$ and hence there is an isomorphism of extensions $\Theta$ as indicated in the following diagram
\[\xymatrix{1 \ar[r] & \Z \ar[r] \ar[d]_{=} & \pi \ar[r] \ar[d]^-\Theta_\cong& \Z^2 \ar[r] \ar[d]^{=} & 1 \\
    1 \ar[r] & \Z \ar[d]_{=} \ar[r] & \pi^\prime \ar[d]^{\psi}_\cong \ar[r] & \Z^2 \ar[d]^\varphi_\cong \ar[r] & 1  \\
    1 \ar[r] & \Z \ar[r] & \pi \ar[r] & \Z^2 \ar[r] & 1}\]

By construction the composite
\[\xymatrix{\pi \ar[r]^\Theta & \pi^\prime \ar[r]^\psi & \pi}\]
is an automorphism of $\pi$ over $\varphi$.

The presentation $\pi$ given in the proof of \cref{lem:center} shows that there is a well-defined automorphism $\pi \to \pi$ given by $a \mapsto a^{-1}$, $x \mapsto y$ and $y \mapsto x$, which induces $\begin{pmatrix} 0&1\\1&0\end{pmatrix} \in \GL_2(\Z)$.
\end{proof}

\smallskip \begin{lemma}
For $z \neq 0$ the cohomology of $\pi$ is given by
\[H^n(\pi;\Z) \cong
\begin{cases}
    \Z & \text{ if } n \in \{ 0,3 \}, \\
    \Z^2 & \text{ if } n = 1, \\
    \Z^2\oplus \Z/z & \text{ if } n = 2.
\end{cases}\]
\end{lemma}

\begin{proof}
We consider the Gysin sequence associated to the fibration
\[\xymatrix{S^1 \ar[r] &  B\pi \ar[r]^-{p} & T^2}\]
which reads as
\[\xymatrix{0 \ar[r] & H^1(T^2;\Z) \ar[r]^-{p^*} & H^1(B\pi;\Z) \ar[r] & H^0(T^2;\Z) \ar[r]^-{-\cup z} & H^2(T^2;\Z) }\]
because the Euler class of the underlying oriented bundle of a complex line bundle is given by the first Chern class. In particular it follows that
\[p^*\colon H^1(B\pi;\Z) \xrightarrow{\cong} H^1(T^2;\Z) \]
is an isomorphism.
Therefore the action of $\Aut(\pi)$ on $H^1(B\pi;\Z)$ is given through the map $\Aut(\pi) \to \GL_2(\Z)$.
The sequence continues as follows
\[\xymatrix{H^0(T^2;\Z) \ar[r]^-{-\cup z} & H^2(T^2;\Z) \ar[r] & H^2(B\pi;\Z) \ar[r] & H^1(T^2;\Z) \ar[r] & 0}\]
which implies that there is a short exact sequence
\[\xymatrix{0 \ar[r] & \Z/z \ar[r] & H^2(B\pi;\Z) \ar[r] & H^1(T^2;\Z) \ar[r] & 0 }\]
which implies that
\[ H^2(B\pi;\Z) \cong \Z^2\oplus \Z/z.\]
We have already argued that there is a model for $B\pi$ which is an orientable closed $3$-manifold, hence also $H^3(B\pi;\Z) \cong \Z$
follows and the lemma is proven.
\end{proof}

\smallskip \begin{prop}
Let $\pi$ be a central extension of $\Z^2$ by $\Z$ with $0 \neq z \in H^2(\Z^2;\Z)$. Then we have that
\smallskip \begin{enumerate}
\item If $z$ is odd there are three stable diffeomorphism classes of spin manifolds with fundamental group $\pi$ and fixed signature.
\item If $z$ is even there are four stable diffeomorphism classes of spin manifolds with fundamental group $\pi$ and fixed signature.
\end{enumerate}
\end{prop}

We already saw in \cref{example:Z-hoch-drei} that if $z=0$ there are three stable diffeomorphism classes with fixed signature.

\begin{proof}
Recall (\cref{cor:stablediffeoclasses}) that we need to show that \[\wt{\Omega}_4^{Spin}(B\pi)/\left(\Out(\pi)\times H^1(\pi;\Z/2) \right)\] has three (respectively four) elements. We have that
\[\wt{\Omega}_4^{Spin}(B\pi) \cong H_2(\pi;\Z/2)\oplus H_3(\pi;\Z/2)\]
and thus
\[\wt{\Omega}_4^{Spin}(B\pi) \cong
\begin{cases}
(\Z/2)^2 \oplus \Z/2 & \text{ if } z \text{ is odd } \\
\left( (\Z/2)^2\oplus \Z/2\right) \oplus \Z/2 & \text{ if } z \text{ is even.}
\end{cases}\]

According to \cref{thm:action-spin-case}, given any two classes $x,y \in H_2(B\pi;\Z/2)$ we see that
\[ (x,1) \sim (y,1) \]
and furthermore
\[ (x,1) \not\sim (y,0).\]
Now assume that $z$ is odd.
To show that there are exactly three orbits of the action it suffices to see that $(x,0) \sim (y,0)$ if and only if $x = 0 = y$ or $x \neq 0 \neq y$. But this follows easily since
\[ H_2(B\pi;\Z/2) \cong H_2(B\pi;\Z)\otimes \Z/2 \cong (\Z/2)^2 \]
by the universal coefficient theorem, the action is given by the morphism
\[\Aut(\pi) \to \GL_2(\Z) \to \GL_2(\Z/2),\]
and the map $\Aut(\pi) \to \GL_2(\Z) \to \GL_2(\Z/2)$ is surjective.

For the case that $z$ is even we want to show that the action of $\Aut(\pi)$ on $H_2(B\pi;\Z/2)$ has exactly three orbits.
We write
\begin{align*}
  H_2(B\pi;\Z/2) \cong & H_2(B\pi;\Z)\otimes \Z/2 \oplus \Tor_1^{\Z}(H_1(B\pi;\Z),\Z/2)   \\
       \cong & H_2(B\pi;\Z)\otimes \Z/2 \oplus \Z/2 \\
      & (\Z/2)^2 \oplus \Z/2
\end{align*}
and elements as pairs $(x,\rho)$.
It follows from our previous arguments that
\[ (x,0) \sim (y,0) \]
if and only if $x = 0 = y$ or $x \neq 0 \neq y$ and that $(x,1) \not\sim (y,0)$ for any choice of $x,y$ because any automorphism acts trivially on the extra $\Z/2$-factor.
It remains to show that $(x,1) \sim (y,1)$ for all $x,y \in H_2(B\pi;\bbZ)\otimes \Z/2$. For this we interpret
\[ H_2(B\pi;\Z/2) \cong \Hom(H_1(B\pi;\Z),\Z/2) \cong (\Z/2)^3.\]
where the last isomorphism sends a morphism $\varphi$ to the triple $(\varphi(u),\varphi(v),\varphi(w))$, where $u=(1,0,0)$, $v=(0,1,0)$ and $w=(0,0,1)$ under a choice of identification of $H_1(B\pi;\Z)$ with $\Z \oplus \Z \oplus \Z/z$.
The statement that $(x,1) \sim (y,1)$ for all such $x,y$ then translates to the statement that for any two functions $\varphi, \psi : H_1(B\pi;\Z) \cong \Z^2\oplus \Z/z \to \Z/2$ with $\varphi(w) = 1 = \psi(w)$, there exists an automorphism $\Theta\colon \pi \to \pi$ such that $\varphi = \psi \circ \Theta$. This automorphism is defined as follows.  First, define $\Theta(w) = w$. Next, if $\varphi(u) = \psi(u)$, define $\Theta(u) = u$, and similarly for $\varphi(v) = \psi(v)$. Finally if $\varphi(u) \neq \psi(u)$, define $\Theta(u) = wu$. We obtain
\[ \psi(\Theta(u)) = \psi(wu) = \psi(w)+\psi(u) = 1 + \psi(u) = \varphi(u).\]
Again from the presentation of \cref{lem:center}, it follows that $\Theta$ is a well-defined automorphism of $\pi$. This concludes the proof of the proposition.
\end{proof}

\section{Parity of equivariant intersection forms}
\label{sec:exmanifolds}

Now we move on to giving the proof of \cref{thm:B}.  \cref{sec:exmanifolds} proves part~(\ref{item:B2}) of that theorem and \cref{sec:tau} proves part~(\ref{item:B3}).

In this section, as before, $X$ denotes a closed, oriented, aspherical $3$-manifold and~$\pi$ denotes its fundamental group. We want to construct representatives for all the stable diffeomorphism classes of spin $4$-manifolds with fundamental group $\pi$ and zero signature, and compute their   intersection forms.  To realise nonzero signatures just take connected sums with the $K3$ surface, whose spin bordism class generates $\Omega_4^{Spin}$.

The purpose of performing such detailed computation with models for each stable diffeomorphism class is to prove that the last $\Z/2$ summand of
 $\Omega_4^{Spin}(B\pi) \cong \Z \oplus H_2(B\pi;\Z/2) \oplus \Z/2$ is determined by the parity of the   intersection form on $\pi_2$; see~\cref{sec:parity}.  In the stable diffeomorphism classification of \cref{cor:stable-class-spin}, this $\Z/2$ corresponds to the extra $\{odd\}$.
The model $4$-manifolds will also be used in \cref{sec:tau}.

\subsection{Algebra of even forms}

We consider the group ring $\Z\pi$ as a ring with involution, where the involution is given on group elements by $g \mapsto \ol{g} := g^{-1}$.  For a left $\Z\pi$-module $N$ define $N^* := \Hom_{\Z\pi}(N,\Z\pi)$.  We consider $N^*$ as a left $\Z\pi$-module via the involution: $(a \cdot f)(n):= f(n)\cdot \ol{a}$.

There is an involution on $\Hom_{\Z\pi}(N,N^*)$ which sends a map $f$ to its {\em adjoint} $f^*$. By definition, this is the dual of $f$, a map $N^{**} \to N^*$, precomposed with the $\Z\pi$-module homomorphism $e \colon N \to N^{**}, n \mapsto (f \mapsto \ol{f(n)})$.

A map $f \colon N\to N^*$ gives a pairing $\lambda \colon  N \otimes N \to \Z\pi$ via $\lambda(m,n):=f(n)(m)$. This slightly awkward assignment has the property that   $f$ is $\Z\pi$-linear if and only if $\lambda$ satisfies the usual sesquilinearity conditions
\[
\lambda(a\cdot m, n ) = a\cdot \lambda(m,n)  \quad \text{ and } \quad \lambda(m, a\cdot n ) = \lambda(m,n) \cdot \bar a.
\]
One can also check that $f^*$ leads to the form $\lambda^*(m,n) = \ol{\lambda(n,m)}$. In particular, the condition $f=f^*$ translates into
\[
\lambda(m,n) = \ol{\lambda(n,m)}
\]
In the future, we shall not distinguish between $f$ and its associated form $\lambda$ and we will call $\lambda$ {\em hermitian} if it satisfies the last condition.
\smallskip

\begin{definition}
Let $N$ be a left $\Z\pi$-module.  A hermitian form $\lambda \in \Hom_{\Z\pi}(N,N^*)$ is \emph{even} if there exists $q \in \Hom_{\Z\pi}(N,N^*)$ such that $\lambda = q+q^*$.
 If $\lambda$ is not even, we sometimes also call it \emph{odd}.  This dichotomy is the \emph{parity} of $\lambda$. The parity of a $4$-manifold $M$ is the parity of its intersection form $\lambda \colon \pi_2(M) \times \pi_2(M) \to \Z\pi$.
\end{definition}
\smallskip

\begin{lemma}\label{lem:parity-stable-diff-invariant}
  The parity of a $4$-manifold is a stable homotopy invariant.
\end{lemma}
\smallskip

\begin{proof}
The parity of the intersection forms of homotopy equivalent $4$-manifolds are the same. Thus it suffices to show that the parities of $M$ and $M\# (S^2\times S^2)$ agree.

We remark that the direct sum of two forms is even if and only if both forms are individually even. Moreover we have that
\[\lambda_{M\# (S^2\times S^2)} \cong \lambda_M \oplus \left({\bbZ\pi} \otimes_\bbZ \lambda_{S^2\times S^2}\right)\]
Since $\lambda_{S^2\times S^2}$ is hyperbolic and thus even, the lemma follows.
\end{proof}

\cref{lem:parity-stable-diff-invariant} immediately implies that parity is a stable diffeomorphism invariant.

\smallskip

\begin{definition}[(Quadratic refinement)~{\cite[Theorem~5.2]{Wall}}]\label{defn:quadratic-refinement}
  A quadratic refinement of a sesquilinear hermitian form $\lambda \colon N \times N \to \Z\pi$ on a left $\Z\pi$-module~$N$ is a group homomorphism
  $\mu \colon N \to \Z\pi/\{g-\ol{g}\}$ such that
  \smallskip \begin{enumerate}[(i)]
 \item $\lambda(x,x) = \mu(x) + \ol{\mu(x)}$ for all $x \in N$.
 \item $\mu(x+y) = \mu(x) + \mu(y) + \lambda(x,y) \in \Z\pi/\{g-\ol{g}\}$ for all $x,y \in N$.
\item $\mu(ax) = a\mu(x)\ol{a}$ for all $x \in N$ and for all $a \in \Z\pi$.
  \end{enumerate}
A {\em quadratic form} is a triple $(N,\lambda,\mu)$ as above. It is called {\em even} if the underlying hermitian form~$\lambda$ is even, i.e.\ if there exists a $q \in \Hom_{\Z\pi}(N,N^*)$ such that $\lambda=q+q^*$.
\end{definition}
\smallskip

Note that, since we are working in the oriented case, that is with the involution on $\Z\pi$ given by $\overline{g}=g^{-1}$, for a quadratic form $(N,\lambda,\mu)$, the quadratic refinement~$\mu$ is uniquely determined by the hermitian form~$\lambda$.

The existence of a quadratic refinement is a necessary condition for a hermitian form to be even. More precisely, if $\lambda=q+q^*$ then $\mu(x):= q(x,x)$ has all properties above. We will see that the converse is not true, even for intersection forms of spin 4-manifolds with COAT fundamental groups. The first such examples were given in the last author's PhD thesis \cite{teichnerthesis} for 4-manifolds with quaternion fundamental groups.

Note that the intersection form on $\pi_2(M)$ of an (almost) spin $4$-manifold $M$ admits a quadratic refinement, as follows.  Represent a class in $\pi_2(M)$ by an immersed sphere, and add cusps to arrange that the normal bundle is trivial, and then count self intersections with sign and $\pi_1(M)$ elements, as in Wall \cite[Chapter~5]{Wall}. Adding a local cusp changes the Euler number of the normal bundle of an immersed 2-sphere by $\pm 2$. We use the (almost) spin condition, which implies that the Euler numbers of the normal bundles of all immersed 2-spheres are even, to guarantee that all Euler numbers can be killed by cusps, and hence the normal bundles can be made trivial.

\smallskip
\begin{lemma}\label{lem:quadratic=even-in-free} A hermitian form on a free $\Z\pi$-module $F$ has a quadratic refinement if and only if it is even. Moreover, a quadratic form $(\lambda,\mu)$ on $N\oplus F$ is even if and only if the restriction of $\lambda$ to $N$ is even.
\end{lemma}
\smallskip

\begin{proof}
First we show that every quadratic form $(\lambda,\mu)$ on $F$ is even. Let $f_i$ be a basis of $F$ and $\mu_i \in\bbZ\pi$ be a lift of $\mu(f_i)$. We define \[
q(f_i,f_i):=\mu_i, \quad q(f_i,f_j):= \lambda(f_i,f_j)\text{ for }  i<j \text{ and }  q(f_i,f_j):= 0 \text{ for } i>j
\]
and extend linearly to get $q: F\to F^*$. Then one simply checks the relation $\lambda = q + q^*$ on the generators $f_i$.
As remarked above every even form has a quadratic refinement defined by $\mu(x):= q(x,x)$, so we have proven the first sentence of the lemma.

Now let a quadratic form $(\lambda,\mu)$ on $N\oplus F$ be given and set
\[
q((m,a),(n,b)):=\lambda((m,0),(0,b))
\]
We see that
\begin{align*} &\lambda((m,a),(n,b)) \\ =& \lambda((m,0),(n,0))+\lambda((0,a),(0,b))+\lambda((m,0),(0,b))+\ol{\lambda((n,0),(0,a))}\\
= &\lambda((m,0),(n,0))+\lambda((0,a),(0,b))+q((m,a),(n,b))+q^*((m,a),(n,b))\end{align*}
Since the form $(a,b)\mapsto \lambda((0,b),(0,a))$ extends via $\mu_{|F}$ to a quadratic form on $F$, it is even by the previous argument. This shows that $\lambda$ and its restriction to $N$ differ by an even form.
\end{proof}
\smallskip

\begin{lemma} \label{lem:ext}
For any group $\pi$, the boundary map
$\Ext^i_{\bbZ \pi}(I\pi,\bbZ\pi) \to H^{i+1}(\pi;\bbZ\pi)$ is an isomorphism for  $i \geq 1$.
Moreover, if $\pi$ is an infinite group with $H^{1}(\pi;\bbZ\pi)=0$ then the canonical map
\[\Hom_{\bbZ\pi}(\bbZ\pi,\bbZ\pi) \longrightarrow \Hom_{\bbZ\pi}(I\pi,\bbZ\pi)= I\pi^* \] is an isomorphism.
In particular, $I\pi^* \cong \bbZ\pi^* \cong \bbZ\pi$ is a free $\bbZ\pi$-module, where the latter isomorphism takes $\varphi \mapsto \varphi(1)$.
\end{lemma}
\smallskip

\begin{proof}
Consider the canonical short exact sequence
\[\xymatrix{0 \ar[r] & I\pi \ar[r]^-i & \Z\pi \ar[r]^-\varepsilon & \Z \ar[r] & 0}\]
where $\varepsilon\colon \Z\pi \to \Z$ denotes the augmentation. We apply the functor $\Hom_{\Z\pi}(-,\Z\pi)$ to this sequence to obtain a long exact sequence in Ext-groups.
For $i\geq 0$ we have
\[ \Ext_{\bbZ\pi}^{i+1}(\bbZ\pi,\bbZ\pi) = 0\quad\text{and}\quad\Ext^{i}_{\bbZ\pi}(\Z,\bbZ\pi) = H^{i}(\pi;\bbZ\pi)\]
by definition of group cohomology. The first part of the lemma follows.

The second part follows by the same long exact sequence of $\Ext$ groups because under our assumptions the two relevant terms around our groups vanish. Recall that $H^0(\pi;N) \cong N^\pi$ is the fixed point set of the $\pi$-action for any $\Z\pi$-module $N$. This fixed point set vanishes for free $\Z\pi$-modules if and only if $\pi$ has infinite order.
\end{proof}

\smallskip

\begin{cor}\label{End-Ipi}
If $\pi$ is an infinite group with $H^1(\pi;\Z\pi)=0$, then $\Z\pi\cong \End(I\pi)$, with the isomorphism given by sending $x \in \Z\pi$ to the endomorphism $b\mapsto bx$.
\end{cor}
\smallskip

\begin{proof}
Every endomorphism $I\pi\to I\pi$ can be extended to $I\pi\to \Z\pi$ and thus can be uniquely described by an element in $\Z\pi$ by \cref{lem:ext}.
\end{proof}

If $\pi$ is a Poincar\'{e} duality group of dimension $n \geq 2$, we observe that it is infinite and satisfies the assumption on first cohomology in \cref{lem:ext}:
\[H^1(\pi;\bbZ\pi) \cong H_{n-1}(\pi;\Z\pi) = 0.\]
\smallskip

\begin{lemma}
\label{lem:intersec}
The involution $a \mapsto \bar a$ on $\Z\pi$ is taken to $f\mapsto f^*$ under the maps
\[
\Z\pi \cong \Hom_{\Z\pi}(\Z\pi,\bbZ\pi^*) \to \Hom_{\Z\pi}(I\pi,I\pi^*)
\]
If $\pi$ is infinite and $H^{1}(\pi;\bbZ\pi)=0$ the second map is an isomorphism, so any pairing on $I\pi$ extends uniquely to a pairing on $\Z\pi$.
\end{lemma}

\begin{proof}
The isomorphism $\bbZ\pi^* \to I\pi^*$ from \cref{lem:ext} is sufficient to obtain the isomorphism claimed under the assumptions made. The compatability of the two involutions works as follows: The group in the middle consists of pairings on $\Z\pi$ and the map to the right just restricts the pairing to $I\pi$. This restriction preserves the involution $f\to f^*$. Given a pairing $\lambda$ on $\Z\pi$, the map to the left just takes the value $\lambda(1,1)\in\Z\pi$. Our claim follows from the fact that $\lambda^*(1,1) = \ol{\lambda(1,1)}$.
\end{proof}

\subsection{Surgery on \texorpdfstring{$X \times S^1$}{X times S}}
\label{subsec:ex1}

Now we proceed to construct the promised representatives for the stable diffeomorphism classes.
Let $\wt{\nu}_{X\times S^1}\colon X\times S^1\to BSpin$ be a choice of lift of $\nu_{X\times S^1}$. Then
\[\xymatrix{X\times S^1\ar[rr]^-{pr_1\times \wt\nu_{X\times S^1}} & &  X\times BSpin}\]
defines an element of $\Omega_4^{Spin}(X)$. In \cref{sec:stablediffeoclasses} we computed that there is an isomorphism
\[\Theta \colon \Omega_4^{Spin}(X) \toiso \bbZ\oplus H_2(X;\bbZ/2)\oplus\bbZ/2.\]
Given $x_0\in X$, the composition
\[\xymatrix{S^1\ar[r]^-{x_0\times \id}&X\times S^1\ar[rr]^-{pr_1\times \wt\nu_{X\times S^1}}&&X\times BSpin\ar[r]^-{pr_2}&BSpin}\]
defines an element $\sigma$ of $\Omega_1^{Spin}\cong \bbZ/2$ which by \cref{lem:arf} agrees with the image of $X \times S^1$ under $\Theta$ followed by projection onto the third factor.

\smallskip
\begin{lemma}\label{lem:sigma}
If $\sigma=0$, then $X\times S^1$ also goes to zero under $\Theta$ followed by the projection onto the second factor. If $\sigma=1$ then any element of $H_2(X;\bbZ/2)$ can be realised by different choices of the lift $\wt\nu_{X\times S^1}$.
\end{lemma}
\begin{proof}
When $\sigma=0$, the original manifold $X \times S^1$ is null bordant over $B\pi \times BSpin$, with null bordism $X \times D^2$.

When $\sigma =1$, the action of the automorphisms of $B\pi \times BSpin$ from \cref{thm:action-spin-case} enables the choice of another $1$-smoothing so that any element is realised.
\end{proof}

We can do a surgery along ${x_0}\times S^1$ to produce a manifold with fundamental group $\pi$. This will be a surgery over $X \times BSpin$, to convert $pr_1\times \wt\nu_{X\times S^1}$ to a $2$-connected map.  Since the cobordism produced as the trace of the surgery will also be over $X \times BSpin$, the element of $\Omega_4^{Spin}(X)$ is unchanged by the surgery.   Therefore we realise the elements of $\Omega_4^{Spin}(X)$ allowed by \cref{lem:sigma} by $4$-manifolds with fundamental group $\pi$.  The remaining elements i.e.\ those not realised when $\sigma=0$, will be constructed by a more complicated procedure in the next subsection.

Let $D^3\subseteq X$ denote a small ball around $x_0$.  Fix an identification of $\partial \cl (X \setminus D^3)$ with $S^2$. Then define
\[M_{\sigma}:=(\cl(X\setminus D^3)\times S^1)\cup_{f} S^2\times D^2.\]
Here $f\colon S^2\times S^1\to S^2\times S^1$ is the identity if $\sigma=0$, whereas if $\sigma=1$, define the diffeomorphism $f$ as follows.  Give $S^2$ coordinates using the standard embedding in $\R^3$ as the boundary of the unit ball, and Euler angles:
\[(\phi,\psi) \mapsto (\cos(\phi),\sin(\phi)\cos(\psi),\sin(\phi)\sin(\psi)).\]
(For a fixed point in $S^2$, there are multiple choices for $(\phi, \psi)$.  The upcoming proscription of $f$ is independent of these choices.)
Then define $f$ by
\[((\phi,\psi),e^{i\theta})\mapsto ((\phi,\psi+\theta),e^{i\theta}).\]
The twist in the glueing map $f$ arranges that the spin structure extends across the cobordism $X \times S^1 \times I \cup_{f} D^3 \times D^2$. The spin structure can then be restricted to the new boundary, to give a spin structure on $M_{\sigma}$. By \cref{lem:sigma}, for $\sigma=1$ every element $(0,\gamma,1)\in \Omega_4^{Spin}(X)$ with $\gamma\in H_2(X;\Z/2)$ can be realised by $M_1$ with an appropriate spin structure. If we want to consider $M_1$ not just as a smooth manifold, but as a spin manifold realising $(0,\gamma,1)$, we denote it by $M_{1,\gamma}.$

We state the computation of $\pi_2$ as a lemma so that we can refer to it in subsequent similar computations.
\smallskip

\begin{lemma}\label{lem:pi2computation}
Let $X$ be an oriented aspherical 3-manifold (with possibly non-empty boundary) and fundamental group $\pi$.  Define $M_{\sigma}$ as above, for $\sigma=0,1$.  Then $\pi_2(M_{\sigma}) \cong \Z\pi \oplus I\pi$, where $I\pi$ is the augmentation ideal of $\Z\pi$ i.e.\ the kernel of the augmentation map $\Z\pi \to \Z$.
\end{lemma}

\begin{proof}
  Let $N\cong \pi\times D^3$ denote the preimage of $D^3\subseteq X$ in $\wt X$.
By assumption, $\wt{X}$ is contractible.
We will compute $\pi_2(M_{\sigma})$ by computing $H_2(\wt M_{\sigma})$, using the  Mayer-Vietoris sequences
\[0\to H_2(\partial N\times S^1)\to H_2(N\times S^1)\oplus H_2(\cl(\wt X\setminus N)\times S^1)\to H_2(\wt X\times S^1)=0\]
and
\begin{align*}
& H_2(\partial N\times S^1)\to H_2(\pi\times S^2\times D^2)\oplus H_2(\cl(\wt{X}\setminus N)\times S^1)\to H_2(\wt M_{\sigma}) \\
 \to & H_1(\partial N\times S^1)\to 0\oplus H_1(\cl(\wt{X}\setminus N)\times S^1)\to 0.
\end{align*}
The first sequence computes the effect of removing $D^3 \times S^1$ from $X \times S^1$ and the second sequence glues in $S^2 \times D^2$ in its stead.
Since $H_2(N\times S^1)=0$, from the first sequence we see that $H_2(\cl(\wt{X}\setminus N)\times S^1)\cong H_2(\partial N\times S^1)\cong \bbZ\pi$. As a $\Z[\pi]$-module, $H_2(\cl(\wt{X}\setminus D^3)\times S^1)$ is generated by $\partial D^3\times \{0\}$.

In the second sequence the maps from $H_2(\partial N\times S^1)$ to $H_2(\pi\times S^2\times D^2)$ and $H_2(\cl(\wt{X}\setminus N)\times S^1)$ are both isomorphisms. Furthermore, we have $H_1(\partial N\times S^1)\cong\bbZ \pi$, $H_1(\cl(\wt X\setminus N)\times S^1)\cong \bbZ$ and the map between them is the augmentation map. Thus, we obtain a short exact sequence
\[0\to \bbZ\pi\to H_2(\wt{M_{\sigma}})\to I\pi\to 0,\]
where $I\pi$ is the augmentation ideal $\ker(\Z\pi \to \Z)$.
\cref{lem:ext} says that
\[ \Ext_{\bbZ\pi}^1(I\pi,\bbZ\pi) = 0\]
so this sequence splits. This proves the lemma since $\pi_2(M_\sigma) \cong H_2(\wt{M_\sigma})$ by the Hurewicz theorem.
\end{proof}

Since we will need it later, we will also geometrically construct a splitting of the short exact sequence in the above proof.
Let $g_1,\ldots,g_m$ be generators of $\pi$ and let $\{s_i^j\}_{1\leq i\leq m, j \in \{0,1\}}$ be a set of disjoint points in $\partial D^3\subseteq \cl(X\setminus D^3)$. Choose $x_0\in \partial D^3$ and for every $1\leq i\leq m, 0\leq j\leq 1$ let $\omega_i^j$ be a path in $\partial D^3$ from $x_0$ to $s_i^j$ and let $w_i$ be a path in $X\setminus D^3$ from $s_i^0$ to $s_i^1$ such that $(\omega_i^1)^{-1}\circ w_i\circ\omega_i^0$ represents $g_i\in \pi\cong \pi_1(\cl(X\setminus D^3),x_0)$. We can assume that all paths $w_i$ are disjointly embedded.
If $\sigma=0$ we can define elements in $\pi_2(M_{\sigma})$ by
\[\alpha_i:=[(\{s_i^0\}\times D^2)\cup (w_i\times S^1)\cup(\{s_i^{1}\}\times D^2)].\]
Under the boundary map $H_2(\wt{M_{\sigma}})\to H_1(\partial N\times S^1)$, the element $\alpha_i$ is mapped to $[\{s_i^0\}\times S^1]-g_i[\{s_i^1\}\times S^1]$.
For $\sigma=1$, let $s_i^j$ be given by $(\phi_i^j,\psi_i^j)$ in Euler coordinates. The image of $\{s_i^j\}\times S^1\subseteq \cl(\wt{X}\setminus N)\times S^1$ in $S^2\times D^2$ is no longer $(\phi_i^j,\psi_i^j)\times S^1$, but gets rotated around the $S^2$, by definition of $f$. Therefore, to cap off the cylinder $w_i\times S^1$, we have to construct more complicated caps. We can define elements in $\pi_2(M_{\sigma})$ by
\[\alpha_i:=[C_i^0\cup(w_i\times S^1)\cup C_i^1],\]
where $C_i^j$ is the image of the map $D^2\to S^2\times D^2$ defined by
\[te^{i\theta}\mapsto ((t\phi_i^j,\psi_i^j+\theta),te^{i\theta}).\]
Note that the image of the point $\{t=0\}$ is $($north pole of $S^2$, centre of $D^2)$.
Under the boundary map $H_2(\wt{M_{\sigma}})\to H_1(\partial N\times S^1)$, the element $\alpha_i$ is again mapped to $[\{s_i^0\}\times S^1]-g_i[\{s_i^1\}\times S^1]$.
Thus $1-g_i \mapsto \alpha_i$ defines a splitting map $I\pi \to H_2(\wt{M_{\sigma}})$ as promised.

We can also compute the  intersection form. In the case $\sigma=0$ we see that the representatives for the $\alpha_i$ are disjointly embedded and that they intersect the generator $\beta:=\partial D^3$ of the free summand transversely in $\{s_i^j\}_{j=0,1}\times \{0\}$.  We therefore have
\[\lambda(\alpha_i,\beta)=1-g_i\in\bbZ\pi.\]
 When $\sigma=1$, the terms $\lambda(\alpha_i,\beta)$ are unchanged, but the representatives of the $\alpha_i$ have additional intersections amongst each other; they intersect transversely in the midpoints of the discs $C_i^j$,
so we have
\[\lambda(\alpha_i,\alpha_\ell)=(1-g_i)(1-g^{-1}_\ell) \in\bbZ\pi.\]
Make a small perturbation of the points $s^j_i$, for $j=0,1$, and the path $w_i$ between them.  Denote the new path by $w_i'$.  This can be done so that $w_i$ and $w_i'$ are disjoint.   It follows that the homological self intersections $\lambda(\alpha_i,\alpha_i)$ are also given by the formula above with $i=\ell$.

Use the identification from \cref{lem:intersec} to write the   intersection form on $\pi_2(M_{\sigma})$ as
\[\begin{blockarray}{ccc}
&I\pi&\bbZ\pi\\
\begin{block}{c(cc)}
I\pi&\sigma&1\\
\bbZ\pi&1&0\\\end{block}\end{blockarray}.\]
In particular, the intersection between $\alpha,\beta\in I\pi$ is zero if $\sigma=0$ and $\lambda((\alpha,0),(\beta,0))=\alpha\ol{\beta}$ if $\sigma=1$.

This completes the construction of elements in the bordism group representing $(0,0,0)$ and $(0,\gamma,1)$ in $\bbZ\oplus H_2(X;\bbZ/2)\oplus\bbZ/2\cong \Omega_4^{Spin}(X)$, and the computation of their   intersection forms.

\subsection{Surgery on a connected sum of two copies of \texorpdfstring{$M_1$}{M}}
\label{subsec:ex2}

So far we have constructed elements in the bordism group representing $(0,0,0)$ and $(0,\gamma,1)$ in $$\bbZ\oplus H_2(X;\bbZ/2)\oplus\bbZ/2\cong \Omega_4^{Spin}(X).$$
In this subsection we construct elements representing the remaining signature zero elements $(0,\gamma,0)$ (as noted above the signature can be changed by repeatedly connect summing with the $K3$ surface).

In this section, for a space $Z$, we use $\wt{Z}$ to denote the universal cover. If the fundamental group is not $\pi$, but we nevertheless have a map $Z \to B\pi$, we can construct the $\pi$-cover, which we denote by $\ol{Z}$.
If $Z$ has a handle decomposition, we will denote the union of the handles of index less than or equal to $k$ by $Z^{(k)}$, and call this the $k$-skeleton of $Z$ i.e.\ we use the same notation as for cell complexes.

 Let $M_{1,\gamma}\xrightarrow{p_\gamma}X\times BSpin$ denote the manifold representing $(0,\gamma,1)$ as above. To represent the elements $(0,\gamma,0)$ we can take the connected sum $M_{1,\gamma}\#M_{1,0}$. We have $\pi_1(M_{1,\gamma}\# M_{1,0})\cong \pi*\pi$, and we can do surgeries, over $B = B\pi \times BSpin$, along curves representing $(g_i,g_i^{-1})$ to obtain a $2$-connected map to $B$ i.e.\ a manifold $P$ with fundamental group $\pi$. Here the $g_i$ again denote generators of $\pi$ as above.
Since all surgeries are over $B$, the resulting manifold $P$ represents the desired element in $\Omega_4(\xi)$.

Next we will perform this construction in detail, and compute $\pi_2(P)$ and the intersection form $\lambda \colon \pi_2(P) \times \pi_2(P) \to \Z\pi$ of the output.  In what follows we often omit the subscript from $M_{1,\gamma}$, and denote both $M_{1,\gamma}$ and $M_{1,0}$ by $M$, where the distinction is not important. Until we glue in copies of $S^2 \times D^2$, the distinction is purely in the map to $BSpin$.  Only in ensuring that the map to $BSpin$ extends over these new parts does the difference between $M_{1,\gamma}$ and $M_{1,0}$ emerge.

Choose a handle decomposition of the $3$-manifold $X$ with one $0$- and one $3$-handle, $n$ 1-handles and $n$ 2-handles.
We will construct the manifold $P$ once again, incrementally, computing $\pi_2$ carefully as we go.
We begin, however, with a digression on the chain complex of $\wt{X}$, which we will need to refer to throughout the construction.  Let $g_1,\dots,g_n$ denote generators of $\pi$ corresponding to the $1$-handles of $X$, as before.  Let $h_1,\dots,h_n$ be generators of $\pi$ corresponding to the cocores of the $2$-handles of $X$; use a path from the centre of the 3-handle to the centre of the 0-handle, so that this latter centre is the basepoint for all loops.  Let $R_1,\dots,R_n$ be relations in a presentation of $\pi$ corresponding to the handle decomposition of $X$, namely the words in the $g_i$ which describe the attaching maps of the 2-handles.

Recall that given generators $g_1,\dots,g_n$ the Fox derivative~\cite{Fox-free-calculus}
 with respect to $g_i$ is a map $\frac{\partial}{\partial g_i} = D_i \colon F_n \to \Z F_n$ which is defined by the following:
$\frac{\partial e}{\partial g_i} = 0$,
$\frac{\partial g_i}{\partial g_j} = \delta_{ij}$ and $\frac{\partial{uv}}{\partial g_i} = \frac{\partial u}{\partial g_i} + u \frac{\partial v}{\partial g_i}$. Taking the quotient this defines a map $F_n \to \Z\pi$, which extends to a map $\Z F_n \to \Z\pi$ by linearity.

The chain complex $C_* = C_*(\wt{X}) \cong C_*(X;\Z\pi)$ of $\wt{X}$ comprises free $\Z\pi$-modules
\[\xymatrix{C_3 = \Z\pi \ar[r]^-{\partial_3^{\wt{X}}} & C_2 = (\Z\pi)^n \ar[r]^-{\partial_2^{\wt{X}}} & C_1 = (\Z\pi)^n \ar[r]^-{\partial_1^{\wt{X}}} & C_0 = \Z\pi }\]
with boundary maps given by
$\partial_1^{\wt{X}} = \begin{pmatrix} g_1-1 & \cdots & g_n-1 \end{pmatrix}^T$, $(\partial_2^{\wt{X}})_{ij} = \frac{\partial R_i}{\partial g_j}$ and $\partial_3^{\wt{X}} = \begin{pmatrix} h_1-1 & \cdots & h_n-1 \end{pmatrix}$.
Here we use the convention that elements of free modules are represented as row vectors and matrices act on the right.

Let $X^{2,3}$ denote $X$ with $0$- and $1$-handles removed; $X^{2,3} = X \setminus X^{(1)}$.
Take $S^1 \times X$, and surger the $S^1$.  That is, remove $S^1 \times D^3 \subset S^1 \times X$ where the $3$-ball lies in the interior of the 3-handle of $X$, and attach $D^2 \times S^2$ with \[f((\phi,\psi),e^{i\theta})=((\phi,\psi+\theta),e^{i\theta})\]
as in the previous subsection.  The rotation in the glueing map ensures, as before, that map to $BSpin$ extends to the outcome of surgery, which is $M$.

Let $Y^{2,3} = M \setminus S^1 \times X^{(1)}$ be the result of performing this surgery on $X^{2,3}$ only.  The surgery takes place in the interior of the 3-handle of $X$.  So $$Y^{2,3} = \big(S^1 \times X^{2,3} \setminus S^1 \times D^3 \big) \cup_f S^2 \times D^2.$$
\smallskip

\begin{lemma}\label{lemma:pi-2-Y-23}
We have $\pi_2(Y^{2,3})\cong  (\bbZ F_n)^{n+1}$ and $H_2(\ol{Y^{2,3}})\cong (\bbZ\pi)^{n+1}$.
\end{lemma}

\begin{proof}
Removing the $0$- and $1$-handles from $X$ is the same as removing the $2$- and $3$-handles from the dual handle decomposition. Thus $\pi_1(X^{2,3})$ is a free group $F_n$ with generators represented by the cocores of the $2$-handles of $X$.
Since $X^{2,3}$ is aspherical, by \cref{lem:pi2computation} we have that
$\pi_2(Y^{2,3})\cong \bbZ F_n\oplus IF_n\cong (\bbZ F_n)^{n+1}$.
This proves the first part of the claim.
Recall that $\ol{Y^{2,3}}$ denotes the pullback of the covering $\wt M\to M$ to $Y^{2,3}$ i.e.\ the $\pi$-covering.
We identify $H_2(\ol{Y^{2,3}}) \cong H_2(Y^{2,3};\Z\pi)$.  We compute this using the universal coefficient spectral sequence:
\[E_2^{p,q}= \Tor_p^{\Z F_n}(H_q(Y^{2,3};\Z F_n),\Z\pi) \Rightarrow H_{p+q}(Y^{2,3};\Z\pi).\]
Here $H_1(Y^{2,3};\Z F_n) =0$ and $\Z F_n$ has homological dimension one, so that all $\Tor_q$ groups with $q \geq 2$ vanish.  Therefore
\[H_{2}(Y^{2,3};\Z\pi) \cong \Tor_0^{\Z F_n}(H_2(Y^{2,3};\Z F_n),\Z \pi) \cong \Z\pi \otimes_{\Z F_n} H_2(Y^{2,3};\Z F_n) \cong (\Z\pi)^{n+1}\]
which completes the proof of the lemma.
\end{proof}

For later use we describe and give names to generators of $H_2(\ol{Y^{2,3}}) \cong (\Z\pi)^{1+n}$.  The first $\Z\pi$ summand, arising from $\Z F_n \otimes \Z\pi$, is represented by $\Sigma_1 := \partial (\pt \times D^3)$, where $\pt \times D^3 \subset S^1 \times D^3$, the $S^1 \times D^3$ which was removed during the surgery.  Next, $IF_n \otimes \Z\pi \cong (\Z F_n)^n \otimes \Z\pi \cong (\Z\pi)^n$.  The basis element $e_i$ of $(\Z\pi)^n$ is represented by the sphere $\alpha_i$ corresponding to $1-h_i \in IF_n$; recall that $h_i$ is the generator corresponding to the cocore of the $i$th 2-handle, and $\alpha_i$ was constructed just after the proof of \cref{lem:pi2computation}.  Call these spheres $\Sigma_2,\dots,\Sigma_{n+1}$ respectively.

Write $S^1 = D^1 \cup_{S^0} D^1$, and take the product of this decomposition with $X^{(1)}$ to split $S^1 \times X^{(1)}$ into two copies of $D^1 \times X^{(1)}$.  Let
\[M^{2,3} =  Y^{2,3}\cup_{D^1 \times \partial X^{(1)}} (D^1 \times X^{(1)}) = M \setminus ((S^1 \setminus D^1) \times X^{(1)}) = M \setminus (D^1 \times X^{(1)}).\]
Let $\ol{X^{(1)}}$ denote the $\pi$-cover: the pullback of the universal cover $\wt{X}\to X$ along the inclusion $X^{(1)}\to X$. Similarly let $\ol{M^{2,3}}$ denote the pullback of $\wt{M}\to M$ along the inclusion $M^{2,3}\to M$.
\smallskip

\begin{lemma}\label{lemma:m23-y23}
We have an isomorphism $H_2(\ol{M^{2,3}}) \cong H_2(\ol{Y^{2,3}})$.
\end{lemma}

\begin{proof}
Note that $\partial\ol{X^{(1)}}$ is a connected, non-compact surface. Consider the sequence
\begin{align*} & H_2(D^1 \times \partial(\ol{X^{(1)}}))=0 \to H_2(\ol{Y^{2,3}})\oplus 0\to H_2(\ol{M^{2,3}}) \\
\to & H_1(D^1 \times \partial(\ol{X^{(1)}}))\to H_1(\ol{Y^{2,3}})\oplus H_1(D^1 \times \ol{X^{(1)}}).\end{align*}
The kernel of $H_1(D^1 \times \partial(\ol{X^{(1)}}))\to H_1(D^1 \times \ol{X^{(1)}})$ is generated by the cocore spheres of the 1-handles of $X$. These circles are the attaching spheres of the 2-handles in the dual handle decomposition.  Therefore we can identify this kernel with the image of $\big(\partial_2^{\wt{X}}\big)^* \colon C^1(\wt{X}) \to C^2(\wt{X})$, which is isomorphic to $C^1(\wt{X})/\im\big((\partial_1^{\wt{X}})^*\big)$.  On the other hand $H_1(\ol{X^{2,3}})$ is also given by $C^1(\wt{X})/\im\big((\partial_1^{\wt{X}})^*\big)$.
Crossing with $S^1$ yields $H_1(\ol{X^{2,3}} \times S^1) \cong H_1(\ol{X^{2,3}}) \oplus \Z$.  Then surgery on this $S^1$ to obtain $Y^{2,3}$ kills the $\Z$ summand, without changing the homology of the first summand.  Thus $H_1(\ol{Y^{2,3}}) \cong \coker\big(\big(\partial_1^{\wt{X}}\big)^*\big)$ and the map
\[ H_1(D^1 \times \partial(\ol{X^{(1)}}))\to H_1(\ol{Y^{2,3}})\]
induces an isomorphism when restricted to
\[\ker\big(H_1(D^1\times \partial (\ol{X^{(1)}})) \to H_1(\ol{X^{(1)}})\big) \to  H_1(\ol{Y^{2,3}}) \]
In particular, the map
\[H_1(D^1 \times \partial(\ol{X^{(1)}}))\to H_2(\ol{Y^{2,3}})\oplus H_1(D^1 \times \ol{X^{(1)}})\]
is injective, and so $H_2(\ol{Y^{2,3}})\toiso H_2(\ol{M^{2,3}})$.  This completes the proof of the lemma.
\end{proof}

Let $M^{0,2,3}$ denote $M\setminus (D^1 \times X^1)$, where $X^1 = X^{(1)} \setminus X^{(0)} \cong\coprod^n D^1\times D^2$ denotes the union of the $1$-handles of $X$. Note that $\pi_1(M^{0,2,3}) \cong \pi_1(M) \cong \pi$, therefore $\wt{M}^{0,2,3} = \ol{M}^{0,2,3}$.
\smallskip

\begin{lemma}\label{lem:H2-M-023}
We have that $H_2(\wt{M}^{0,2,3}) \cong \bbZ\pi \oplus I\pi$.
\end{lemma}
\begin{proof}
Let $N$ be the preimage of $X^1\times D^1\cong \coprod^n D^4$ in $\wt M$. Then from the Mayer-Vietoris sequence associate to the decomposition $\wt{M} = \wt{M}^{0,2,3} \cup N$, namely
\[H_2(\partial N)=0\to H_2(\wt{M}^{0,2,3})\oplus0\to H_2(\wt{M})\to H_1(\partial N)=0,\]
we see that $H_2(\wt{M}^{0,2,3})\toiso H_2(\wt{M})\cong \bbZ\pi \oplus I\pi$; recall that the second isomorphism was shown in \cref{lem:pi2computation}. This proves the lemma.
\end{proof}

Note that $M^{0,2,3}$ can also be obtained from $M^{2,3}$ by glueing in the product $D^1 \times X^{(0)}$ of $D^1$ with the $0$-handle of~$X$.  In fact
\[M^{0,2,3} = M \setminus (D^1 \times X^1) = M \setminus (\coprod^n D^1 \times D^3) = M^{2,3} \cup (D^1 \times X^{(0)}) = M^{2,3} \cup D^4. \]
The final glueing is performed along $\partial D^4 \setminus (\coprod^{2n} S^0 \times D^3)$, where the removed $3$-balls correspond to the feet of the~$n$ 1-handles.
Let $M'$ denote two copies of $M^{2,3}$ glued together along this same $S^3\setminus \coprod^{2n}D^3$, and let $\ol{M'}$ denote the $\pi$-covering.
\smallskip

\begin{lemma}
We have $H_2(\ol{M'}) \cong (\Z\pi)^{n+2} \oplus I\pi$ and $H_1(\ol{M'})\cong I\pi$.
\end{lemma}

\begin{proof}
From the decomposition $M^{0,2,3} = M^{2,3} \cup D^4$ we obtain a Mayer-Vietoris sequence
\[H_2(\pi \times (S^3\setminus \coprod^{2n}D^3))\to H_2(\ol{M^{2,3}})\oplus 0\to H_2(\wt{M}^{0,2,3}) \to 0.\]
 Then we have
\[H_2(\pi\times (S^3\setminus \coprod^{2n}D^3))\to \bigoplus^2 H_2(\ol{M^{2,3}})\to H_2(\ol{M'})\to 0\]
and from the above we see that $H_2(\ol{M'})\cong H_2(\ol{M^{2,3}})\oplus H_2(\wt{M}^{0,2,3})$.  This follows from the general fact that given a homomorphism $A \to B$ of modules, and the corresponding diagonal morphism $\Delta \colon A \to B \times B$, the quotient $(B \times B) / \Delta(A)$ is isomorphic to $B \times (B/A)$; the map $(b,b')\Delta(A) \mapsto (b'-b,b'A)$ is an isomorphism with inverse $(b,b'A) \mapsto (b'-b,b')\Delta(A)$.
In our case $A = H_2(\pi \times (S^3\setminus \coprod^{2n}D^3))$ and $$B= H_2(\ol{M^{2,3}}) \cong H_2(\ol{Y^{2,3}}) \cong (\Z\pi)^{n+1}$$ generated by the spheres $\Sigma_1^1,\dots,\Sigma_{n+1}^1$, as can be seen by combining \cref{lemma:pi-2-Y-23,lemma:m23-y23}, where $\Sigma_i^j$ denotes the $i$th sphere $\Sigma_i$ in the $j$th copy of $M^{2,3}$.
\cref{lem:H2-M-023} therefore implies that $$B/A = H_2(\wt{M}^{0,2,3}) \cong \Z\pi \oplus I\pi$$
generated by the diagonal elements $\Sigma_i^1 \cup \Sigma_i^2$.  Here $\Sigma_1^1 \cup \Sigma_1^2$ represents $(1,0) \in \Z\pi \oplus I\pi$, while $\Sigma_{i+1}^1 \cup \Sigma_{i+1}^2$ represents $(0,1-h_i) \in \Z\pi \oplus I\pi$ (here a union of spheres can be replaced by the connected sum if desired).  This completes the proof of the first part of the lemma.

To see the second part of the lemma we need to compute $H_1(\ol{M'})$.
Since removing $D^3\times S^1$ from a $4$-manifold does not change the fundamental group, we have $\pi_1(M')\cong \pi_1(M_{1,\gamma}\#M_{1,0})\cong \pi*\pi$.
We can compute the homology $H_1(\ol{M'})$ using the Mayer-Vietoris sequence for $\ol{M'}= \wt{M} \sm (\pi \times D^4) \cup_{\pi \times S^3} \wt{M}\sm (\pi \times  D^4)$.  This yields
\[\xymatrix @C-0.3cm @R-0.7cm{0 = \bigoplus_2 H_1(\wt{M} \sm (\pi \times D^4)) \ar[r] & H_1(\ol{M'}) \ar[r] & H_0(\pi \times S^3)\cong \Z\pi \ar[r]^-{\mathrm{aug}^2} & \\ \bigoplus_2 H_0(\wt{M} \sm (\pi \times D^4)) \cong \Z^2 & & & }\]
which easily implies the second part of the lemma.
\end{proof}

The manifold $M'$ is homeomorphic to the manifold obtained from $M\#M$ (glued together by taking out $D^1 \times X^{(0)}$ from each copy and identifying the boundaries), by removing $D^1 \times X^1$ in each copy. The cores of these solid tori removed from $M\#M$ represent elements $g_i^{-1}\cdot g_i'$ of $\pi_1(M\#M)\cong \pi \ast \pi$, where $g_i'$ is the same generator as $g_i$ in the second copy of $M$.

To obtain the closed manifold $P$ we need to glue in $n$ copies of $S^2\times D^2$ to $M'$. Using the same identification of the boundary components $S^2\times D^1\subseteq M^{2,3}$ in both copies, we have a unique identification of the $n$~boundary components of~$M'$ with $S^2\times S^1$. Use Poincar\'{e} duality to view $\gamma$ as a homomorphism $H_1(X;\Z)\to \bbZ/2$. Use a point in the $0$-handle of~$X$ as a basepoint, so that every $1$-handle defines an element in $\pi_1(X)$. If this element is mapped to zero under $\pi_1(X)\xrightarrow{h} H_1(X;\Z)\xrightarrow{\gamma}\bbZ/2$, then glue $S^2\times D^2$ to the torus corresponding to the $1$-handle via the identity on $S^2\times S^1$. Otherwise, glue using
\[f((\phi,\psi),e^{i\theta})=((\phi,\psi+\theta),e^{i\theta}).\]
This completes our reconstruction of the 4-manifold $P$, whose   intersection form we are trying to compute.
\smallskip

\begin{lemma}
We have $\pi_2(P) \cong (\Z\pi)^{2n+1} \oplus I\pi$.
\end{lemma}

\begin{proof}
On the level of $\pi$-coverings we obtain the following Mayer-Vietoris sequence
\begin{align*} & H_2(\coprod_{i=1}^n\pi\times S^2\times S^1)\to H_2(\overline{M'})\oplus H_2(\coprod_{i=1}^n\pi\times S^2\times D^2)\to H_2(\wt P) \\
\to & H_1(\coprod_{i=1}^n\pi\times S^2\times S^1)\to H_1(\overline{M'})\oplus 0\to 0.\end{align*}
Using the computations above, the Mayer-Vietoris sequences yields the following exact sequence
\[0\to(\bbZ\pi)^{n+2}\oplus I\pi\to H_2(\wt P)\xrightarrow{\delta} (\bbZ\pi)^n\to I\pi\to 0.\]
The last map sends the generator of the $j$-th summand to $1-g_j$ where $g_j\in\pi$ denotes the element represented by the $j$-th $1$-handle.

Let $\{b_i\}_{1\leq i\leq n}$ denote the cores of the $2$-handles of~$X$. Take one copy of $b_i\times \{0\}$ in each copy of $M^{2,3}$ in~$M'$. Their boundaries coincide where they lie on the $0$-handle (which we used to connect the two copies of $M^{2,3}$). Thus, we obtain a two sphere with one $D^2$ removed for every time that the boundary of $b_i$ runs over a $1$-handle. We can fill these copies of $D^2$ in using $S^2\times D^2$.
This produces an element $B_i$ in $H_2(\wt{P})$ which is sent under the boundary map $\delta$ in the Mayer-Vietoris sequence to the Fox derivatives; i.e.\ if we identify the free submodule of $H_2(\wt{P})$ generated by each of the $B_i$ with $\Z\pi \subset C_2(\wt{X})$ then $\delta| = \partial_2^{\wt{X}} \colon \Z\pi \to (\Z\pi)^n$.  Also note that we have the relation
$$[\Sigma_1^1] + [\Sigma_1^2] = \sum_{i=1}^n (p_i \circ \partial_3^{\wt{X}}(1))[B_i] \in H_2(P;\Z\pi),$$
where $p_i \colon (\Z\pi)^n \to \Z\pi$ is projection onto the $i$-th summand, since the $\Sigma^j_1$ sphere represents the boundary of the 3-handle of $X$.  The use of the surgery discs in the $B_i$ cancel homologically.

In the following diagram let $E:= I\pi \oplus (\Z\pi)^{n+1}$ so that $H_2(\overline{M'}) \cong E \oplus \Z\pi$, with the $\Z\pi$ summand which has been separated out generated by $\Sigma_1^1 \cup \Sigma_1^2$.  The bottom row is the exact sequence computed from the Mayer-Vietoris sequence above.
\[\xymatrix @C+0.3cm{0 \ar[r] \ar[d]^-{=} & E \oplus \Z\pi \ar[r]^-{\begin{pmatrix} \Id & 0 \\ 0 & \partial_3^{\wt{X}} \end{pmatrix}} \ar[d]^-{=} & E \oplus (\Z\pi)^n \ar[r]^-{\begin{pmatrix} 0 & \partial_2^{\wt{X}}\end{pmatrix}} \ar[d] & (\Z\pi)^n \ar[r]^-{\partial_1^{\wt{X}}} \ar[d]^-{=} & I\pi \ar[r] \ar[d]^-{=} & 0  \\
0 \ar[r] & E \oplus \Z\pi \ar[r]  & H_2(\wt{P}) \ar[r]  & (\Z\pi)^n \ar[r] & I\pi \ar[r]  & 0
}\]
The central vertical map is defined as follows.  Send $E$ to $H_2(\wt{P})$ using the same map as in the bottom row.  Send the $i$-th basis vector $e_i \in (\Z\pi)^n$ to the class $[B_i]$.  We can see from the description of the $B_i$ above that the diagram commutes.  The five lemma then implies that the middle vertical map is an isomorphism, so that $\pi_2(P) \cong H_2(\wt{P}) \cong I\pi \oplus (\bbZ\pi)^{2n+1}$ as claimed.
\end{proof}

Next we describe the generators of $$\pi_2(P) \cong \Z\pi \oplus I\pi \oplus (\Z\pi)^n \oplus (\Z\pi)^n$$ explicitly.
The first $\Z\pi$ summand is represented by $\Sigma_1^1$.  In the $I\pi$ summand, $g_i-1$ is represented by $\Sigma_{i+1}^1 \cup \Sigma_{i+1}^2$.  The first $(\Z\pi)^n$ summand has $e_i$ represented by $\Sigma^2_{i+1}$.  Here we choose $\Sigma^2_{i+1}$ instead of $\Sigma^1_{i+1}$, as can be achieved by a basis change, in order to obtain a simpler matrix representing the intersection form in the next lemma.  The last $(\Z\pi)^n$ summand is generated by the $B_i$ spheres.  In each instance of $i$ we have $i=1,\dots,n$.

\smallskip

\begin{lemma}\label{lemma:int-form-on-P}
  The intersection form on $\pi_2(P)$ is given as follows:
\[\begin{blockarray}{ccccc}
&\bbZ\pi&I\pi&(\bbZ\pi)^n&(\bbZ\pi)^n\\
\begin{block}{c(cccc)}
\bbZ\pi&0&1&0&0\\
I\pi&1&2&(1-g_j^{-1})& 0\\
(\bbZ\pi)^n&0&(1-g_i)&(1-g_i)(1-g_j^{-1})& \delta_{ij}\\
(\bbZ\pi)^n&0& 0 & \delta_{ij} &\sum_k\gamma(g_k)(D_kR_i)T(D_kR_j)
\\\end{block}\end{blockarray}\]
where the precise meaning of the entries is explained in the next two paragraphs.
\end{lemma}

The entries are interpreted as follows.  In the bottom right $2 \times 2$ block, each entry represents an $n \times n$ matrix over $\Z\pi$, and we have written the $(i,j)$ entry. The Kronecker delta symbols $\delta_{ij}$ correspond to $n \times n$ identity matrices.  Recall that $g_1,\dots,g_n$ are our chosen generators of $\pi$.
In the bottom right entry, recall that $\gamma \in \Hom(\pi,\Z/2)$, and then for each $g_k$ consider $\gamma(g_k)$ as an element of $\Z$ via the natural inclusion $\Z/2 = \{0,1\} \subset \Z$.

The $(2,3)$ and $(3,2)$ entries are row and column vectors respectively, and we have written the $j$th, $i$th entries.
Intersections involving the $I\pi$ summand should be interpreted via the inclusion $I\pi \subset \Z\pi$ and the identification of \cref{lem:intersec}.
For example for $\zeta\cdot e_j$, where $e_j$ is the $j$th basis vector in the first $\Z\pi^n$, $\zeta \in \Z\pi$, and for $\beta\in I\pi$, we have $\lambda(\beta,\zeta\cdot e_j)=\beta(1-g_j^{-1})\ol{\zeta}$.

\begin{proof}[of Lemma~\ref{lemma:int-form-on-P}]

The intersections between the $\Sigma_i^j$ spheres have already been computed in the previous subsection.

The spheres $\Sigma_i^j$ use the cocore of the $i$-th 2-handle of $X$ in what was the $j$-th copy of $M_1$ of the connected sum, before the final set of surgeries.  This cocore of course intersects the core $b_i$ of the $i$-th 2-handle, which forms part of $B_i$.  The entries in the last row and column, excluding the bottom right entry, follow.

It remains to compute the intersection form on the free summands corresponding to the cores $b_i$ of the $2$-handles. Let $B_i,B_j$ denote the generators of the $i$-th and $j$-th summand respectively. They intersect only in the caps constructed in those copies of  $S^2\times D^2$ which were attached using the non-trivial glueing i.e.\ we have $\lambda(B_i,B_j)=\sum_k\gamma(g_k)(D_kR_i)\overline{(D_kR_j)}$; recall that $D_kR_i$ denotes the $k$-th Fox derivative of the relation corresponding to the $i$-th $2$-handle.
\end{proof}

\subsection{Parity of intersection forms detects the \texorpdfstring{$\Z/2$}{Z/2} summand of \texorpdfstring{$\Omega_4^{Spin}(X)$}{the bordism group}}\label{sec:parity}

We emphasise that the parity does not depend on a choice of spin structure; like the signature it can be computed independently of the choice of normal 1-smoothing, from the intersection form.
\smallskip

\begin{thm}\label{thm:parity-detects-sec}
Let $[M\xrightarrow{c} X]\in\Omega^{Spin}_4(X)$ with $c_*\colon \pi_1(M)\toiso \pi_1(X)$ an isomorphism. Then $[M\xrightarrow{c} X]$ lies in the kernel of the projection to $H_3(X;\bbZ/2)\cong\bbZ/2$ if and only if the equivariant intersection of $M$ is even.
\end{thm}
\smallskip

\begin{remark}
  The proof of the theorem is identical for topological 4-manifolds considered up to stable homeomorphism.
\end{remark}
\smallskip

\begin{proof}
The parity of the equivariant intersection form is a stable diffeomorphism invariant by \cref{lem:parity-stable-diff-invariant}, and thus it suffices to prove the statement of the theorem for the representatives constructed above. As seen in \cref{subsec:ex1}, for every element $[M\xrightarrow{c} X]$ as in the statement that does not lie in the kernel of the projection, $M$ is stably diffeomorphic to a manifold $M_1$ with $\pi_2(M_1)\cong \bbZ\pi\oplus I\pi$ and equivariant intersection form
\[\begin{blockarray}{ccc}
&I\pi&\bbZ\pi\\
\begin{block}{c(cc)}
I\pi& 1 &1\\
\bbZ\pi&1&0\\\end{block}\end{blockarray}\]
By \cref{lem:intersec} this form is odd, since $1$ does not lie in the image of $\bbZ\pi\xrightarrow{1+T}\bbZ\pi$.

If $[M\xrightarrow{c} X]$ lies in the kernel of the projection to $\Z/2$, then by \cref{subsec:ex2}, $M$ is stably diffeomorphic to a manifold $P$ with equivariant intersection form
\[\begin{blockarray}{ccccc}
&\bbZ\pi&I\pi&(\bbZ\pi)^n&(\bbZ\pi)^n\\
\begin{block}{c(cccc)}
\bbZ\pi & 0 & 1 & 0 & 0\\
I\pi & 1 & 2 & (1-g^{-1}_j) & 0\\
(\bbZ\pi)^n & 0 & (1-g_i) & (1-g_i)(1-g_j^{-1}) & \delta_{ij}\\
(\bbZ\pi)^n & 0 & 0 & \delta_{ij} & \sum_k\gamma(g_k)(D_kR_i)\overline{(D_kR_j)}
\\\end{block}\end{blockarray}.\]
See below the statement of \cref{lemma:int-form-on-P} for an explanation of the meaning of the entries.
By \cref{lem:quadratic=even-in-free} the form is even if and only if its restriction to $I\pi$ is even, and this latter statement is evident from the $2$ in the above matrix.
\end{proof}

\section{The \texorpdfstring{$\tau$}{tau}-invariant of spin $4$-manifolds}
\label{sec:tau}

Recall that we denote the map given by augmentation composed with reduction modulo 2 by $\phi \colon \Z\pi \xrightarrow{\epsilon} \Z \to \Z/2$.
\smallskip

\begin{definition}\label{defn:spherically-characteristic}
An element $\alpha\in\pi_2(M)$ is called \emph{spherically characteristic} if $\phi(\lambda(\alpha,\beta)) = \phi(\lambda(\beta,\beta)) \in \Z/2$ for all $\beta \in \pi_2(M)$.
Note that the right hand side vanishes identically if and only if $M$ has universal covering spin.
\end{definition}
\smallskip

Let $S:S^2 \looparrowright M$ be an immersed 2-sphere with vanishing self-intersection number $\mu_M(S)=0$.  Then the self-intersection points of $S$ can be paired up so that each pair consists of two points having oppositely signed group elements.  Therefore, one can choose a Whitney disc $W_i$ for each pair of self-intersections and arrange that all the boundary arcs are disjoint.  The normal bundle to $W_i$ has a unique framing and the Whitney framing of $W_i$ differs from this framing by an integer $n_i \in \Z$.

If $S$ is spherically characteristic, then the following expression is independent of the choice of Whitney discs:
\[
\tau(S) := \sum_i |W_i \cap S| + n_i \mod{2}.
\]
Moreover, $\tau(S)$ only depends on the regular homotopy class of the immersion.  Restricting to immersions with $\mu_M(S)=0$ fixes a regular homotopy class within a homotopy class: non-regular cusp homotopies change the self-intersection number by (even) multiples of the identity group element.

\smallskip

\begin{remark}
If $S$ is not spherically characteristic then $\tau(S)$ is not well-defined since adding a sphere that intersects $S$ in an odd number of points to one of the Whitney discs would change the sum by one.
\end{remark}
\smallskip

The invariant $\tau(S)$ first appeared in R.~Kirby and M.~Freedman~\cite[p.~93]{Kirby-Freedman} and Y.~Matsumoto~\cite{Matsumoto} and a similiar invariant was later used by M. Freedman and F.~Quinn~\cite[Definition~10.8]{Freedman-Quinn}.
In \cite{schneiderman-teichner-tau}, Schneiderman and the fourth author defined a generalisation $\tau_1(S)$ with values in a quotient of $\Z[\pi\times\pi]$. They considered primary and secondary group elements, in analogy with the passage from the ordinary to the equivariant intersection form.

\smallskip

\begin{lemma}
\label{lem:augtau}
Let $x,y\in\pi_2(M)$ be such that $\lambda(x,y)=0$, $\mu(x)$ and $\mu(y)$ are trivial and $x$ is spherically characteristic. Then for every element $\kappa\in \ker(\phi \colon \bbZ\pi\to \bbZ/2)$, we have $\tau(x)=\tau(x+\kappa y)\in\bbZ/2$.
\end{lemma}
\begin{proof}
First let $\kappa=1\pm g$. Choose immersed representatives for $x$ and $y$ and take a parallel copy of $y$ together with a loop representing $g$ as a representative for $\pm gy$. Choose framed Whitney discs with disjoint boundary arcs for the self-intersections of $x$, those of $y$ and the intersections between $x$ and $y$. We denote these by $W_{x,x}$, $W_{x,y}$ and $W_{y,y}$ respectively.

Whitney discs $W_{x,gy}$ for the intersection between $x$ and $\pm gy$ can be obtained by taking a parallel copy of each Whitney disc $W_{x,y}$ for the intersections between $x$ and $y$.

Take three parallel copies of the Whitney discs for the self-intersections of $y$ to produce Whitney discs $W_{gy,gy}$, $W_{y,gy}$ and $W_{gy,y}$  for the self-intersections of $\pm gy$ and for the intersections between $y$ and $\pm gy$.

Whenever a Whitney disc intersects $y$, it also intersects $\pm gy$, and therefore the total contribution to $\tau$ vanishes modulo two.  See \cref{figure:Whitney-discs}.

\begin{figure}[ht]
\begin{center}
\begin{tikzpicture}
\node[anchor=south west,inner sep=0] at (0,0){\includegraphics[scale=0.22]{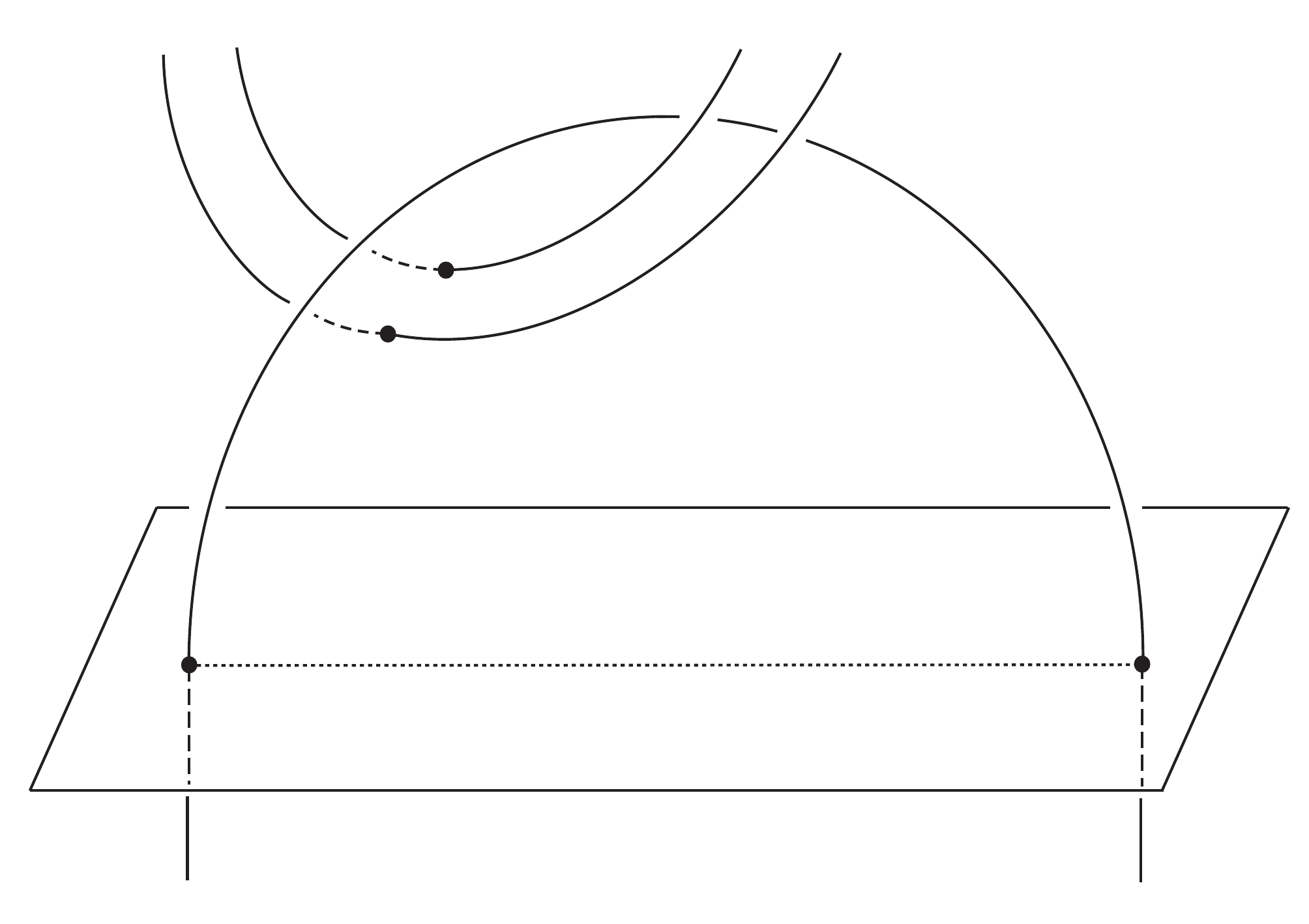}};
\node at (0.7,4.5)  {$y$};
\node at (1.8,4.5)  {$gy$};
\end{tikzpicture}
\end{center}
     \caption{Whenever $y$ intersects a Whitney disc, so does $gy$.  Following a standard convention for diagram in dimension~$4$, some surfaces are shown as arcs, which we imagine to propagate through time.  The part of the horizontal surface that we see only lives in the present.}
    \label{figure:Whitney-discs}
\end{figure}

Thus for the computation of $\tau(x+(1\pm g)y)$ we only need to count the intersections of all the Whitney discs with $x$. For every intersection of $x$ with a Whitney disc $W_{y,y}$  or $W_{x,y}$, there are three or one further intersections respectively, from the parallel copies. Therefore, these intersections also cancel modulo~$2$. See \cref{figure:Whitney-discs-2}.

\begin{figure}[ht]
\begin{center}
\begin{tikzpicture}
\node[anchor=south west,inner sep=0] at (0,0){\includegraphics[scale=0.22]{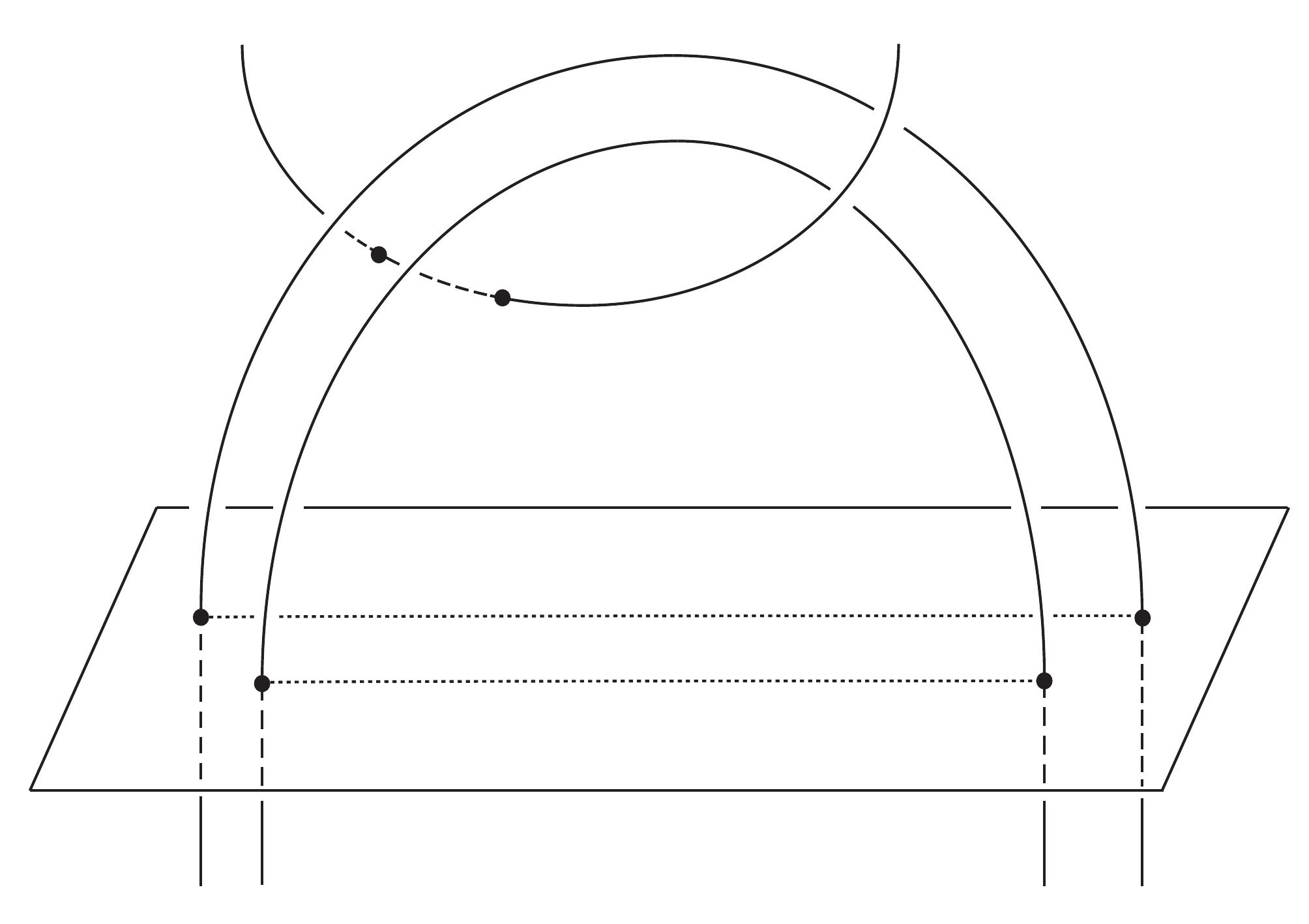}};
\node at (6.1,4.1)  {$W_{y,\ast}$};
\node at (5,3)  {$W_{gy, \ast}$};
\node at (1.1,4.8) {$x$};
\node at (1,0) {$y$};
\node at (1.5,0) {$gy$};
\node at (0.3,1.5) {$\ast$};
\end{tikzpicture}
\end{center}
     \caption{Whenever $x$ intersects a Whitney disc $W_{y,\ast}$, it also intersects the parallel copy $W_{gy,\ast}$.}
    \label{figure:Whitney-discs-2}
\end{figure}
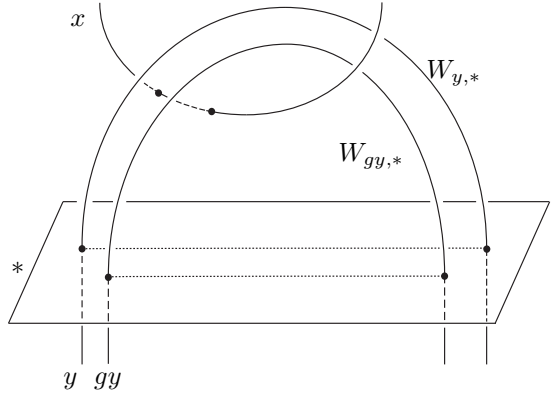

The remaining intersections are those between $x$ and the Whitney discs $W_{x,x}$ for the self-intersections of $x$. This proves the lemma for $\kappa=1\pm g$.

For the general case, observe that every $\kappa\in \ker\phi$ can be written as a sum $\kappa=\sum_{i=1}^n(1\pm g_i)h_i$, and apply the first case successively with $(1 \pm g_i)(h_iy)$ as $(1\pm g)y$, $i=1,\dots,n$.
\end{proof}

Later we will need the following lemma.
\smallskip

\begin{lemma}
\label{lem:H2kernel}
Let $Y$ be a finite $2$-dimensional CW complex with fundamental group $\pi$. Every element in the kernel of $H^2(Y;\Z\pi)\to H^2(Y;\Z/2)$ can be written as $\sum_{i=1}^n\kappa_ix_i$ with $\kappa_i\in \ker \phi$ and $x_i\in H^2(Y;\Z\pi)$.
\end{lemma}
\begin{proof}
Let $(C^*,\delta^*)$ be the $\Z\pi$-module cellular cochain complex associated to $Y$. The lemma follows from a diagram chase in the following diagram.
\[\xymatrix{
&H^2(Y;\Z\pi)\ar[r]&H^2(Y;\Z/2)\\
\ker\phi\otimes_{\Z\pi} C^2\ar[r]&C^2\ar[r]^-{\phi}\ar[u]&\Z/2\otimes_{\Z\pi} C^2\ar[u]\\
\ker\phi\otimes_{\Z\pi} C^1\ar[r]\ar[u]^{\delta^2}&C^1\ar[r]^-{\phi}\ar[u]^{\delta^2}&\Z/2\otimes_{\Z\pi} C^1\ar[u]^{\delta^2}}\]
\end{proof}
\smallskip

\begin{Conventions}
\label{conv:Sec7}From now on $\pi$ denotes a COAT group,~$X$ denotes a oriented, closed, connected aspherical $3$-manifold with fundamental group~$\pi$, $M$ denotes a spin $4$-manifold with fundamental group $\pi$ and $c \colon M \to X$ denotes a classifying map for the universal covering space of~$M$.
Assume that $M$ has signature zero and even equivariant intersection form.
\end{Conventions}
\smallskip

\begin{lemma}
\label{lem:tauiso}
There exist $n,k\in\bbN$ and an isomorphism \[\psi\colon I\pi\oplus(\Z\pi)^n\xrightarrow{\cong}\pi_2(M\#k(S^2\times S^2))\] such that $\lambda(\psi(I\pi),\psi(I\pi))=0$.
\end{lemma}

\begin{proof}
Since $M$ is stably diffeomorphic to one of the examples constructed in \cref{sec:exmanifolds}, there exist $n_1,k_1\in\bbN$ such that there is an isomorphism $\psi'\colon I\pi\oplus (\Z\pi)^{n_1}\to \pi_2(M\#k_1(S^2\times S^2))$. By \cref{lem:intersec}, the intersection form on $\psi'(I\pi)$ is determined by an element $\alpha\in\Z\pi$. Since, by the assumptions above, the intersection form of $M$ is even, there exists $p\in \Z\pi$ with $\alpha=p+\ol{p}$. Let $k:=k_1+1$, $n:=n_1+2$ and define
\[\begin{array}{rcl} \psi_1 \colon I\pi &\to&  \pi_2(M\#k(S^2\times S^2))\cong \pi_2(M\#k_1(S^2\times S^2))\oplus (\Z\pi)^2\\
 \beta &\mapsto & (\psi'(\beta),\beta p,-\beta). \end{array}\]

 It is not hard to see that the map $\psi_1$ can be extended to an isomorphism
\[\psi \colon I\pi\oplus (\Z\pi)^n \toiso \pi_2(M\#k(S^2\times S^2)).\]
For $\beta,\beta'\in I\pi$ we have
\[\lambda(\psi(\beta),\psi(\beta'))=\lambda(\psi'(\beta),\psi'(\beta'))+\beta(-p-\ol{p})\ol{\beta'}=\beta\alpha\ol{\beta'}-\beta\alpha\ol{\beta'}=0.\]
Thus $\psi$ is a map satisfying the desired properties.
\end{proof}
\smallskip

\begin{lemma}
\label{lem:spherchar}
For every isomorphism $\psi$ as in \cref{lem:tauiso}, the elements in $\psi(I\pi)$ are spherically characteristic.
\end{lemma}
\smallskip

\begin{proof}
By \cref{lem:intersec} there exist $(y,x_1,\ldots,x_n)\in\Z\pi^{n+1}$, that determine how elements of $\psi(I\pi)$ pair with all of $\pi_2(M\# k (S^2 \times S^2))$, in the following sense: for all $b,b'\in I\pi, (a_1,\ldots, a_n)\in\Z\pi^n$ we have
\[\lambda(\psi(b,0,\ldots,0),\psi(b',a_1,\ldots, a_n))=by\overline{b'}+\sum_{i=1}^nbx_i\overline{a_i}\in \Z\pi.\]
In particular, every summand contains $b$ as a factor, so the sum lies in $I\pi \subset \Z\pi$ and we have that $\phi(\lambda((b,0,\ldots,0),(b',a_1,\ldots, a_n)))=0$.
\end{proof}

\begin{lemma}
\label{lem:tauquotient}
For every isomorphism $\psi$ as in \cref{lem:tauiso}, the $\tau$ invariant defines a map $I\pi\xrightarrow{\psi}\psi(I\pi)\xrightarrow{\tau}\Z/2$.
This map factors through the map
\[
I\pi\toiso H^2(X^{(2)};\Z\pi)\to H^2(X^{(2)};\Z/2).
\]
We denote the induced map $H^2(\pi;\Z/2)\cong H^2(X;\Z/2)\xrightarrow{i^*}H^2(X^{(2)};\Z/2)\to\Z/2$ by
$\tau_\psi$.
\end{lemma}

\begin{proof}
Since the intersection form vanishes on $\psi(I\pi)$ by assumption, $\mu$ vanishes on these elements. It is not too hard to see that the self-intersection number $\mu$, for a $(+1)$-hermitian quadratic form, is determined by the intersection pairing $\lambda$.
 The elements of $\psi(I\pi)$ are spherically characteristic by \cref{lem:spherchar}, and thus the $\tau$ invariant gives a well-defined element in $\Z/2$.

An element in the kernel of $I\pi\to H^2(X^{(2)};\Z/2)$ is given by $\sum_i\kappa_i\beta_i$ with $\beta_i\in I\pi$ and $\kappa_i\in\ker\phi$ by \cref{lem:H2kernel}. For any $\beta\in I\pi$  it follows from \cref{lem:augtau} that
\[\tau(\psi(\beta+\sum_i\kappa_i\beta_i))=\tau(\psi(\beta)+\sum_i\kappa_i\psi(\beta_i))=\tau(\psi(\beta)).\]
Therefore $\tau \circ \psi$ factors through $H^2(X^{(2)};\Z/2)$ as claimed.
\end{proof}

\begin{lemma}\label{lem:tau-indep}
The map $\tau_\psi$ from \cref{lem:tauquotient} is independent of $\psi$, and is a stable diffeomorphism invariant.
\end{lemma}

\begin{proof}
For any two choices $\psi,\psi' \colon I\pi \oplus (\Z\pi)^n \to \pi_2(M \#k (S^2 \times S^2))$, the map
\[\xymatrix{f\colon I\pi\ar[r]^-{\inc} & I\pi\oplus\Z\pi^n\ar[r]^{\psi^{-1}\circ \psi'} & I\pi \oplus \Z\pi^n}\]
is uniquely determined by an $(n+1)$-tuple $(x,z_1,\dots,z_{n})\in (\Z\pi)^{n+1}$.
\smallskip

\begin{claim}
We can write $x=\pm 1+y$ for some $y\in I\pi$.
\end{claim}
\smallskip

Since $f$ is the inclusion of a direct summand, we can consider a map $g\colon I\pi\oplus(\Z\pi)^n \to I\pi$ with $g\circ f=\id_{I\pi}$.  This is also determined by an $(n+1)$-tuple $(w_1,\dots,w_{n+1}) \in (\Z\pi)^{n+1}$.  Here $w_i \in I\pi \subset \Z\pi$ for $i \geq 2$.  The composite $g \circ f$ is therefore determined by $U = w_1x+\sum_{i=1}^{n} w_{i+1}z_i$.  Since $g \circ f =\Id$ we have that $U=1$.  For $i \geq 2$, we have $\epsilon(w_i)=0$, where $\epsilon\colon \Z\pi\to\Z$ is the augmentation. So
\[1 = \epsilon\Big(w_1x+\sum_{i=1}^{n+1} w_{i+1}z_i\Big) = \epsilon(w_1)\epsilon(x).\]
 Thus $\epsilon(x)=\pm 1$ and $x\pm 1\in I\pi$ as claimed.

From now on we use the identification $\psi\colon I\pi \oplus (\Z\pi)^n \toiso \pi_2(M \#k (S^2 \times S^2)$.   Therefore, for any $\beta\in I\pi$, we write $\psi'(\beta,0,\ldots,0)\in I\pi\oplus(\Z\pi)^n$ as $\beta(\pm 1+y,z_1,\ldots,z_n)$ as above. Let \[\Lambda:=\lambda((\pm 1+y,z_1,\ldots,z_n),(-y,-z_1,\ldots,-z_n)) \in \Z\pi,\]
where we formally extend $\lambda$ to $(\Z\pi)^{n+1}$ using \cref{lem:intersec}.
Now we stabilise the manifold $M \#k (S^2 \times S^2)$ twice more and consider the following sequence of equations:
\begin{align*}
&\tau(\psi'(\beta,0,\ldots,0))  \\
=&\tau(\beta(\pm 1+y,z_1,\ldots,z_n,0,0,0,0))\\
=&\tau(\beta(\pm 1+y,z_1,\ldots,z_n,1,0,0,0))\\
=&\tau(\beta(\pm 1+y,z_1,\ldots,z_n,1,0,0,0)+\beta(-y,-z_1,\ldots,-z_n,0,-\ol{\Lambda},1,\ol{\Lambda}))\\
=&\tau(\beta(\pm 1,0,\ldots,0,1,-\ol{\Lambda},1,\ol{\Lambda}))\\
=&\tau(\beta(\pm 1,0,\ldots,0,0,0,0,0)) = \tau((\beta,0,\ldots,0,0,0,0,0)).
\end{align*}
The last equation uses the fact that $\tau(x)=\tau(-x)$ whenever~$\tau(x)$ is defined. The second, third and fifth equation follow from \cref{lem:augtau}. The application of \cref{lem:augtau} for the third equation requires some justification, since the hypotheses of that lemma require that various intersection and self-intersection numbers vanish.  We will work with the intersection form $\lambda$, formally extended to $(\Z\pi)^{n+5}$, similarly to above. We also extend the domain of $\psi'$ to $(\Z\pi)^{n+5}$.  The quantity $\Lambda$ is defined in such a way that the intersection between $(\pm 1+y,z_1,\ldots,z_n,1,0,0,0)$ and $(-y,-z_1,\ldots,-z_n,0,-\ol{\Lambda},1,\ol{\Lambda})$ is trivial. Using the key property of $\psi'$ that the intersection pairing vanishes on $\psi'(I\pi)$, and denoting $\lambda(x,x)=\lambda(x)$, we have
\[\lambda((\pm 1+y,z_1,\ldots,z_n,1,0,0,0))=\lambda\big(\psi'(1,0,\ldots,0,0,0,0,0)\big)=0.\]
We also use that the last $1$ in the first tuple represents an embedded sphere in the first extra copy of $S^2 \times S^2$, so does not change the intersection number.

We also have that $\lambda$ vanishes on the sum \begin{align*}
& \lambda\big((\pm 1+y,z_1,\ldots,z_n,1,0,0,0) + (-y,-z_1,\ldots,-z_n,0,-\ol{\Lambda},1,\ol{\Lambda})\big)  \\ = & \lambda\big((\pm 1, 0,\ldots,0,1,-\ol{\Lambda},1,\ol{\Lambda})\big) = 0.
\end{align*}
Therefore from the formula $$\lambda(a+b,a+b)=\lambda(a,a) + \lambda(b,b) + \lambda(a,b)+\ol{\lambda(a,b)},$$ we see that
 $\lambda\big((-y,-z_1,\ldots,-z_n,0,-\ol{\Lambda},1,\ol{\Lambda})\big)=0$.
 As observed in the proof of \cref{lem:tauquotient}, $\lambda(a,a)=0$ implies that $\mu(a)=0$ for any $a \in \pi_2(M\#(k+2) (S^2 \times S^2))$.  Also recall that $\mu(\beta a)=\beta\mu(a)\ol{\beta}$.
 This completes the justification of the application of \cref{lem:augtau} in the third equation above. The sequence of equalities above shows that $\tau_\psi$ is independent of $\psi$.

Thus, $\tau_\psi$ is invariant under stable diffeomorphism, since
\[\psi\colon I\pi\oplus \Z\pi^n\to \pi_2(M\#k(S^2\times S^2))\] can be extended to an isomorphism $I\pi\oplus \Z\pi^{n+2}\to \pi_2(M\#(k+1)(S^2\times S^2))$, and this does not change the computation of $\tau$, as we only compute on an $I\pi$ direct summand.
\end{proof}

\smallskip

\begin{definition}\label{def:tau-M}
Define $\tau_M:= \tau_{\psi} \colon H^2(B\pi;\Z/2) \to \Z/2$ for some choice of map $\psi$.  This is a well-defined stable diffeomorphism invariant by \cref{lem:tau-indep}.
\end{definition}
\smallskip

\begin{lemma}\label{lem:tauinv} Under \cref{conv:Sec7} the following holds.
\smallskip \begin{enumerate}
\item The map $\tau_{M} \colon H^2(B\pi;\Z/2) \to \Z/2$ of \cref{def:tau-M} is a homomorphism.
\item  Under the identification $\Hom_{\Z/2}(H^2(B\pi;\Z/2),\Z/2) \cong H_2(B\pi;\Z/2)$, the image of $\tau_M$ agrees with the image of $[M \xrightarrow{c} X] \in \Omega_4^{Spin}(B\pi)$ in $H_2(B\pi;\Z/2)$.
\end{enumerate}
\end{lemma}
\smallskip

\begin{proof}
We will prove both parts of the lemma by computing in the model 4-manifolds that we constructed in \cref{subsec:ex1,subsec:ex2}.

In \cref{subsec:ex1} we constructed a model $M_0$ for the null bordant element of $\Omega_4^{Spin}(X)=\Omega_4^{Spin}(B\pi)$ with $\pi_2(M_0)\cong I\pi\oplus \Z\pi$ such that there exists an embedded sphere representing each element of $I\pi \subset \pi_2(M_0)$. It follows that $\tau_{M_0}\equiv 0$, which in particular is a homomorphism that agrees with the image of $0=[M_0 \to B\pi] \in \Omega^{Spin}_4(X)$ in  $H_2(X;\Z/2)$.

In \cref{subsec:ex2} we constructed models $P$ for all stable diffeomorphism classes with signature zero and even intersection form described in \cref{lemma:int-form-on-P}. The intersection form vanishes on the image of the map
\[I\pi\to \pi_2(P)\cong \Z\pi\oplus I\pi\oplus \Z\pi^n\oplus \Z\pi^n\]
given by $\beta\mapsto (-\beta,\beta,0,0)$ (to see this, compute using the matrix in the proof of \cref{thm:parity-detects-sec}).

Recall that the manifold $M'$ was obtained in \cref{subsec:ex2} from $M_1\#M_1$ by removing $D^1 \times X^1 = D^1 \times X^{(1)}\sm X^{(0)}$ in each copy. The cores of these solid tori removed from $M_1\#M_1$ represent elements $g_i^{-1}\cdot g_i'$ of $\pi_1(M_1\#M_1)\cong \pi \ast \pi$, where $g_i'$ is the same generator as $g_i$ in the second copy of $M_1$.
The closed manifold $P$ was then obtained by glueing in copies of $S^2\times D^2$ to $M'$.

There exists an immersed sphere representing $(-\beta,\beta,0,0) \in \pi_2(P)$ that lives in $M' \subset P$.  However $\mu$ only vanishes on all such spheres after passing to $P$.  Thus the Whitney discs witnessing that $\mu$ vanishes make use of the surgery discs, and the framing on the surgeries determines whether the Whitney discs are framed.  The details follow.

Let $g_1,\ldots, g_n$ be the generators of $\pi$ on which the surgery on $M \# M$ was done. For $\beta =1-g_i$, the element $(-(1-g_i),1-g_i,0,0) \in \pi_2(P)$ is represented by the sphere $$\Sigma_{i+1}^1 \# \Sigma_{i+1}^2 \#(g_i-1)\Sigma_1^1 \subset M' \subset P,$$ where the summand spheres were defined in \cref{sec:exmanifolds}.  The self-intersection number of this sphere is $g_i -g'_i$.  Therefore, all but two self-intersection points of this representative of $(-(1-g_i),1-g_i,0,0) \in \pi_2(M')$ can be paired up by Whitney discs in $M'$.

The homotopy classes $g_i, g_i' \in \pi_1(M')$ have the same image in $\pi_1(P)$, so that the self-intersection number vanishes.  The Whitney disc that pairs up the corresponding self-intersections passes over the $i$th surgery disc $D^2\times\{\pt\}\subset D^2 \times S^2$ precisely once.

The framing of this Whitney disc changes when the twisted surgery is used instead of the untwisted one.
Since we already know that $\tau_P=\tau_{M_0}=0$ if $P$ is null bordant, and the twists for the surgery precisely depend on the image of $[P \xrightarrow{c} X]$ in $\Hom(H^2(X;\Z/2),\Z/2)$, this shows that $\tau_P$ changes in the same way, restricted to elements of the form $(g_i-1,1-g_i,0,0)$.

For elements of the form $(g-1,1-g,0,0)$, for general $g=g^{\epsilon_{i_1}}_{i_1}\cdots g^{\epsilon_{i_k}}_{i_k}\in\pi, \epsilon_{i_j}\in \{\pm 1\}$, we can argue in the same way.
Represent $(g-1,1-g,0,0)$ by
$$(g-1)\Sigma_1^1 \# \left(\#_{j=1}^2 \#_{i=1}^n \frac{\partial g}{\partial g_i} \Sigma_{i+1}^j\right)$$
and observe that all but one pair of self-intersections can be paired up by Whitney discs in $M'$.  The self-intersection number in $M'$ is $g-g'$, which becomes zero after passing to $P$.
The last Whitney disc can now be chosen to use each of the $i_j$th surgery discs exactly as they appear in the word for $g$.

Since the elements $1-g\in I\pi\cong H^2(X^{(2)};\Z\pi)$ map surjectively onto $H^2(X;\Z/2)$, this shows that $\tau_P$ agrees with the image of $[P \xrightarrow{c} X]$ in $\Hom(H^2(X;\Z/2),\Z/2)$, and in particular is a homomorphism.
\end{proof}

We are now in a position to prove \cref{thm:B}.

\begin{proof}[Proof of \cref{thm:B}]
By \cref{thm:bordism-group-spin-case}, the bordism group $\Omega_4(\xi)$ is isomorphic to $\Z\oplus H_2(X;\Z/2)\oplus \Z/2$, and the first summand is given by the signature. By \cref{thm:parity-detects-sec}, the $\Z/2$ summand is given by the parity of the equivariant intersection form. By \cref{thm:action-spin-case}, in the case where the invariant in the $\Z/2$ summand is $1$, $\Out(\xi)$ acts transitively on the second summand. Thus, two $4$-manifolds with odd equivariant intersection form are stably diffeomorphic if and only if they have the same signature.  In the case where the invariant in the $\Z/2$ summand is trivial, $\Out(\xi)$ acts by $\Out(\pi)$ on the second summand and in light of \cref{lem:tauinv}, the invariant there is given by $[\tau_M]\in H_2(X;\Z/2)/\Out(\pi)$.
\end{proof}

\section{Detecting the classification from equivariant intersection forms}\label{section-tau-versus-int-form}

In this section we prove \cref{thm:diffeo-iff-int-forms-intro} which says that the stable homeomorphism classification is determined by the stable isomorphism class of the equivariant intersection form.  We have already proven \cref{thm:diffeo-iff-int-forms-intro} in \cref{thm:3.3} for totally non-spin manifolds, so we only need to address the spin and almost spin cases, which is the content of \cref{thm:diffeo-iff-int-forms} below.

Let $\pi$ be a COAT group and let $H$ be the standard hyperbolic form on $(\Z\pi)^2$.
Let $w \in H^2(B\pi;\Z/2)$ and let $B_w$ be the resulting normal $1$-type of (almost) spin topological manifolds from \cref{prop:top-1-types}, with characteristic element $w$ as in \cref{lem:w}.  Recall that this can be defined via the pullback diagram
\[\xymatrix{B_w \ar[r] \ar[d] & B\pi \ar[d]^w \\ BSTOP \ar[r]_-{w_2} & K(\Z/2,2), }\]
as in \cite[Theorem~2.2.1]{teichnerthesis}.  Here $w=0$ corresponds to the spin case.

Recall that the hermitian augmented normal 1-type of a $4$-manifold $M$ is the quadruple
$$HAN_1(M) =(\pi_1(M),w_M,\pi_2(M), \lambda_M),$$ where $\pi_2(M)$ is considered as a module over the group ring $\Z[\pi_1(M)]=\Z\pi$ and $w_M \in H^2(B\pi;\Z/2) \cup \{\infty\}$ corresponds to the normal $1$-type.
Fix $\pi$.  We say that two $HAN_1$ types $(\pi,w,\pi_2,\lambda)$ and $(\pi,w',\pi_2',\lambda')$ are \emph{stably isomorphic} if there is an automorphism $\theta \in \Out(\pi)$ with $\theta^*(w') =w$, integers $k,k'$, and an isomorphism $\Upsilon \colon \pi_2 \oplus k H \toiso \pi_2' \oplus k' H$ of $\Z\pi$-modules, over $\theta$, that respects $\lambda$ and $\lambda'$.  That is, $\Upsilon(gq) = \theta(g)\Upsilon(q)$ and $\lambda'(\Upsilon(p),\Upsilon(q)) = \theta(\lambda(p,q))$ for all $g \in \pi$ and for all $p,q \in \pi_2 \oplus kH$.

\smallskip

\begin{theorem}\label{thm:diffeo-iff-int-forms}
Two closed 4-manifolds with COAT fundamental group and  universal covering spin are stably homeomorphic if and only if their $HAN_1$-types
\[
HAN_1 := (\pi_1,w, \pi_2, \lambda)
\]
are stably isomorphic.

\end{theorem}
\smallskip

In the totally non-spin case, the $HAN_1$-types are determined simply by the fundamental group and signature. One also needs the Kirby-Siebenmann invariants to coincide to deduce that two such manifolds are stably homeomorphic.
On the other hand, note that for manifolds with universal covering spin, the Kirby-Siebenmann invariants are determined by the (algebraic) $HAN_1$-types.
Since two smooth 4-manifolds are stably diffeomorphic if and only if they are stably homeomorphic, we obtain the corresponding result in the smooth category. It is easier to state due to the vanishing of the Kirby-Siebenmann invariant.

\smallskip

\begin{cor}
Two closed smooth 4-manifolds with COAT fundamental group are stably diffeomorphic if and only if their $HAN_1$-types $(\pi_1,w, \pi_2, \lambda)$ are stably isomorphic.
\end{cor}
\smallskip

The only if direction of \cref{thm:diffeo-iff-int-forms} is straightforward.  For the other direction, observe that the summand $H_3(B\pi;\Z/2)\subseteq \Omega_4(B_0)$ is detected by the parity of the equivariant intersection form by \cref{thm:parity-detects-sec}. For $w\neq 0$, we have $\Omega_4(B_w)=F_{2,2}$ by \cref{thm:almostspinbord}. Therefore, it only remains to show that elements in $F_{2,2} \subset \Omega_4(B_w)$ are detected by their equivariant intersection form.
We begin with the following important lemma.
\smallskip

\begin{lemma}\label{lem:null-bordant}
Let $N$ be a $4$-manifold with fundamental group $\pi$, representing the trivial element of $\Omega_4(B_w)$. Then $\pi_2(N)$ is stably isomorphic to $I\pi\oplus\Z\pi$ and the canonical extension of $\lambda_N$ to $(\Z\pi)^2$ is hyperbolic.
\end{lemma}

\begin{proof}
Any two null bordant manifolds with the same normal 1-type are stably homeomorphic, thus it suffices to prove the lemma for one choice of null bordant element~$N,$ having the correct fundamental group, for each normal 1-type.

In the case $w=0$, that is in the spin case, choose $N$ to be $M_0$ as constructed in \cref{subsec:ex1}. It was calculated that $\pi_2(M_0) \cong I\pi\oplus \Z\pi$ and that the intersection form becomes hyperbolic when extended to $(\Z\pi)^2$.

To show the lemma in the almost spin case we construct $N$ as follows. Let $X$ be a 3-manifold model for $B\pi$, and choose an element of $H^2(X;\Z)$ whose reduction modulo 2 is equal to $w \in H^2(B\pi;\Z/2)$.  Let $E \to X$ be the complex line bundle over $X$ whose first Chern class is the given element of $H^2(X;\Z)$.  The sphere bundle of the associated 2-dimensional real vector bundle is a circle bundle over $X$, which is a $4$-manifold $N'$ whose stable tangent bundle fits into a pullback diagram of stable bundles
\[\xymatrix{\tau_{N'} \ar[r] \ar[d] & E \ar[d] \\ N' \ar[r] & X}\]
Using this bundle data, perform surgery on a fibre $S^1 \subset N'$ to obtain a new manifold $N \xrightarrow{c} X$.  The stable tangent bundle of $N$ is given by $c^*(E)$ and $c$ induces an isomorphism on fundamental groups. In particular, it follows from \cref{lem:w}, translated to the topological category, that $N$ has $B_w$ as normal $1$-type, because $w_2(N) = c^*(w_2(E)) = c^*(w)$.

The resulting 4-manifold $N$ is null bordant because the trace of the surgery is a bordism over the normal 1-type of $N$ and the disc bundle is a null bordism of the sphere bundle, also over the normal 1-type of $N$.  The computation of the intersection form of $N$ is similar to the computation of the intersection form of the null bordant element in the spin case. In the proof of \cref{lem:pi2computation}, replace $X\times S^1$ by $N'$. The $\pi$-covering $\overline{N'}$ defined by the pullback
\[\xymatrix{\overline{N'}\ar[r]\ar[d]&\wt X\ar[d]\\N'\ar[r]&X}\]
is homeomorphic to $\wt X\times S^1$ since $\wt{X}$ is contractible. Performing a surgery on an $S^1$ fibre corresponds to $\pi$-equivariant surgery on $\overline{N'}$. The computation of the second homotopy group and the intersection form of $M_0$ in the proof of \cref{lem:pi2computation} was entirely in terms of the $\pi$-cover. Thus the same computation yields $H_2(N;\Z\pi)\cong I\pi\oplus \Z\pi$ and
\[\lambda_N=
\begin{blockarray}{ccc}
& I\pi & \bbZ\pi\\
\begin{block}{c(cc)}
I\pi & 0 & 1\\
\bbZ\pi & 1 & 0\\\end{block}\end{blockarray}~.\]
The extended equivariant intersection form is therefore hyperbolic as claimed.
\end{proof}

Let $N$ be a null bordant (almost) spin $4$-manifold with fundamental group $\pi$ with normal 1-type $B_w$.  For definiteness, take $N$ to be the manifold constructed in \cref{lem:null-bordant}. Next consider the following diagram.

\begin{equation}\label{diagram:L-group}
\xymatrix{\bbL\langle 1 \rangle_4(B\pi) \ar[r]^-{\cong} & L_4(\Z\pi)  & \\
        \bbL\langle 1 \rangle_4(N) \ar@{-->}[r]_-{\Theta} \ar[u]^{c_*} & F_{2,2} \ar@{^{(}->}[r] \ar@{-->}[u]_-{\widehat{\Lambda}} &  \Omega_4(B_w)
}\end{equation}

We will proceed by first defining the sets in the diagram, then the maps in the diagram, before showing that the diagram commutes.  We only define the dashed arrows as maps of sets.  \cref{thm:diffeo-iff-int-forms} will follow from the commutativity of the diagram.

Here $\bbL = \bbL(\Z)$ is the quadratic $L$-theory spectrum of the integers~\cite[$\mathsection$~13]{Ranicki-blue-book}, whose homotopy groups coincide with the $L$-theory of the integers; that is \mbox{$\pi_n(\bbL(\Z)) \cong L_n(\Z)$}.
The notation $\bbL\langle 1\rangle$ refers to the $1$-connected quadratic $\bbL$-spectrum, obtained from $\bbL$ by killing the non-positive homotopy groups.

The group $L_4(\Z\pi)$ is defined to be the Witt group of nonsingular quadratic forms (on finitely generated free $\Z\pi$-modules), considered up to stable isometry~\cite[Chapter~5]{Wall}.

The classifying map $c\colon N \to B\pi$ induces a map $c_*$ on $\bbL\langle 1\rangle$-homology.

The top horizontal arrow arises from the assembly map in quadratic $L$-theory. Define this map to be the composite
\[\xymatrix{\bbL\langle 1 \rangle_4(B\pi) \ar[r]^{\cong} & \bbL_4(B\pi) \ar[r]_-{\mathcal{A}}^-{\cong} & L_4(\Z\pi), }\]
where the first map is induced by $\bbL\langle 1 \rangle \to \bbL$ and the second map is the assembly map~\cite{Ranicki-blue-book}, which has been proven to be an isomorphism for COAT groups by A.~Bartels, T.~Farrell and W.~L\"uck \cite[Corollary 1.3]{Bartels-Farrell-Luck}.  Furthermore, since COAT groups are $3$-dimensional, it follows that the first map is also an isomorphism.

If $Y$ is a closed oriented manifold, it satisfies Poincar\'e duality in $L$-theory; see for example\ A.~Ranicki \cite[B9~p.~324]{Ranicki-blue-book}. This is due to the Sullivan-Ranicki orientation $\MSTOP \to \bbL^{sym}$ which gives a fundamental class for $Y$ in the \emph{symmetric} theory~$\bbL^{sym}$. It follows that $Y$ has Poincar\'e duality in any module spectrum over $\bbL^{sym}$, such as $\bbL\langle 1 \rangle$. If $Y$ is $4$-dimensional this implies that
\[ \bbL\langle 1 \rangle_4(Y) \cong \bbL\langle 1 \rangle^0(Y)\]
Now $\bbL\langle 1 \rangle^0(Y) \cong [Y,\Omega^\infty \bbL\langle 1 \rangle]$.  But the infinite loop space $\Omega^\infty \bbL \langle 1 \rangle$ of the 1-connective $\bbL$-spectrum is $G/TOP$, by the Poincar\'e conjecture combined with the surgery exact sequence in the topological category.  Therefore we have that \[\bbL\langle 1 \rangle^0(Y) \cong [Y,G/TOP]. \]
In particular, elements of $\bbL\langle 1 \rangle_4(Y)$ can be identified with normal bordism classes of degree one normal maps $X \to Y$ for $X$ a closed topological manifold; see for example~\cite[Theorem~3.45]{Luck-basic-intro}.

After identifying $\bbL\langle 1 \rangle_4(N)$ with degree one normal maps, the up-then-right composition of diagram (\ref{diagram:L-group}) coincides with taking the surgery obstruction of a degree one normal map $f \colon M \to N$, again according to~\cite[B9,~p.~324]{Ranicki-blue-book}.  The operation of taking the surgery obstruction is defined as follows. Perform surgery below the middle dimension to make the normal map $1$-connected, then consider the intersection and self-intersection form on the surgery kernel $\ker(f_*\colon H_2(M;\Z\pi) \to H_2(N;\Z\pi))$.  This yields a nonsingular quadratic form $\kappa(f)$ on a finitely generated free $\Z\pi$-module~\cite[Lemma~2.2]{Wall}.
The equivariant intersection form of $M$ decomposes as
\[ \lambda_M \cong \kappa(f) \oplus \lambda_N \]
because the Umkehr map $f^{!}$ provides a splitting of the map $f_* \colon H_2(M;\Z\pi) \to H_2(N;\Z\pi)$, and the intersection form of $M$ respects the splitting; for example, see \cite[Proposition 10.21]{Ranicki-AGS-book}.

The identification of the surgery obstruction and assembly also involves the identification of the Wall $L$-groups with the Ranicki $L$-group of quadratic Poincar\'{e} chain complexes~\cite{Ranicki-ATS-I,Ranicki-ATS-II} via the process of algebraic surgery below the middle dimension.

For the definition of the map $\Theta\colon \bbL\langle 1 \rangle_4(N) \to F_{2,2}$ we need the following observation.

Note that the map $BSTOP \xrightarrow{w_2} K(\Z/2,2)$ factors through the canonical map $BSTOP \to BSG$, where $BSG$ denotes the classifying space for oriented stable spherical fibrations. Define $BSG_w$ by the following pullback diagram:
\[\xymatrix{BSG_w \ar[r] \ar[d] & B\pi \ar[d]^w \\ BSG \ar[r]_-{w_2} & K(\Z/2,2).}\]
Since the map $B\pi \to K(\Z/2,2)$ is $2$-coconnected, so is the map $BSG_w \to BSG$.

We say that an $n$-dimensional Poincar\'e complex $Y$ has \emph{Spivak normal $1$-type} $B$ if there is a $2$-coconnected fibration $B \to BSG$ such that the Spivak normal fibration $SF(Y)\colon Y \to BSG$ lifts to a 2-connected map $\wt{SF(Y)} \colon Y \to BSG_w$, called a \emph{Spivak normal $1$-smoothing}, such that
\[\xymatrix @C+0.3cm{ & BSG_w \ar[d] \\ Y \ar[r]^-{SF(Y)} \ar@/^1pc/[ur]^-{\wt{SF(Y)}} & BSG}\]
commutes.
\smallskip

\begin{lemma}\label{types-and-normal-invariants}
Let $Y \to BSG_w$ be a $n$-dimensional Poincar\'e complex, $n\geq 4$, with a normal $1$-smoothing to $BSG_w$, and let $f \colon M \to Y$ be a $2$-connected degree one normal map from a closed topological manifold $M$ to $Y$. Then there is an induced normal $1$-smoothing $M \to B_w$.
\end{lemma}
\smallskip

\begin{proof}
The datum of a degree one normal map consists of a pullback diagram
\[\xymatrix{\nu_X \ar[r]^{\widehat{f}} \ar[d] & \xi \ar[d] \\ X \ar[r]_f & Y}\]
where $\xi$ is some vector bundle lift of the Spivak fibration $SF(Y)$ of $Y$.
Let $\wt{SF(Y)}$ be the Spivak $1$-smoothing. Then the following diagram commutes:
\[\xymatrix@R+0.3cm @C+0.3cm {Y \ar[r]^-{\wt{SF(Y)}} \ar[d]_\xi \ar[dr]^-{SF(Y)} & BSG_w \ar[d] \\ BSTOP \ar[r] & BSG.}\]
Furthermore, we can consider the diagram
\[\xymatrix{B_w \ar[r] \ar[d] & BSG_w \ar[r] \ar[d] & B\pi \ar[d]^w \\ BSTOP \ar[r] & BSG \ar[r]_-{w_2} & K(\Z/2,2)}\]
in which the outer rectangle and the right square are pullbacks by definition. Thus by the pullback lemma, the left square is also a pullback.
By the universal property of this pullback there is a unique map $\wt{\xi} \colon Y \to B_w$ that gives $\xi$ an induced $B_w$-structure.

Since $\widehat{f}^*(\xi) \cong \nu_X$, now we have the following commutative diagram:
\[\xymatrix@R+0.1cm @C+0.1cm {Y \ar[d]_{\wt{\xi}} \ar[dr]_{\xi} & X \ar[l]_-{f} \ar[d]^{\nu_X} \\ B_w \ar[r] & BSTOP.}\]
We claim that the composition $\wt{\xi}\circ f$ is a normal $1$-smoothing.
As $BSTOP \to BSG$ induces an isomorphism on $\pi_1$ and $\pi_2$, by considering the homotopy fibres in the left hand square of the above rectangular diagram, we see that the $B_w \to BSG_w$ also induces an isomorphism on $\pi_1$ and $\pi_2$; here we use that the map $B_w \to BSTOP$ is $2$-coconnected. The claim that $X \to B_w$ is a normal $1$-smoothing now follows from the fact that $Y \to BSG_w$ is a Spivak normal $1$-smoothing, the fact that $f$ is 2-connected, and applying the functors $\pi_1$ and $\pi_2$ to the following diagram:
\[\xymatrix{&  B_w \ar[r] & BSG_w \\
X \ar[ur]^{\wt{\nu}_X} \ar[r]^f & Y \ar[u] \ar[ur]_{\wt{SF(Y)}}}\]
\end{proof}

Now we are in the position of being able to define the map $\Theta\colon \bbL\langle 1 \rangle_4(N) \to F_{2,2}$. This is very similar to constructions by J.~Davis~\cite[Theorem 3.10 and 3.12]{Davis-05}.

Represent an element of $\bbL\langle 1 \rangle_4(N)$ by a degree one normal map $f \colon M \to N$.  We can assume that $f$ induces an isomorphism on fundamental groups by performing surgeries on $M$ within the normal bordism class. By \cref{types-and-normal-invariants}, the normal $1$-smoothing of $N$ induces a normal $1$-smoothing of $M$.  Moreover, we can apply \cref{types-and-normal-invariants} to a $2$-connected normal bordism of normal maps, to obtain a $B_w$ bordism of the resulting normal $1$-smoothings. Thus we obtain a well defined element of $\Omega_4(B_w)$.

For $w \neq 0$ we have shown that $F_{2,2} = \Omega_4(B_w)$ in \cref{thm:almostspinbord}. For $w = 0$, we have that $\Omega_4(B_w) = \Omega_4^{TOPSpin}(B\pi)$, and $F_{2,2}$ is given by elements whose reference maps to $B\pi$ stably factor through the $2$-skeleton $B\pi^{(2)}$ of $B\pi$. In particular, since $0 = [N] \in \Omega_4^{TOPSpin}(B\pi)$, there exists a representative of the null-bordant class such that the classifying map to $B\pi$ factors through the $2$-skeleton, and any two null bordant manifolds are stably homeomorphic, it follows that up to stable homeomorphism and up to homotopy, the map $c\colon N \to B\pi$ factors through the $2$-skeleton of $B\pi$.
Thus the composite $c \circ f  \colon M \to N \to B\pi$
also stably factors through the $2$-skeleton of $B\pi$, whence also $M$ lies in $F_{2,2}$.

Next we will define a map $\widehat{\Lambda} \colon \Im(\Theta) \to L_4(\Z\pi)$.  In the proof of \cref{thm:diffeo-iff-int-forms} below we will see that $\Im(\Theta)=F_{2,2}$, so that in fact we define the map $\widehat{\Lambda}$ claimed in Diagram \eqref{diagram:L-group}.  An element of $\Im(\Theta)$ can be represented by a $4$-manifold $M$ which has a degree one normal map $f \colon M \to N$ that induces an isomorphism on fundamental groups.  We saw above that $\lambda_M \cong \kappa(f) \oplus \lambda_N$.    By \cref{lem:null-bordant} the intersection form $\lambda_N$ on $I\pi \oplus \Z\pi$ extends to a hyperbolic form $\widehat{\lambda}_N$ on $\Z\pi^2$. Therefore, $\lambda_M$ can be extended to a nonsingular quadratic form \[\widehat{\lambda}_M = \kappa(f)+\widehat{\lambda}_N\] defined on a free $\Z\pi$-module and we define $\widehat{\Lambda}([M])=[\widehat{\lambda}_M]\in L_4(\Z\pi)$. In \cref{lem:lambda-well-defined} below we show that this is independent of the choice of $M$. Since $\widehat{\lambda}_N$ is hyperbolic, $\widehat{\Lambda}(M)=[\kappa(f)]\in L_4(\Z\pi)$ and it follows that Diagram \eqref{diagram:L-group} is commutative.
\smallskip

\begin{lemma}\label{lem:lambda-well-defined}
The definition above determines a well-defined map $\Im(\Theta) \to L_4(\Z\pi)$.
\end{lemma}
\smallskip

\begin{proof}
We need to see that  $[\kappa(f) \oplus \widehat{\lambda}_N] = [\kappa(f')\oplus\widehat{\lambda}_N]$ if $\Theta[M \xrightarrow{f} N] = \Theta[M' \xrightarrow{f'} N].$ But being the same element in $F_{2,2}$ implies that $M$ and $M'$ are stably homeomorphic. In particular we see that $\lambda_M$ and $\lambda_{M'}$ are stably isomorphic. Thus we obtain
\[ \kappa(f) \oplus \widehat{\lambda}_N \cong \widehat{\lambda}_M \cong \widehat{\lambda}_{M'} \cong \kappa(f') \oplus \widehat{\lambda}_N. \]
It follows that $[\kappa(f)] = [\kappa(f')] \in L_4(\Z\pi)$.
\end{proof}

\begin{proof}[Proof of \cref{thm:diffeo-iff-int-forms}]
We will prove that the map $\Theta\colon \bbL\langle 1 \rangle_4(N) \to F_{2,2}$ is surjective and the map $\widehat{\Lambda}\colon F_{2,2}\to L_4(\Z\pi)$ is injective.
First we note that in the diagram
\[\xymatrix{\bbL\langle 1 \rangle_4(B\pi) \ar[r]^-{\cong} & L_4(\Z\pi)  \\
        \bbL\langle 1 \rangle_4(N) \ar@{-->}[r]_-{\Theta} \ar[u]^{c_*} & \Im(\Theta), \ar@{-->}[u]_-{\widehat{\Lambda}} }\]
by \cref{top-bordism-groups}, all groups contain a $8\Z \cong L_4(\Z)$ direct summand, which is detected by the signature.
The map $c_*$ respects this decomposition. Also the map $\Theta$ commutes with the projections onto the $8\Z$-summands, because in $\bbL\langle 1 \rangle$-homology the projection takes a normal invariant $[f \colon M \to N]$ to $\sigma(M) - \sigma(N) = \sigma(M)$. Here $\sigma(N)=0$ because $N$ is null bordant.
Now we can perform connected sums of $M$ with $E_8$-manifolds to see surjectivity for the $8\Z$-summands. Therefore we may consider the reduced version of the diagram:
\[\xymatrix{\wt{\bbL\langle 1 \rangle}_4(B\pi) \ar[r]^-{\cong} & \wt{L}_4(\Z\pi)  & \\
        \wt{\bbL\langle 1 \rangle}_4(N) \ar@{-->}[r]_-{\wt{\Theta}} \ar[u]^{c_*} & \Im(\wt{\Theta}) \ar@{-->}[u]_-{\widehat{\Lambda}}}\]
where now we view $\Im(\wt\Theta) \subset H_2(B\pi;\Z/2) \cong \wt{F}_{2,2}$.

It follows from the Atiyah-Hirzebruch spectral sequence that the map
\[
c_*\colon \wt{\bbL\langle 1 \rangle}_4(N) \to \wt{\bbL\langle 1 \rangle}_4(B\pi)
\]
is surjective if the map $c_*\colon H_2(N;\Z/2) \to H_2(B\pi;\Z/2)$
is surjective. This in turn follows from the Serre spectral sequence associated to the fibration
$ \wt{N} \to N \xrightarrow{c} B\pi, $
which, as in the proof of \cref{lem:w}, gives rise to an exact sequence
\[\xymatrix{H_0(\pi;H_2(\wt{N};\Z/2)) \ar[r] & H_2(N;\Z/2) \ar[r] & H_2(B\pi;\Z/2) \ar[r] & 0.}\]

In particular, it follows that the up-then-right composite in the diagram is surjective.
This implies that
$ \widehat{\Lambda}\colon \Im(\wt{\Theta}) \to \wt{L}_4(\Z\pi) $
is surjective.
Since $\Im(\wt{\Theta}) \subset H_2(B\pi;\Z/2)$ and $\wt{L}_4(\Z\pi) \cong H_2(B\pi;\Z/2)$, a counting argument shows that $\widehat{\Lambda}$ can only be surjective if $\Im(\wt{\Theta}) = \wt{F}_{2,2} \cong H_2(B\pi;\Z/2)$ and $\widehat{\Lambda}$ is bijective.
This, in particular the fact that $\widehat{\Lambda}$ is injective, completes the proof of \cref{thm:diffeo-iff-int-forms}.
\end{proof}

Finally, we describe exactly which stable isomorphism classes of intersection forms are realised by stable diffeomorphism classes of spin $4$-manifolds with COAT fundamental group.

Recall that for $\sigma=0,1$ we constructed, in \cref{subsec:ex1}, $4$-manifolds $M_0,M_1$ with fundamental group $\pi_1(M_{\sigma}) \cong \pi$, $\pi_2(M_\sigma)\cong I\pi \oplus \Z\pi$, and intersection form \[\lambda_{M_\sigma}=\begin{blockarray}{ccc}
&I\pi&\bbZ\pi\\
\begin{block}{c(cc)}
I\pi&\sigma&1\\
\bbZ\pi&1&0\\\end{block}\end{blockarray}~.\]
\smallskip

\begin{theorem}\label{theorem:realisation-of-forms}
Let $\pi$ be a COAT group.  The following constitutes a complete list of nonsingular hermitian forms on $I\pi \oplus (\Z\pi)^n$ that occur as the stable isomorphism class of the intersection form of some topological  $4$-manifold with fundamental group~$\pi$ and normal $1$-type~$w$.
\smallskip \begin{enumerate}
\item For $w=\infty$, \[\lambda_{M_0}\oplus\begin{pmatrix}\Id_m&0\\0&-\Id_n\end{pmatrix},\]
with identity matrices $\Id_n,\Id_m$ of size $m,n \geq 1$ and  signature $=m-n$.
\item\label{item:rof2} For $w\neq\infty$, $\lambda_{M_0}\oplus \lambda$,
where $\lambda$ is any form in $L_4(\Z\pi) \cong 8\cdot\Z \oplus H_2(\pi;\Z/2)$.
\item For $w=0$, in addition to part \ref{item:rof2}, $\lambda_{M_1}\oplus n\cdot E_8$,
where $n\in\Z$ is determined by the signature $8\cdot n$.
\end{enumerate}
\end{theorem}
\smallskip

Note that by \cref{thm:diffeo-iff-int-forms}, within each normal $1$-type, the equivariant intersection form determines the stable homeomorphism type of a manifold. By the above result, each fixed stable form $\lambda_{M_0}\oplus\lambda$ with $\lambda\in L_4(\Z\pi)$ is realised by multiple stable diffeomorphism classes.

More precisely, each such form appears $2^d$ times, $d=\dim H^2(B\pi;\Z/2)$, namely exactly once for each normal $1$-type $w\neq \infty$. Note that within our class of COAT groups, this number~$d$ can be arbitrarily large.

\bibliographystyle{alpha}
\bibliography{classification}

\end{document}